\documentclass[a4paper]{amsart}

\usepackage[all]{xy}
\usepackage{amsmath}
\usepackage{amssymb}
\usepackage{amsthm}
\usepackage{latexsym}
\usepackage{enumerate}
\usepackage{prettyref}
\usepackage{hyperref}

\newtheorem{theorem}{Theorem}[section]
\newtheorem{definition}[theorem]{Definition}
\newtheorem{lemma}[theorem]{Lemma}
\newtheorem{proposition}[theorem]{Proposition}
\newtheorem{corollary}[theorem]{Corollary}
\newtheorem{remark}[theorem]{Remark}
\newtheorem{conjecture}[theorem]{Conjecture}
{
\newtheorem{examplecore}[theorem]{Example}}

\newenvironment{proofof}[1]{\vspace{.2cm}\noindent\textsc{Proof of
    #1:}}{\hspace*{\fill} $\blacksquare$\par\vspace{.1cm}}

\newcommand{\op}{\operatorname}

\begin{document}

\title[Quillen's general conjecture is false for trivial reasons]{On the cohomology of $\op{GL}_3$ of elliptic curves and Quillen's conjecture}  

\author{Matthias Wendt}
\thanks{Matthias Wendt was partially supported by EPSRC grant EP/M001113/1.}

\date{September 2016}

\address{Matthias Wendt, Fakult\"at Mathematik und Physik, Leibniz Universit\"at Hannover, Welfengarten 1, 30167, Hannover, Germany} 
\email{wendt@math.uni-hannover.de}

\subjclass[2010]{20G10 (11F75,20E42)}
\keywords{group homology, Quillen conjecture, buildings, vector bundles, elliptic curves}

\begin{abstract}
The paper provides a computation of the additive structure as well as a partial description of the Chern-class module structure of the cohomology of $\op{GL}_3$ over the function ring of an elliptic curve over a finite field. The computation is achieved by a detailed analysis of the isotropy spectral sequence for the action of $\op{GL}_3$ on the associated Bruhat--Tits building. This provides insights into the function field analogue of Quillen's conjecture on the structure of cohomology rings of arithmetic groups. The computations exhibit a lot of explicit classes which are torsion for the Chern-class ring. In some examples, even the torsion-free quotient of cohomology fails to be free. A possible variation of Quillen's conjecture is also discussed.
\end{abstract}

\maketitle
\setcounter{tocdepth}{1}
\tableofcontents

\section{Introduction}
\label{sec:intro}

The present paper draws some consequences from the computations in \cite{ell-dilog} of equivariant cell structures of Bruhat--Tits buildings for the groups $\op{GL}_3(k[E])$ where $E=\overline{E}\setminus\{\op{O}\}$ is an affine elliptic curve over the base field $k$. 
In the present paper, we now consider the case of a finite field $k=\mathbb{F}_q$ and investigate the consequences for cohomology. The computations of equivariant cell structures in \cite{ell-dilog} provide exactly the input relevant for the isotropy spectral sequence for the action of $\op{GL}_3(k[E])$ on the associated Bruhat--Tits building $\mathfrak{B}(E/k,3)$. It is possible to completely compute all the relevant differentials, resulting in a complete calculation of the additive structure of the cohomology of $\op{GL}_3(k[E])$ with $\mathbb{F}_\ell$-coefficients where $\ell\geq 5$ is a prime with $\ell\mid q-1$. Partial information on the multiplicative structure, in particular the action of the universal Chern classes, can also be obtained, which allows to discuss a function field analogue of a conjecture of Quillen \cite{quillen:spectrum} on the structure of cohomology rings of $S$-arithmetic groups. In the present situation, Quillen's conjecture would predict that the cohomology $\op{H}^\bullet(\op{GL}_3(k[E]);\mathbb{F}_\ell)$ is a free module over the ring $\op{H}^\bullet_{\op{top}}(\op{GL}_3(\mathbb{C});\mathbb{F}_\ell)\cong \mathbb{F}_\ell[\op{c}_1,\op{c}_2,\op{c}_3]$ of universal Chern-classes, but the computations in the paper show that this is false in far greater generality than known previously. On the other hand, having complete computations of the cohomology  $\op{H}^\bullet(\op{GL}_3(k[E]);\mathbb{F}_\ell)$ allows to discuss possible reformulations of the conjecture, resulting in a further refinement and correction of the formulation attempted in \cite{quillen-note}. 

The formulation of the results is best guided by Quillen's analysis \cite{quillen:spectrum} of the cohomology of an arithmetic group $\Gamma$ in terms of detection on elementary abelian $\ell$-subgroups, cf. also the discussion in Section~\ref{sec:quillenconj}. For each elementary abelian $\ell$-subgroup $A\leq\Gamma$, there is a restriction map $\op{H}^\bullet(\Gamma;\mathbb{F}_\ell)\to \op{H}^\bullet(A;\mathbb{F}_\ell)$, and one of the central inventions in \cite{quillen:spectrum} was the description of a homomorphism 
\[
\op{H}^\bullet(\Gamma;\mathbb{F}_\ell)\to \lim_A\op{H}^\bullet(A;\mathbb{F}_\ell)
\]
from the cohomology of the arithmetic group to the limit of cohomology rings over the category of all elementary abelian $\ell$-subgroups of the arithmetic group. The Quillen conjecture, or more generally the structure of  $\op{H}^\bullet(\Gamma;\mathbb{F}_\ell)$ as module over the Chern-class ring, can now be investigated by studying kernel and image of the Quillen homomorphism. Quillen also showed in his paper that there is a strong link between the isotropy spectral sequence for the $\Gamma$-action on the associated symmetric space and the Quillen homomorphism. Translated to the present situation of $\op{GL}_3(k[E])$, if we denote by $\op{R}_\bullet\op{H}^\bullet(\op{GL}_3(k[E]);\mathbb{F}_\ell)$ the ascending filtration (by Chern-class ring submodules) of the group cohomology associated to the isotropy spectral sequence for the action of $\op{GL}_3$ on the building, then it is clear that  $\op{R}_0\op{H}^\bullet=0$, $\op{R}_3\op{H}^\bullet=\op{H}^\bullet$ and Quillen's results imply that $\op{R}_2\op{H}^\bullet$ is precisely the kernel of the Quillen homomorphism. 

This brings us to the first result of the paper which provides a detailed description of the kernel of the Quillen homomorphism, as graded module over the Chern-class ring, cf.~page~\pageref{pfkernel} in Section~\ref{pfkernel}. 

\begin{theorem}
\label{thm:kernel}
Let $k=\mathbb{F}_q$ be a finite field, let $\overline{E}$ be an elliptic curve over $k$ with $k$-rational point $\op{O}$ and set $E=\overline{E}\setminus\{\op{O}\}$. Let $\ell\geq 5$ be a prime with $\ell\mid q-1$. Denote by $\op{rk}_\ell(\overline{E})$ the $\ell$-rank of $\overline{E}$. 
\begin{enumerate}
\item The submodule $\op{R}_2\subset \op{H}^\bullet(\op{GL}_3(k[E]);\mathbb{F}_\ell)$ is precisely the sub-$\mathbb{F}_\ell[\op{c}_1,\op{c}_2,\op{c}_3]$-module of torsion elements.
\item 
The submodule $\op{R}_1$ is obtained by restriction of a free $\mathbb{F}_\ell[\op{c}_1]$-module via the natural projection $\mathbb{F}_\ell[\op{c}_1,\op{c}_2,\op{c}_3]\to\mathbb{F}_\ell[\op{c}_1]$. The Hilbert--Poincar{\'e} series of $\op{R}_1/\op{R}_0$ is 
\[
\frac{(1+T)(q^3-\op{rk}_\ell\overline{E})}{(1-T^2)}.
\]
\item The quotient $\op{R}_2/\op{R}_1$ is obtained by restriction of a free $\mathbb{F}_\ell[\op{c}_1,\op{c}_2]$-module via the natural projection $\mathbb{F}_\ell[\op{c}_1,\op{c}_2,\op{c}_3]\to \mathbb{F}_\ell[\op{c}_1,\op{c}_2]$. The Hilbert--Poincar{\'e} series of $\op{R}_2/\op{R}_1$ is 
\[
\frac{(N_a+\op{rk}_\ell(\overline{E}))T(1+T)^2}{(1-T^2)(1-T^4)}.
\]
Here $N_a$ denotes the number of unordered triples of pairwise distinct  isomorphism classes of degree $0$ line bundles on $\overline{E}$ whose tensor product is trivial, cf.~Definition~\ref{def:na}.
\end{enumerate}
In particular, the kernel of the Quillen homomorphism is always non-trivial and has a filtration whose subquotients are free over appropriate Chern-class rings.
\end{theorem}

Note that all previously known examples of the failure of Quillen's conjecture were due to the inferred existence of torsion classes in the kernel, based on a failure of detection on elementary abelian $\ell$-subgroups. The above classes in the kernel of the Quillen homomorphism provide similar counterexamples to Quillen's conjecture in the function field case, but are actually very explicit. At the end of the day, one of the central underlying assumptions in Quillen's conjecture is that the cohomology of the arithmetic group is essentially controlled by elementary abelian $\ell$-subgroups of maximal rank, but this is not true in the present situation.\footnote{One could possibly say that the non-contractibility of subgroup complexes for arithmetic groups is one of the reasons for the existence of torsion classes.} The classes in $\op{R}_1$ are directly related to the non-trivial (cuspidal) cohomology in the quotient $\op{GL}_3(k[E])\backslash\mathfrak{B}$ of the building which has nothing to do with finite subgroups at all. The classes in $\op{R}_2/\op{R}_1$ are directly related to the non-triviality of a certain graph of semistable rank 3 vector bundles on $\overline{E}$. In a way, even under the simplifying assumptions made in Quillen's formulation of the conjecture, the category of elementary abelian $\ell$-subgroups, in particular the interactions between those of non-maximal rank, is still sufficiently complicated to produce many torsion classes in the cohomology of arithmetic groups.

An argument of Henn \cite{henn}, cf. Appendix~\ref{sec:henn}, shows that the failure of the Quillen conjecture propagates to higher-rank arithmetic groups, giving us the following consequence, cf.~page~\pageref{proof:henn} in Section~\ref{proof:henn}.

\begin{corollary}
\label{cor:henn}
Let $k=\mathbb{F}_q$ be a finite field, let $\overline{E}$ be an elliptic curve over $k$ with $k$-rational point $\op{O}$ and set $E=\overline{E}\setminus\{\op{O}\}$. Let $\ell\geq 5$ be a prime with $\ell\mid q-1$. Quillen's conjecture fails for $\op{H}^\bullet(\op{GL}_n(k[E]);\mathbb{F}_\ell)$ if and only if $n\geq 3$.
\end{corollary}

After the detailed analysis of the kernel, we can now turn to the image of the Quillen homomorphism. By Quillen's results, the image coincides with the part  $E^{0,\bullet}_\infty$ of the $E_\infty$-page of the isotropy spectral sequence and contains exactly those classes which can be detected on some elementary abelian $\ell$-subgroup. For the analysis of the image, it turns out to be very useful to define a so-called \emph{detection filtration} $\op{D}^\bullet$ on the image of the Quillen homomorphism which measures exactly how large the rank of an elementary abelian $\ell$-group has to be to detect a given cohomology class. Considering only subquotients for the detection significantly facilitates the spectral sequence computations. The result is the following, cf.~page~\pageref{pfimage} in Section~\ref{pfimage}: 

\begin{theorem}
\label{thm:image}
Let $k=\mathbb{F}_q$ be a finite field, let $\overline{E}$ be an elliptic curve over $k$ with $k$-rational point $\op{O}$ and set $E=\overline{E}\setminus\{\op{O}\}$. Let $\ell\geq 5$ be a prime with $\ell\mid q-1$. The numbers $N_a,N_b$ and $N_c$ below are introduced in Definition~\ref{def:na} and count various types of vector bundles on $\overline{E}$.
\begin{enumerate}
\item The filtration step $\op{D}^2\subset E^{0,\bullet}_\infty$ is a free $\mathbb{F}_\ell[\op{c}_1,\op{c}_2,\op{c}_3]$-module with Hilbert--Poincar{\'e} series 
\[
\frac{T^2(1+T)^3(1-T+T^2)\left((1+2T^2+2T^4+T^6)N_a+(T^2+T^4+T^6)N_b+ T^6N_c\right)}{(1-T^2)(1-T^4)(1-T^6)}.
\]
Its total rank is $8\cdot(\#\overline{E}(\mathbb{F}_q))^2$.
\item The subquotient $\op{D}^1/\op{D}^2$ of $E^{0,\bullet}_\infty$ is induced from a free $\mathbb{F}_\ell[\op{c}_1,\op{c}_2]$-module with Hilbert--Poincar{\'e} series
\[
\frac{T^2(1+T)^2\left(\op{rk}_\ell(\overline{E})+ (N_b+N_c)T+ (N_a+N_b+N_c)T^3)\right)}{(1-T^2)(1-T^4)}.
\]
Its total rank is $4(N_a+\op{rk}_\ell(\overline{E}))+8\cdot\#\overline{E}(\mathbb{F}_q)$.
\item The quotient $E^{0,\bullet}_\infty/\op{D}^1$ is induced from a free $\mathbb{F}_\ell[\op{c}_1]$-module with Hilbert--Poincar{\'e} series
\[
\frac{(1+T)}{(1-T^2)}.
\]
\end{enumerate}
\end{theorem}

While the detection filtration behaves very nicely, with free subquotients whose ranks are strongly related to the arithmetic of the underlying elliptic curve, the behaviour of the full image $E^{0,\bullet}_\infty$ (or just the term $E_2^{0,\bullet}$ related to sheaf cohomology of the building) is not as nice as Quillen's conjecture would suggest.

\begin{theorem}
\label{thm:free}
Let $k=\mathbb{F}_q$ be a finite field, let $\overline{E}$ be an elliptic curve over $k$ with $k$-rational point $\op{O}$ and set $E=\overline{E}\setminus\{\op{O}\}$. Let $\ell\geq 5$ be a prime with $\ell\mid q-1$. Assume furthermore that $\op{rk}_\ell(\overline{E})=0$ and $\op{rk}_3(\overline{E})>0$. Then $E^{0,\bullet}_\infty=E^{0,\bullet}_2$ is a torsionless $\mathbb{F}_\ell[\op{c}_1,\op{c}_2,\op{c}_3]$-module which is not free. In particular, the image of the Quillen homomorphism on $\op{H}^\bullet(\op{GL}_3(k[E]);\mathbb{F}_\ell)$ is not necessarily free over the Chern-class ring.
\end{theorem}

The above theorems concerning the failure of the Quillen conjecture significantly improve results known so far. The examples above show that the Quillen conjecture can already fail in the case $\op{GL}_3$,  the lowest possible rank for elliptic curves, while previous counterexamples were for ranks $\geq 14$. On the other hand, we can provide an infinite collection of rings supporting counterexamples while previous work on the Quillen conjecture focused mainly on the two examples $\mathbb{Z}[1/2]$ and $\mathbb{Z}[\zeta_3,1/3]$. Moreover, the computations in Theorems~\ref{thm:kernel} and \ref{thm:image} make the Chern-class module structure very explicit where previous counterexamples were established by indirect methods. The possibility of the failure of freeness of the torsion-free part is also a qualitatively new example. This gives precise information to guide possible reformulations of the conjecture which we discuss in Section~\ref{sec:quillenconj}. Finally, to my knowledge, the computations provided here are the first complete calculation of mod $\ell$ cohomology for rank 2 arithmetic groups in situations with non-trivial class groups. 

\subsection{Disclaimer}
Lengthy calculations lead to proportionally many mistakes and substantial efforts have been made to eliminate at least the most obvious ones. However: This paper makes use of the computations of equivariant cell structures of Bruhat--Tits buildings in \cite{ell-dilog} which have not yet been checked independently or on a computer. The fact that the first homology of the parabolic graph computes $\overline{E}(\mathbb{F}_q)\otimes\mathbb{F}_q^\times$ (which appears in the Somekawa-style presentation of $\op{K}_1$ of elliptic curves) has been checked using \verb!Sage! in a number of examples and provides evidence but no guarantee that the computation of the parabolic graph, the structure of automorphism groups and the description of the inclusions of stabilizer groups are correct. The computations of the $E^2$-page up to cohomological degree 4  (assuming quotient, stabilizer groups and inclusion maps as described in Sections~\ref{sec:equivariant} and \ref{sec:isotropy}) have also been verified using \verb!Sage! in a number of examples. Plans concerning implementation of  quotient computations for $\op{GL}_n(k[C])$-actions on Bruhat--Tits buildings for computer verification of the present results have been put in motion. At present, too little is known about general structural properties of arithmetic groups to allow even the most basic sanity checks.

\subsection{Structure of the paper} We begin with some preliminaries on cohomology of finite groups in Section~\ref{sec:reccohom}. Section~\ref{sec:equivariant} recalls the computation of the building quotient and Section~\ref{sec:isotropy} sets up the spectral sequence for the cohomology computations. The main bulk of computations, using the detection filtration to compute the $E_2$-page, is done in Sections~\ref{sec:detection} and \ref{sec:e2}. The $\op{d}_2$-differential computation is done in Section~\ref{sec:d2}. The discussion of consequences for Quillen's conjecture is done in Section~\ref{sec:quillenconj}. Some computations with alternating polynomials relevant for studying the detection filtration are deferred to Appendix~\ref{sec:eltalt}. Appendix~\ref{sec:reduction} recalls the general method to reduce quotients of arithmetic group actions on buildings to finitely many cells and Appendix~\ref{sec:henn} provides the function field generalization of Henn's argument.

\subsection{Acknowledgements} The results described in the paper would have been impossible to achieve without extensive experiments conducted using the \verb!Sage! computer algebra system as well as \verb!SageMathCloud!. Helpful discussions with Ga\"el Collinet, Hans-Werner Henn and Alexander D. Rahm are gratefully acknowledged. In particular Hans-Werner Henn's questions and suggestions concerning the category of elementary abelian $\ell$-subgroups and the Quillen homomorphism very much influenced the organization of the material. I am very grateful to Wolfgang Soergel for explaining to me the general picture behind the alternating polynomials, cf. Remark~\ref{rem:soergel}, which suggests that the methods in this paper are not restricted to the case $\op{GL}_3$. 

\section{Recollections on cohomology of finite groups}
\label{sec:reccohom}

For the computations in the paper, we need a short recollection on cohomology of finite groups. Most of the necessary results can be found in the book \cite{adem:milgram}; of central relevance are cohomology statements for general linear groups $\op{GL}_n(\mathbb{F}_q)$ which can be found in Quillen's paper \cite{quillen:ktheory} or Knudson's book \cite{knudson:book}. Since we only need the odd primary case for the present paper, we exclude the discussion of $\ell=2$ below.

\subsection{Cohomology of elementary abelian groups}
First, note that for $\ell\mid n$ the inclusion $\mathbb{Z}/\ell\mathbb{Z}\hookrightarrow\mathbb{Z}/n\mathbb{Z}$ induces an isomorphism $\op{H}^\bullet(\mathbb{Z}/n\mathbb{Z};\mathbb{F}_\ell)\cong \op{H}^\bullet(\mathbb{Z}/\ell\mathbb{Z};\mathbb{F}_\ell).$

From the K\"unneth formula, we obtain the cohomology of elementary abelian $\ell$-groups. We view the elementary abelian $\ell$-group $(\mathbb{Z}/\ell\mathbb{Z})^n$ as an $\mathbb{F}_\ell$-vector space, denoted by $V$. For $\ell$ odd, we have  
\[
\op{H}^\bullet(V;\mathbb{F}_\ell)\cong 
\op{Sym}^\bullet(V^\vee)\otimes_{\mathbb{F}_\ell}\bigwedge(V^\vee)
\]
where the symmetric algebra is generated in degree $2$ and the exterior algebra is generated in degree $1$. This is non-canonically isomorphic to 
\[
\mathbb{F}_\ell[y_1,\dots,y_n]\langle x_1,\dots,x_n\rangle, \quad  \deg x_i=1,\quad \deg y_i=2.
\]
Here and throughout the paper, $\mathbb{F}_\ell\langle a_1,\dots,a_m\rangle$ is our choice of notation for the exterior algebra generated by $a_1,\dots,a_m$.

Let $\phi:V_1\to V_2$ be a homomorphism of elementary abelian $\ell$-groups, viewed as $\mathbb{F}_\ell$-vector spaces. Then the induced homomorphism $\phi^\ast:\op{H}^\bullet(V_2;\mathbb{F}_\ell)\to\op{H}^\bullet(V_1;\mathbb{F}_\ell)$ on cohomology corresponds, under the above identifications, to the one induced by the  dual morphism $\phi^\vee:V_2^\vee\to V_1^\vee$. 

\subsection{Cohomology of general linear groups}
We provide a short recollection of Quillen's computation of cohomology of general linear groups over finite fields, cf. \cite{quillen:ktheory} or \cite[Chapter 1.1]{knudson:book}.

\begin{theorem}
\label{thm:glnfq}
Let $p$ and $\ell$ be distinct primes, both assumed to be odd. Let $r$ be the smallest positive number such that $q^r-1\equiv 0\bmod \ell$. Then there are isomorphisms
\[
\op{H}^\bullet(\op{GL}_n(\mathbb{F}_q);\mathbb{F}_\ell)\cong
\mathbb{F}_\ell[\op{c}_r,\dots,\op{c}_{r[n/r]}]\langle \op{e}_r,\dots,\op{e}_{r[n/r]}\rangle, 
\]
with $\deg \op{c}_{rj}=2rj$ and $\deg \op{e}_{rj}=2rj-1$.
\end{theorem}

In the present paper we make the assumption $\ell\mid q-1$ which is the function field analogue of the requirement $\zeta_\ell\in K$ made in Quillen's conjecture. This means that we will only need the case $r=1$ of the above theorem for the computations of $\op{GL}_3(k[E])$. In this case, the standard representative for the conjugacy class of maximal elementary $\ell$-subgroups is given by the diagonal matrices with entries $\ell$-th roots of unity. We want to describe the restriction morphism from  cohomology of $\op{GL}_n(\mathbb{F}_q)$ to a maximal elementary abelian $\ell$-subgroup. Therefore, let $i:T\hookrightarrow\op{GL}_n(\mathbb{F}_q)$ be  an elementary abelian  $\ell$-subgroup of maximal rank.  The corresponding restriction morphism factors as
\[
\op{res}:\op{H}^\bullet(\op{GL}_n(\mathbb{F}_q);\mathbb{F}_\ell)
\stackrel{\cong}{\longrightarrow}
\op{H}^\bullet(T;\mathbb{F}_\ell)^{\Sigma_n}\hookrightarrow
\op{H}^\bullet(T;\mathbb{F}_\ell).
\] 
The Weyl-group invariants are the symmetric polynomials, so that the restriction morphism can also, using the above identifications of cohomology rings, be written explicitly as follows: 
\begin{align*}
\mathbb{F}_\ell[\op{c}_1,\dots,\op{c}_{n}]\langle\op{e}_1,\dots,\op{e}_{n} \rangle&\stackrel{\op{res}}{\longrightarrow}
\op{Sym}^\bullet(x_1,\dots,x_n)\otimes_{\mathbb{F}_\ell}\bigwedge(y_1,\dots,y_n):\\
\op{c}_j&\mapsto \sum_{i_1<\cdots<i_j}x_{i_1}\otimes\cdots\otimes x_{i_j},\\
\op{e}_j&\mapsto \sum_{\substack{i_1<\cdots<i_j\\ 1\leq k\leq j}}x_{i_1}\otimes\cdots\otimes \widehat{x_{i_k}}\otimes\cdots\otimes x_{i_j}\otimes y_{i_k}. 
\end{align*}
These formulas can be found as (2.11) in an unpublished preprint of Quillen entitled ``The K-theory associated to a finite field I'' which was a preliminary version of \cite{quillen:ktheory}.

\subsection{Modules over Chern class rings}
As a next step, we will describe the structure of cohomology rings of elementary abelian $\ell$-groups as modules over the Chern class ring $\mathbb{F}_\ell[\op{c}_1,\dots,\op{c}_n]$. 

For this, let $k=\mathbb{F}_q$ be a finite field, let $k(E)$ be a function field of a curve over $k$ and let $A\leq \op{GL}_3(k(E))$ be an elementary abelian $\ell$-subgroup of rank 3. Since $\ell$ is different from the characteristic of $k$, the associated representation $A\leq\op{GL}_3(k[E]) \hookrightarrow \op{GL}_3(\overline{k(E)})$ is conjugate to one that factors through diagonal matrices and therefore through the inclusion $\op{GL}_3(\overline{k})\subset \op{GL}_3(\overline{k(E)})$. If $\ell\mid q-1$, the representation is even conjugate to one factoring through an inclusion $A\leq\op{GL}_3(k)$. But in this case, the restriction morphism $\op{H}^\bullet(\op{GL}_3(k);\mathbb{F}_\ell)\to \op{H}^\bullet(A;\mathbb{F}_\ell)$ for the inclusion is the one described in the previous subsection. The Chern-classes of this representation are then obviously the symmetric polynomials. As a consequence, via this restriction morphism, $\op{H}^\bullet(A;\mathbb{F}_\ell)$ is a free module over the Chern-class ring $\mathbb{F}_\ell$ of rank $\#\Sigma_3$. From the discussion in Section~\ref{known}, the Chern-class module structure induced on $A$ via the inclusion $A\leq \op{GL}_3(k[E])$ from the Chern-class module structure $\op{res}\circ\delta$ of Section~\ref{known} equals the one above. Similar statements are then true for the other relevant groups $\op{GL}_2(k)\times k^\times$ and $\op{GL}_3(k)$ in their natural 3-dimensional representations. 

\subsection{Hilbert--Poincar{\'e} series} 
We collect information on Hilbert--Poincar{\'e} series for cohomology rings of the finite groups which we will need for our computations. Recall that if $k$ is a field and $M$ is an $\mathbb{N}$-graded $k$-module with all $M_i$ finite-dimensional,  the Hilbert--Poincar{\'e} series is the formal power series
\[
\op{HP}(M,T)=\sum_{i\geq 0}\dim_k M_i\cdot T^i.
\]

Hilbert--Poincar{\'e} series for graded modules are additive for exact sequences, and multiplicative for graded tensor products. The Hilbert--Poincar{\'e} series of a graded module $M$ and its shift $M(-d)$ are related by $\op{HP}(M(-d),T)=T^d\cdot \op{HP}(M,T)$. From this, it follows that the Hilbert--Poincar{\'e} series of a free graded module $M$ over a ring $R$ is of the form $\op{HP}(M,T)=\op{HP}(R,T)\cdot F(T)$ where $F(T)$ is a Laurent polynomial with positive coefficients. This is the simplest test for non-freeness of graded modules.

We will be using Hilbert--Poincar{\'e} series for cohomology rings for groups (finite and arithmetic), as well as the graded Chern class rings $\mathbb{F}_\ell[\op{c}_1,\dots,\op{c}_n]$ and their graded modules; we provide the following collection of relevant Hilbert--Poincar{\'e} series for later reference.

\begin{itemize}
\item With $\deg \op{c}_i=2i$, the Hilbert--Poincar{\'e} series of the $n$-th Chern class ring is 
\[
\sum_{i=0}^\infty \dim_{\mathbb{F}_\ell}(\mathbb{F}_\ell[\op{c}_1,\dots,\op{c}_n]_i)\cdot T^i=
\frac{1}{\prod_{i=1}^n\left(1-T^{2i}\right)}.
\]
\item
If $A$ is an elementary abelian $\ell$-group of rank $n$, then the Hilbert--Poincar{\'e} series of its mod $\ell$-cohomology is
\[
\op{HP}\left(\op{H}^\bullet(A;\mathbb{F}_\ell),T\right)=
\frac{(1+T)^n}{(1-T^2)^n}.
\]
\item In the case where $\ell\mid q-1$, the Hilbert--Poincar{\'e} series for mod $\ell$ cohomology of $\op{GL}_n(\mathbb{F}_q)$  is of the form
\[
\op{HP}\left(\op{H}^\bullet(\op{GL}_n(\mathbb{F}_q);\mathbb{F}_\ell),T\right)=
\frac{\prod_{i=1}^n\left(1+T^{2i-1}\right)}{\prod_{i=1}^{n}\left(1-T^{2i}\right)}.
\]
\end{itemize}

\section{Equivariant cell structures of Bruhat--Tits buildings}
\label{sec:equivariant}

Let $k=\mathbb{F}_q$ be a finite field, let $\overline{E}$ be an elliptic curve over $k$ with $k$-rational point $\op{O}$, and set $E=\overline{E}\setminus\{\op{O}\}$. In this section we will recall from \cite{ell-dilog} the $\op{GL}_3(k[E])$-equivariant cell structure of the Bruhat--Tits building $\mathfrak{B}(E/k,3)$ associated to $\op{GL}_3(k(E))$ where $k(E)$ is equipped with the valuation $v_{\op{O}}$. 

\subsection{Equivariant cell structure}
The precise description of the cell structure is fairly complicated, because the quotient is in fact non-compact. The following two results describe the quotient and the essential statements on stabilizer groups of cells. 

Fix a projective embedding $\overline{E}\hookrightarrow \mathbb{P}^2$ having $\op{O}$ as point at infinity, corresponding to a choice of Weierstra\ss{} equation for $E$. We denote the associated degree 2 covering of the projective line by $\beta:\overline{E}\to\mathbb{P}^1$.

\begin{theorem}
\label{thm:quotient}
The quotient $\op{GL}_3(k[E])\backslash \mathfrak{B}(E/k,3)$ has the homotopy $\Sigma\mathcal{F}\ell_3(k)$ of the suspension of the flag complex for the $k$-vector space $k^3$. 
\end{theorem}

First note that the center $\mathcal{Z}(\op{GL}_3(k[E]))\cong k^\times$  acts trivially on the building, in particular all cell stabilizers contain the center of $\op{GL}_3(k[E])$. In the following description of the subcomplex with non-trivial stabilizers, unipotent contributions are completely ignored. For the purposes of cohomology computations, this is unproblematic because under our assumption $\ell\mid q-1$ the projection from a stabilizer to its maximal reductive quotient will induce isomorphisms in $\mathbb{F}_\ell$-cohomology. In the following result, ``non-trivial stabilizer'' means that a Levi subgroup of the stabilizer is not contained in the center of $\op{GL}_3(k[E])$. 

\begin{theorem}
\label{thm:parabolic}
There is an $\mathcal{H}_{G}^\ast$-reduction, cf.~Definition~\ref{def:hg}, from the subcomplex of cells with non-trivial stabilizer to a graph of groups $\Gamma_E$, called the \emph{parabolic graph of $E$}, described as follows: 
\begin{itemize}
\item the underlying graph of $\Gamma_E$ is given by the following diagram of moduli spaces of vector bundles on $\overline{E}$:
\[
\xymatrix{
\mathcal{M}_{2,1}(\overline{E})\ar[d]_\cong & \mathcal{M}_{2,0}(\overline{E}) \ar[d]^\cong \ar[l] \ar[r] & \mathcal{M}_{3,\mathcal{O}}(\overline{E}) \ar[d]^\cong
\\
\overline{E} & \op{Sym}^2\overline{E} \ar[l]^\mu \ar[r]_{\op{Sym}^2 \beta} &
\mathbb{P}^2
}
\]
The semi-stable bundles $\mathcal{V}$ for points in $\mathcal{M}_{2,i}(\overline{E})$ correspond to the rank 3 bundles $\mathcal{V}\oplus\det^{-1}\mathcal{V}$ on $\overline{E}$. 
\item the stabilizer of the vertices are given by the automorphism groups of the corresponding rank $3$ vector bundles, the stabilizers of the edges are the intersections of the vertex groups. 
\end{itemize}
\end{theorem}

\begin{remark}
There is a slight inaccuracy pertaining to the use of $k$-points in the above description of the graph. On the left-hand side, parametrizing sums of stable rank 2 bundles and the inverses of their determinants, we are really interested in the $k$-points giving the determinant line bundles of the stable summand. On the right-hand side, parametrizing rank 3 bundles with trivial determinant, we are really interested in the $k$-points giving unordered triples of line bundles invariant (as an unordered triple) under the Galois group. However, in the middle, we are interested in possibly more than $k$-points of $\op{Sym}^2(\overline{E})$ -- we are interested in all $\overline{k}$-points which map to $k$-points of $\mathbb{P}^2$. This essentially means that the corresponding semistable bundle of rank $2$ over $E\times_k\overline{k}$ has a determinant defined over $k$. In spite of this inaccuracy in the above description, it is probably more helpful (also for higher-rank generalizations) to think of the quotient being described in terms of a diagram of algebraic varieties resp. coarse moduli spaces.  
\end{remark}

\begin{remark}
\label{rem:chi}
The size of the graph for an elliptic curve $\overline{E}$ over a finite field $k=\mathbb{F}_q$ is easily computed: there are $\#\overline{E}(k)$ vertices of type (I) and $\#\mathbb{P}^2(k)= q^2+q+1$ vertices of type (II). Each type (I) point has a $\mathbb{P}^1(k)$ of adjacent edges, hence we have $\#\mathbb{P}^1(k)\cdot\#\overline{E}(k)$ edges. In particular, the Euler characteristic of the parabolic graph is 
\[
\#\overline{E}(k)+q^2+q+1-(q+1)\cdot\#\overline{E}(k)= q^2+q+1-q\cdot\#\overline{E}(k).
\]
Note however, that the parabolic graph is not necessarily connected, due to the bundles of type (IIe). There is, however, a single ``nontrivial'' connected component (the one containing the trivial bundle).
\end{remark}

For later computations of Hilbert--Poincar{\'e} series we introduce numbers counting the different types of bundles: 

\begin{definition}
\label{def:na}
For $?=a,b,c$, $N_?$ denotes the number of vertices with stabilizer of type (II?), and $N_e$ denotes the number of edges in the parabolic graph.
\end{definition}

\begin{remark}
Note that the above numbers are not independent, they all depend essentially on group of $\mathbb{F}_q$-rational points of $\overline{E}$. For example, $N_e=(q+1)\cdot\#\overline{E}(\mathbb{F}_q)$. 
\end{remark}

\begin{lemma}
\label{lem:ellform}
Let $k=\mathbb{F}_q$ be a finite field, let $\overline{E}$ be an elliptic curve over $k$ with $k$-rational point $\op{O}$ and set $E=\overline{E}\setminus\{\op{O}\}$. 
\begin{enumerate}
\item $6N_a+3N_b+N_c=(\#\overline{E}(\mathbb{F}_q))^2$.
\item  $N_b+N_c=\#\overline{E}(\mathbb{F}_q)$.
\item $N_c=3\op{rk}_3(\overline{E})$.
\item  $2N_a=-\chi'+N_c$ where $\chi'$ is the Euler characteristic of the nontrivial connected component of the parabolic graph. 
\end{enumerate}
\end{lemma}

\begin{proof}
(1) The identification $6N_a+3N_b+N_c=(\#\overline{E}(\mathbb{F}_q))^2$ can be seen as follows. Denote $X=\overline{E}(\mathbb{F}_q)$. The map $\mu:X^3\to X$ given by multiplication is surjective. Mapping a tuple in $X^3$ to the direct sum of the associated line bundles of degree $0$ provides a bijection between the $\Sigma_3$-orbits on $\ker\mu$ and the isomorphism classes of bundles of types (IIa-c). Now a bundle of type (IIa) has $\#\Sigma_3=6$ preimages in $X^3$, a bundle of type (IIb) has $\#(\Sigma_3/\Sigma_2)=3$ preimages, and a bundle of type (IIc) has only a single preimage. This proves the formula.\footnote{This can also be formulated as an argument using $\op{Sym}^2$ of the Weierstra\ss{} covering $E\to\mathbb{P}^1$, restricted to points such that each fiber has trivial residue field extensions.}

(2) Fix a line bundle $\mathcal{L}$ of degree $0$ on $\overline{E}$. Since we are counting bundles of type (IIb) or (IIc), the bundle must be of the form $\mathcal{L}\oplus\mathcal{L}\oplus\mathcal{L}^{-2}$. If $\mathcal{L}$ corresponds to a $3$-torsion point of the Jacobian, we get a bundle of type (IIc), otherwise of type (IIb). On the other hand, all bundles of type (IIb) and (IIc) arise like this, and no bundles are counted twice because of the asymmetry in (IIb). 

(3) This is clear since the number $N_c$ counts isomorphism classes of vector bundles of the form $\mathcal{L}^{\oplus 3}$ where $\mathcal{L}$ is a degree $0$ line bundle with $\mathcal{L}^{\otimes 3}\cong\mathcal{O}_{\overline{E}}$. 

(4) Note that the parabolic graph is a union of the nontrivial connected component containing the trivial bundle and a number of isolated points corresponding to bundles of type (IIe). This means the formula from Remark~\ref{rem:chi} is not sufficient for our purposes. We also want to ignore the points of type (IId) in the graph. This is ok, since these are leaves of the graph and omitting them does not change the Euler characteristic, but then we also need to ignore the corresponding adjacent edges as well. We can count the remaining edges because we know the numbers of edges adjacent to vertices of types (IIa-c), and every remaining edge is adjacent to exactly one of these. The number of edges is therefore $3N_a+2N_b+N_c$. The Euler characteristic of the nontrivial connected component is then
\[
\chi'=\left(N_a+N_b+N_c+\#\overline{E}(\mathbb{F}_q)\right)-(3N_a+2N_b+N_c)= -2N_a-N_b+\#\overline{E}(\mathbb{F}_q).
\]
Combining this equality with the one proved in (2) proves the claim.
\end{proof}

\subsection{Description of inclusion maps}
\label{sec:inclusions}
We now give a more detailed description of the parabolic graph, in particular of the inclusion homomorphisms mapping the edge stabilizers into the vertex stabilizers. First, we need a description of the automorphism groups of rank 3 vector bundles on $\overline{E}$. Actually, it suffices to describe Levi subgroups of the automorphism groups; the unipotent radicals of the automorphism groups will be $p$-groups and therefore have trivial $\mathbb{F}_\ell$-cohomology. 

\begin{enumerate}[(I)]
\item We first consider bundles corresponding to $\mathcal{M}_{2,1}(\overline{E})\cong\overline{E}$. A $k$-point in $\overline{E}$ corresponds to a degree $0$ line bundle $\mathcal{L}$ on $\overline{E}$, the associated point in $\mathcal{M}_{2,1}(\overline{E})$ corresponds to a stable bundle of rank $2$, degree $1$ and determinant $\mathcal{L}$, denoted by $\mathcal{V}_{\mathcal{L}}(2,1)$. The automorphism group of $\mathcal{V}_{\mathcal{L}}(2,1)\oplus\mathcal{L}^{-1}$ is
\[
\op{Aut}(\mathcal{V}_{\mathcal{L}}(2,1)\oplus\mathcal{L}^{-1})\cong k^\times\times k^\times.
\]
\item Now we consider bundles corresponding to $k$-points in $\mathcal{M}_{3,\mathcal{O}}(\overline{E})\cong\mathbb{P}^2$. Actually, these are $S$-equivalence classes of semistable rank 3 bundles with trivial determinant. In the $S$-equivalence class, we can choose a representative $\mathcal{V}$ which is geometrically completely decomposable, corresponding to an unordered triple $(\mathcal{L}_1,\mathcal{L}_2,\mathcal{L}_3)$ of line bundles over $\overline{E}\times_k\overline{k}$ with $\mathcal{L}_1\otimes\mathcal{L}_2\otimes\mathcal{L}_3\cong\mathcal{O}$. We have the following possibilities for the automorphism group of the associated vector bundle $\op{Aut}(\mathcal{V})$: 
\begin{enumerate}
\item If the splitting of $\mathcal{V}$ is defined over $k$ and all the summands $\mathcal{L}_i$ are pairwise non-isomorphic, then the automorphism  group is 
\[
\op{Aut}(\mathcal{L}_1\oplus\mathcal{L}_2\oplus\mathcal{L}_3)\cong k^\times\times k^\times\times k^\times. 
\]
\item If the splitting of $\mathcal{V}$ is defined over $k$ and two of the line bundles are isomorphic but non-isomorphic to the third, then the automorphism group is 
\[
\op{Aut}(\mathcal{L}\oplus\mathcal{L}\oplus\mathcal{L}^{-2})\cong \op{GL}_2(k)\times k^\times.
\]
\item If the splitting of $\mathcal{V}$ is defined over $k$ and all line bundles $\mathcal{L}_i$ are isomorphic, then the automorphism group is 
\[
\op{Aut}(\mathcal{L}^{\oplus 3})\cong \op{GL}_3(k). 
\]
\item If two of the line bundles form a Galois orbit corresponding to a degree $2$ point of $\overline{E}$ with residue field $L$, then the vector bundle is of the form $\mathcal{V}\cong \pi_\ast\mathcal{L}\oplus\op{Nm}(\mathcal{L})^{-1}$, where $\mathcal{L}$ is a line bundle $\mathcal{L}$ defined over $E\times_kL$ and $\pi:E\times_kL\to E$ is the natural degree 2 \'etale cover. In this case, the automorphism group is  
\[
\op{Aut}(\pi_\ast\mathcal{L}\oplus\op{Nm}(\mathcal{L})^{-1})\cong L^\times\times k^\times.
\]
\item If all three line bundles form a Galois orbit corresponding to a degree $3$ point of $\overline{E}$ with residue field $L$, the vector bundle is of the form $\mathcal{V}\cong\pi_\ast\mathcal{L}$ where again $\mathcal{L}$ is a line bundle on $E\times_kL$ and $\pi:E\times_kL\to E$ is the projection. The automorphism group is of the form
\[
\op{Aut}(\pi_\ast\mathcal{L})\cong L^\times. 
\]
\end{enumerate}
\item Finally, we discuss the bundles attached to the edges of the graph, i.e., bundles corresponding to points on $\op{Sym}^2\overline{E}\cong\mathcal{M}_{2,0}(\overline{E})$. Similar to (I), a point in $\op{Sym}^2\overline{E}$ corresponds to an unordered pair $(\mathcal{L}_1,\mathcal{L}_2)$ of degree $0$ line bundles on $\overline{E}\times_k\overline{k}$, and the associated point in $\mathcal{M}_{2,0}(\overline{E})$ is an $S$-equivalence class of semistable bundles, the split representative of which is $\mathcal{L}_1\oplus\mathcal{L}_2$. However, for the vector bundles labelling the edges we choose the non-split representative of the $S$-equivalence class (if available), which is denoted by $\mathcal{V}_{\mathcal{L}_1\otimes\mathcal{L}_2}(2,0)$. The edge group is the intersection of the automorphism groups of the vertices, which is of the form 
\[
\op{Aut}(\mathcal{V}_{\mathcal{L}_1\otimes\mathcal{L}_2}(2,0) 
\oplus(\mathcal{L}_1\otimes\mathcal{L}_2)^{-1}) \cong k^\times\times k^\times. 
\]
This is the automorphism group of the elementary transformation between the semi-stable and the stable bundle, and not necessarily isomorphic to the automorphism group of the semistable bundle associated to the edge. 
\end{enumerate}

Now we need to discuss the precise inclusion of the edge groups (automorphism groups listed in item (III) above) into the vertex groups (groups listed in items (I) and (II) above). Essentially, the inclusion of edge groups into vertex groups is given by the fact that the edge group is the intersection of the two vertex groups, i.e., it is the automorphism group of the elementary transformation. We make a case distinction for edges connecting type (I) and type (II), based on the subcases for type (II): 

\begin{enumerate}[(a)]
\item Assume the type (II) bundle is of subtype (a), i.e., it is of the form $\mathcal{V}\cong\mathcal{L}_1\oplus\mathcal{L}_2\oplus\mathcal{L}_3$ with all three summands being pairwise non-isomorphic. The edge connects this to  a stable bundle of the form $\mathcal{V}_{\mathcal{L}_i^{-1}}(2,1)\oplus\mathcal{L}_i$. For ease of notation, we now assume $i=3$. The inclusion of the automorphism group of the stable bundle $\mathcal{V}_{\mathcal{L}_3^{-1}}(2,1)$ into the automorphism group of $\mathcal{L}_1\oplus\mathcal{L}_2$ is as the center, hence the inclusion of stabilizers is
\[
(k^\times)^{\times 2}\to (k^\times)^{\times 3}:\op{diag}(x,y)\mapsto \op{diag}(x,x,y). 
\]
There are two other edges, for $i=1$ and $i=2$. The corresponding inclusions are given by
\[
\op{diag}(x,y)\mapsto \op{diag}(y,x,x) \quad\textrm{ and }\quad 
\op{diag}(x,y)\mapsto \op{diag}(x,y,x),
\]
respectively. The edge group inclusion into the automorphism group of the stable bundle is by the identity. 
\item Assume the type (II) bundle is of subtype (b), i.e., it is of the form $\mathcal{V}\cong\mathcal{L}\oplus\mathcal{L}\oplus\mathcal{L}^{-2}$ with $\mathcal{L}^3\not\cong\mathcal{O}$. There are two edges to consider, one where the corresponding stable bundle is obtained from an elementary transformation on $\mathcal{L}\oplus\mathcal{L}$, the other where the stable bundle is obtained from $\mathcal{L}\oplus\mathcal{L}^{-2}$. 

Consider first the edge connecting to $\mathcal{V}_{\mathcal{L}^2}(2,1)\oplus\mathcal{L}^{-2}$. As in case (a), the inclusion of the automorphism group of the stable bundle $\mathcal{V}_{\mathcal{L}^{2}}(2,1)$ into the automorphism group of  $\mathcal{L}\oplus\mathcal{L}$ is as the center, hence the inclusion of stabilizers is
\[
(k^\times)^{\times 2}\to \op{GL}_2(k)\times k^\times:\op{diag}(x,y)\mapsto \op{diag}(x,x,y). 
\]
For the other edge, we now have the inclusion of the automorphism group of $\mathcal{V}_{\mathcal{L}^{2}}(2,1)$ into the automorphism group of $\mathcal{L}\oplus\mathcal{L}^{-2}$ as the center. Therefore, it is of the form 
\[
(k^\times)^{\times 2}\to \op{GL}_2(k)\times k^\times:\op{diag}(x,y)\mapsto \op{diag}(x,y,x). 
\]
Note that this is conjugate to the inclusion via $\op{diag}(y,x,x)$ (or any other choice of maximal torus of $\op{GL}_2(k)$), so this choice does not affect the restriction maps on cohomology.
\item Assume the type (II) bundle is of subtype (c), i.e., it is of the form $\mathcal{V}\cong\mathcal{L}^{\oplus 3}$. There is a single edge to a type (I) bundle corresponding to an elementary transformation changing any choice of  $\mathcal{L}\oplus\mathcal{L}$ to $\mathcal{V}_{\mathcal{L}^2}(2,1)$. Again, the inclusion map on automorphism groups is the inclusion of the center on the stable summand, hence it is of the form
\[
(k^\times)^{\times 2}\to \op{GL}_3(k):\op{diag}(x,y)\mapsto \op{diag}(x,x,y). 
\]
As before, any choice of torus in $\op{GL}_3(k)$ will work because they are all conjugate. 
\item Assume the type (II) bundle is of subtype (d), i.e., it is of the form $\mathcal{V}\cong\pi_\ast\mathcal{L}\oplus\op{Nm}(\mathcal{L})^{-1}$ where $\mathcal{L}$ is a line bundle on $E\times_kL$, $L/k$ the residue field extension for a degree $2$ point on $E$ and $\pi:E\times_kL\to E$ the corresponding degree 2 covering. There is a single edge connecting this bundle to the bundle $\mathcal{V}_{\op{Nm}(\mathcal{L})}(2,1)\oplus\op{Nm}(\mathcal{L})^{-1}$. As usual, the inclusion of the automorphism groups is as homotheties of the stable bundle. This implies that the inclusion of automorphism groups is of the form 
\[
k^\times\times k^\times\to L^\times\times k^\times
\]
where the first factor is included via the obvious algebra map $k^\times\hookrightarrow L^\times$ and the second factor is the identity. 
\item Finally, assume that the type (II) bundle is of subtype (e). The bundle is an indecomposable rank 3 bundle which is geometrically decomposable, obtained from a degree 3 point of $\overline{E}$. Such bundles are isolated vertices of the parabolic graph, there is no subgroup isomorphic to $k^\times \times k^\times$ in $L^\times$ when $L/k$ is a field extension of degree $3$. 
\end{enumerate}

It should be pointed out that, while not made explicit in each of the above cases, the edge group is always included into the type (I) vertex group via the identity. Of course, this is implicit in the above descriptions. 

Finally, it is important to remark that the center $\mathcal{Z}(\op{GL}_3(k[E]))\cong k^\times$ is always included in the stabilizers as the homotheties of the corresponding rank three vector bundles. 

\subsection{Setwise vs pointwise stabilizers}
\label{sec:rigid}

For the spectral sequence calculation, we need that the cell stabilizers always fix the cells pointwise. We give an argument why this is true in our setting. 

Recall that the type of a lattice class $\Lambda$  corresponding to a point in the Bruhat--Tits building $\mathfrak{B}(E/k,3)$ is defined as $v_{\op{O}}(\det \Lambda)\bmod 3$. If the action of  a group $G$ on $\mathfrak{B}(E/k,3)$ preserves types, then the stabilizers of simplices of the building will fix these simplices pointwise because no two vertices of any simplex will have the same type. 

Now, in our specific situation, we only remove a single $k$-rational point $\op{O}$ from the elliptic curve $\overline{E}$ and therefore $k[E]^\times\cong k^\times$. This implies that for any matrix $A\in\op{GL}_3(k[E])$, the determinant of $A$ will satisfy $v_{\op{O}}(\det A)=0$. In particular, the action of $\op{GL}_3(k[E])$ on $\mathfrak{B}(E/k,3)$ will preserve types, hence cell stabilizers fix their cells pointwise.

\subsection{First homology of the parabolic graph}

We provide one more bit of computation related to the parabolic graph, namely a part of its first homology. This is directly related to computations of the first K-group of the elliptic curve \cite{ell-dilog} and will be necessary input for the cohomology computations in Section~\ref{sec:e2}. 

Recall that the present section introduced the parabolic graph as a graph of groups: the vertices and edges of the graph are rank three vector bundles and the associated groups are the automorphism groups of the vector bundles. 

\begin{definition}
Let $\Gamma=(V_0,V_1,G,s,t)$ be a graph of groups, i.e., 
\begin{itemize}
\item a graph $(V_0,V_1,s,t)$ consisting of a set $V_0$ of vertices and a set $V_1$ of edges together with maps $s$ and $t$ mapping edges to source and target,
\item an assignment $G$ mapping edges and vertices to groups, together with embeddings of edge groups into the adjacent vertex groups.
\end{itemize}
Then we can define the following homology group
\[
\widehat{\op{H}}_1(\Gamma):=\op{coker}\left(\bigoplus_{e\in V_1}G(e)^{\op{ab}}\stackrel{G(s)-G(t)}{\longrightarrow} \bigoplus_{v\in V_0}G(v)^{\op{ab}}\right).
\]
If $\Gamma$ is the nontrivial connected component of the graph of groups associated to an elliptic curve $E/k$ as in the Theorem~\ref{thm:parabolic}, the homology group is denoted by $\widehat{\op{H}}_1(E/k)$.
\end{definition}

The group defined is not the first homology of the parabolic graph because it does not include the kernel of the induced maps on $\op{H}_0$ which would count the loops in the graph. The notation $\widehat{\op{H}}_1$ is inspired from the analogy with Farrell--Tate homology for groups of finite virtual cohomological dimension. However, the group is also not the first Farrell--Tate homology of $\op{GL}_3(k[E])$. 

From a conceptual perspective, the following computation essentially provides a Somekawa-style presentation of  $\op{H}_1(\op{GL}_3(k[E]);\mathbb{Z})\cong\op{K}_1(E)$ as quotient of $\overline{E}(k)\otimes_\mathbb{Z}k^\times$, cf. \cite{ell-dilog}.

\begin{proposition}
\label{prop:k1}
Let $k$ be a field, let $\overline{E}$ be an elliptic curve over $k$ with $k$-rational point $\op{O}$ and set $E=\overline{E}\setminus\{\op{O}\}$. 
Then there is an isomorphism
\[
\widehat{\op{H}}_1(E/k)\cong k^\times\oplus  \left(\overline{E}(k)\otimes_{\mathbb{Z}}k^\times\right)\oplus \bigoplus_Pk(P)^\times/k^\times
\]
where $P$ runs through degree 2 points on $E$.
\end{proposition}

\begin{proof}
(0) Recall that $\widehat{\op{H}}_1(E/k)$ is the cokernel of a map to abelianizations of automorphism groups of rank three vector bundles on $\overline{E}$. In particular, we can represent elements of $\widehat{\op{H}}_1(E/k)$ as tuples of automorphisms, indexed by rank three vector bundles.  

(1) Note that there is a bijection between the bundles of type (IId) and points of degree 2. The stabilizers of points of type (IId) are of the form $k(P)^\times\times k^\times$ for a point $P$ of degree 2. The inclusion of the edge group is the natural inclusion $k^\times\times k^\times\hookrightarrow k(P)^\times\times k^\times$. We can therefore split off the contributions from type (IId) points and remove the adjacent edges. This contributes the last sum, so that we can concentrate on the subgraph containing only the points of type (IIa-c) and the stable bundles. 

(2) 
The map to the first direct summand is given by the determinant 
\[
\widehat{\op{H}}_1(E/k)\to k^\times: (\phi_v)\in \bigoplus_v \op{Aut}(\mathcal{V}_v)^{\op{ab}}\mapsto \prod_v \det\phi_v.
\]
This map is well-defined: an element of an edge group of the parabolic graph is an automorphism of an elementary transformation between vector bundles. The boundary map defining $\widehat{\op{H}}_1(E/k)$ maps this to an automorphism of the two vector bundles involved, with opposite signs. Therefore, the relations coming from edges map to $1\in k^\times$ under the above map. Note that the determinant map is surjective. Since $k^\times$ can be embedded as automorphisms of the trivial bundle, the $k^\times$ also splits off as a direct summand from $\widehat{\op{H}}_1(E/k)$. 

(3) 
Now we want to define a morphism $\widehat{\op{H}}'_1(E/k)\to \overline{E}(k)\otimes_{\mathbb{Z}}k^\times$ where $\widehat{\op{H}}'_1$ denotes the quotient of $\widehat{\op{H}}_1(E/k)$ modulo the contributions identified in (1) and (2). We noted in (0) that we can represent elements of $\widehat{\op{H}}'_1(E/k)$ by tuples of automorphism of rank three vector bundles over $\overline{E}/k$, and by the reduction in (1) we know that we only need vector bundles of types (IIa-c) and (I). 

To define the map on automorphisms of vector bundles of type (I) recall that such a vector bundle is of the form $\mathcal{V}_{\mathcal{L}}(2,1)\oplus\mathcal{L}^{-1}$ where the first summand is a stable bundle of rank 2 with determinant $\mathcal{L}$. An automorphism of this bundle is of the form $(h,u)$ where $h$ is a homothety of the stable bundle and $u$ is a homothety of $\mathcal{L}^{-1}$. The automorphism $(h,u)$ on $\mathcal{V}_{\mathcal{L}}(2,1)\oplus\mathcal{L}^{-1}$ is mapped to $\det\mathcal{V}\otimes h+\mathcal{L}^{-1}\otimes u$ in $\overline{E}(k)\otimes_{\mathbb{Z}}k^\times$.

To define the map on automorphisms of vector bundles of type (II) recall that each such vector bundle is of the form $\mathcal{L}_1\oplus\mathcal{L}_2\oplus\mathcal{L}_3$, with subtypes a,b and c determined by possible isomorphisms between these degree $0$ line bundles. Since we are only interested in abelianizations of automorphism groups, the automorphism can be chosen to be diagonalized, i.e., of the form $(u_1,u_2,u_3)$ with units $u_i\in k^\times$ acting as homotheties on the summands $\mathcal{L}_i$, respectively. This automorphism of the bundle is mapped to $\sum_{i=1}^3\mathcal{L}_i\otimes u_i$. 

Now we need to check that this map is well-defined. Recall that the relations from edges come from compatible automorphisms of elementary transformations: if we have a bundle $\mathcal{L}_1\oplus\mathcal{L}_2\oplus\mathcal{L}_3$ and transform it to the bundle $\mathcal{V}_{\mathcal{L}_3^{-1}}(2,1)\oplus\mathcal{L}_3$, a compatible automorphism must act as homothety on the stable rank 2 summand \emph{and} as homothety on $\mathcal{L}_1\oplus\mathcal{L}_2$. The differential maps this automorphism $(u_1,u_1,u_3)$ to the corresponding automorphisms of the two bundles above, with signs coming from the orientation of the parabolic graph. The (IIa) automorphism will be mapped to $-\mathcal{L}_1\otimes u_1-\mathcal{L}_2\otimes u_1-\mathcal{L}_3\otimes u_3$. The (I) automorphism will be mapped to $(\mathcal{L}_1\otimes\mathcal{L}_2)\otimes u_1+\mathcal{L}_3\otimes u_3$. This shows that the map is well-defined on the quotient $\widehat{\op{H}}_1(E/k)$.

(4)
The map defined is obviously surjective because $[\mathcal{L}]\otimes u$ is in the image of the corresponding automorphism $(u,1,1)$ of the bundle $\mathcal{L}\oplus\mathcal{L}^{-1}\oplus\mathcal{O}$. 

(5) The assignment in (4) does not directly give rise to a potential inverse morphism from the tensor product to the cohomology. We have a built-in linearity on the side of the units, but not on the side of the line bundles. To show that the map $\overline{E}(k)\otimes_{\mathbb{Z}}k^\times\to \widehat{\op{H}}'_1$ described in (4) is actually well-defined, we need to establish bilinearity in the line bundle component.

Therefore, consider vector bundles $\mathcal{L}_1\oplus\mathcal{L}_1^{-1}\oplus\mathcal{O}$ and $\mathcal{L}_2\oplus\mathcal{L}_2^{-1}\oplus\mathcal{O}$. Assume both are equipped with an automorphism of the form $(u,1,1)$. We want to use  the edge relations map to show that the sum of these two automorphisms equals the automorphism $(u,1,1)$ on $\mathcal{L}_3\oplus\mathcal{L}_3^{-1}\oplus\mathcal{O}$ where $\mathcal{L}_3\cong\mathcal{L}_1\otimes\mathcal{L}_2$. 
From the latter bundle, we can make elementary transformations to the unstable bundle $\mathcal{L}_3\oplus\mathcal{L}_1^{-1}\oplus\mathcal{L}_2^{-1}$. Since we are working in $\widehat{\op{H}}_1'(E/k)$, we can change the automorphism modulo homotheties. The result is the automorphism $(1,u^{-1},u^{-1})=(1,u^{-1},1)+(1,1,u^{-1})$ on $\mathcal{L}_3\oplus\mathcal{L}_1^{-1}\oplus\mathcal{L}_2^{-1}$. Now we can use the edges from this bundle to the bundles $\mathcal{L}_1\oplus\mathcal{L}_1^{-1}\oplus\mathcal{O}$ and $\mathcal{L}_2\oplus\mathcal{L}_2^{-1}\oplus\mathcal{O}$ which are compatible with the summands of the automorphism $(1,u^{-1},u^{-1})$. 

We need to show that these automorphisms (where the $u^{-1}$ acts on the $\mathcal{L}_i^{-1}$-component) are equal to the automorphisms given above. Alternatively, we need to show that for a bundle $\mathcal{L}\oplus\mathcal{L}^{-1}\oplus\mathcal{O}$, the automorphisms $(u,1,1)$ and $(1,u,1)$ are inverses in $\widehat{\op{H}}'_1$. The latter automorphism, up to multiplication by the center is $(u^{-1},1,u^{-1})$ and the summand $(1,1,u^{-1})$ is trivial because we are computing in $\widehat{\op{H}}'_1$. 
This shows bilinearity in the line bundle component.


(6)
It remains to show that the above morphisms are inverses. To do this, we will provide simple representatives for the elements in $\widehat{\op{H}}_1(E/k)$.

We can use the edge relations to get rid of all automorphisms of bundles of type (II) -- by moving the automorphism $(u,1,1)$ of $\mathcal{L}_1\oplus\mathcal{L}_2\oplus\mathcal{L}_2$ to the automorphism $(u,1)$ of $\mathcal{L}_1\oplus\mathcal{V}_{\mathcal{L}_1^{-1}}(2,1)$ etc.

Now if we have an automorphism of a stable bundle $\mathcal{V}_{\mathcal{L}^{-1}}(2,1)\oplus\mathcal{L}$, there is an elementary transformation to $\mathcal{O}\oplus\mathcal{L}^{-1}\oplus\mathcal{L}$ compatible with it. Using edge relations, we can reduce to have only automorphisms of bundles $\mathcal{L}\oplus\mathcal{L}^{-1}\oplus\mathcal{O}$; these are all in the link of the stable bundle $\mathcal{V}(2,1)\oplus\mathcal{O}$. Now we apply the map to compute the corresponding element in $\overline{E}(k)\otimes_{\mathbb{Z}}k^\times$. Note that the line bundle $\mathcal{O}$ corresponds to the unit of $\overline{E}(k)$, hence the corresponding units do not contribute to the result. Homotheties of $\mathcal{L}\oplus\mathcal{L}^{-1}$ can be moved to the center stable bundle using edge relations, and this way we can assume that the unit $u_1$ on the summand $\mathcal{L}^{-1}$ is always 1. Note that the units moved to the center stable bundle do not affect the end result because this stable bundle also maps to the unit of $\overline{E}(k)$. This special representative is now mapped to 
\[
\sum_{x\in\mathbb{P}^1(k)}\mathcal{L}_x\otimes u_x
\]
with almost all units being $1\in k^\times$.

From this, it is now obvious that the two maps are inverses, and the proof is complete.
\end{proof}

\section{The isotropy spectral sequence for group cohomology}
\label{sec:isotropy}

In this section, we discuss the structure of the $E_1$-page of the isotropy spectral sequence associated to the $\op{GL}_3(k[E])$-action on the Bruhat--Tits building $\mathfrak{B}(E/k,3)$. 

\subsection{Recollections on the isotropy spectral sequence}
We shortly recall the Borel isotropy spectral sequence for the computation of equivariant cohomology, cf. \cite[Section VII]{brown:book} or \cite[Appendix A]{knudson:book}.  For a $G$-complex $X$, there is an action of $G$ on the cellular chain complex $\op{C}_\bullet(X)$. Fixing a resolution $\mathcal{P}_\bullet\to\mathbb{Z}$ by projective $\mathbb{Z}[G]$-modules, the corresponding cohomology groups 
\[
\op{H}_G^\bullet(X;M):=\op{H}^\bullet(\op{Hom}_G(\mathcal{P}_\bullet,\op{C}_\bullet(X)\otimes_{\mathbb{Z}[G]}M))
\]
are called \emph{equivariant cohomology groups of $(G,X)$}. If $X$ is contractible, then the equivariant cohomology of $X$ is isomorphic to the group cohomology of $G$. 

The stupid filtration of the complex $\op{C}_\bullet(X)\otimes_{\mathbb{Z}[G]}M$ gives rise to the following spectral sequence, which will be called \emph{isotropy spectral sequence}
\[
E^{i,j}_1=\prod_{\sigma\in(G\backslash X)_{(i)}}\op{H}^j(G_\sigma,M_\sigma)\Rightarrow
\op{H}_G^{i+j}(X;M). 
\]
In the above, $(G\backslash X)_{(i)}$ denotes a set of representatives for the orbits of the $G$-action on the $i$-cells of $X$. The module $M_\sigma$ is the coefficient $G$-module $M$ twisted by the orientation character $\chi_\sigma:G_\sigma\rightarrow\mathbb{Z}/2$ of the stabilizer group $G_\sigma$. 
\begin{remark}
In the special case where each cell stabilizer fixes the cell pointwise, the orientation character is trivial. By the discussion in \ref{sec:rigid}, this is the case for the $\op{GL}_3(k[E])$-action on $\mathfrak{B}(E/k,3)$. 
\end{remark}
\begin{remark}
Note that the $E_1$-page features direct products as opposed to the direct sums for homology. This is relevant, since the quotient $\op{GL}_3(k[E])\backslash\mathfrak{B}(E/k,3)$ has infinitely many cells. However, the description of the quotient in Section~\ref{sec:equivariant} shows that (because of a suitable $\mathcal{H}_G^\ast$-reduction) only finitely many cells are relevant for the computation of the $E_2$-page. Still, it is important to remember that the initial description of the $E_1$-page (before making any reductions to finite skeleta) contains infinite direct products. 
\end{remark}

With the above indexing, the differentials of the isotropy spectral sequence are of the form 
\[
\op{d}_r^{i,j}:E_r^{i,j}\to E_r^{i+r,j-r+1}.
\]
Up to signs coming from the choice of orientation for the cells, the first differential is induced from the boundary map of the complex $X$ and inclusions of stabilizers, cf. \cite[VII.8]{brown:book}:
\[
\op{d}_1|_{\op{H}^i(G_\sigma,M_\sigma)}:\op{H}^i(G_\sigma,M_\sigma)\mapsto \bigoplus_{\tau\supseteq \sigma} \op{H}^i(G_\tau,M_\tau)
\]

\begin{remark}
Note that at this point, it is justified to write a direct sum, because the building is locally compact and therefore any vertex is contained in only finitely many simplices. Note also that any given target group $\op{H}^i(G_\tau,M_\tau)$ can only be hit by finitely many differentials because there are only finitely many simplices contained in $\tau$. In particular, the differential is completely described by the above restrictions to single factors.
\end{remark}

The isotropy spectral sequence will be applied to the action of  $\op{GL}_3(k[E])$ on the associated Bruhat--Tits building $\mathfrak{B}(E/k,3)$. The building $\mathfrak{B}(E/k,3)$ is a $2$-dimensional simplicial complex, therefore the only non-trivial $d_2$-differentials are $d^{0,j}_2:E^{0,j}_2\to E^{2,j-1}_2$, and the spectral sequence degenerates at the term $E_3=E_\infty$. 

As our main interest is the module structure of $\op{H}^\bullet(\op{GL}_3(k[E]);\mathbb{F}_\ell)$ over the Chern class ring $\mathbb{F}_\ell[\op{c}_1,\op{c}_2,\op{c}_3]$, we need more precise information on the multiplicative structure of the isotropy spectral sequence. The isotropy spectral sequence is a spectral sequence of algebras in the sense \cite{hatcher,mccleary}, corresponding statements for Farrell--Tate cohomology are discussed in \cite[Section X.4]{brown:book}. 
This means that there are bilinear products $E_r^{i,j}\times E_r^{s,t}\to E_r^{i+s,j+t}$ such that 
\begin{itemize}
\item each differential $\op{d}$ is a derivation, i.e., $\op{d}(xy)=(\op{d}x)y+(-1)^{i+j}x(\op{d}y)$ for $x\in E^{i,j}_r$. The given product on $E_{r+1}$ coincides with the product induced from $E_r$, and
\item denoting the filtration on $\op{H}^\bullet_{\op{GL}_3(k[E])}(\mathfrak{B}(E/k,3);\mathbb{F}_\ell)$ by $\op{FH}^\bullet_r$, the cup product on $\op{H}^\bullet_{\op{GL}_3(k[E])}(\mathfrak{B}(E/k,3);\mathbb{F}_\ell)$ restricts to maps $\op{FH}^m_s\times \op{FH}^n_t\to \op{FH}^{m+n}_{s+t}$. The induced quotient maps 
\[
\op{FH}^m_s/\op{FH}^m_{s+1}\times \op{FH}^n_t/\op{FH}^n_{t+1}\to \op{FH}^{m+n}_{s+t}/\op{FH}^{m+n}_{s+t+1}
\]
coincide with the products $E_\infty^{s,m-s}\times E^{t,n-t}_\infty\to E_\infty^{s+t,m+n-s-t}$.
\end{itemize}
As a particular case of the derivation property, the kernel of $\op{d}_1$ in $E^{0,\bullet}_1$ is closed under $\cup$-products. However, it is not in general an ideal in the product ring $\prod_{\sigma\in (G\backslash X)_{(0)}}\op{H}^\bullet(G_\sigma;\mathbb{F}_\ell)$.

It remains to identify the product structure on the $E_1$-page. Denoting $G=\op{GL}_3(k[E])$ and $X=\mathfrak{B}(E/k,3)$, the product on the $E_1$-page is induced from the composition
\begin{eqnarray*}
\op{H}^q(G;\op{C}^p(X))\otimes \op{H}^{q'}(G;\op{C}^{p'}(X))&\to& \op{H}^{q+q'}(G;\op{C}^p(X)\otimes\op{C}^{p'}(X))\\&\to& \op{H}^{q+q'}(G;\op{C}^{p+p'}(X))
\end{eqnarray*}
where the first map is the usual cup product in group cohomology and the second map is the cup product on $\op{C}^\bullet(X)$. From this description, we deduce the following projection formula: given two elements 
\[
u\in \op{H}^i(G_\sigma;\mathbb{F}_\ell)\subseteq E^{0,i}_1\quad\textrm{ and }\quad v\in \op{H}^j(G_\tau;\mathbb{F}_\ell)\subseteq E^{n,j}_1,
\]
the product structure on the $E_1$-page yields 
\[
u\cup_1 v=\op{res}_{G_\sigma}^{G_\tau}(u)\cup v\in \op{H}^{i+j}(G_\tau;\mathbb{F}_\ell)\subseteq E^{n,i+j}_1,
\]
where $\op{res}_{G_\sigma}^{G_\tau}:\op{H}^i(G_\sigma;\mathbb{F}_\ell)\to
\op{H}^i(G_\tau;\mathbb{F}_\ell)$ is the usual restriction on group cohomology, set to $0$ if $\sigma\not\subseteq \tau$, and the product on the right-hand side is the ordinary $\cup$-product in $\op{H}^\bullet(G_\tau;\mathbb{F}_\ell)$.

This is now enough to describe the structure of the $E_1$-page as a module over the Chern class ring $\mathbb{F}_\ell[\op{c}_1,\op{c}_2,\op{c}_3]$. Recall from the discussion in Section~\ref{sec:quillenconj} that the module structure is induced by 
\[
\op{res}\circ\delta:\mathbb{F}_\ell[\op{c}_1,\op{c}_2,\op{c}_3]\to
\op{H}^\bullet(\op{GL}_3(\overline{k(E)});\mathbb{F}_\ell)\to \op{H}^\bullet(\op{GL}_3(k[E]);\mathbb{F}_\ell).
\]
Denoting again $G=\op{GL}_3(k[E])$, composition with the projection 
\[
\op{H}^\bullet(G;\mathbb{F}_\ell)\to E_1^{0,\bullet}=\prod_{\sigma\in(G\backslash X)_{(0)}} \op{H}^\bullet(G_\sigma;\mathbb{F}_\ell)
\]
maps $\op{c}_i$ to the element which in each factor $\op{H}^\bullet(G_\sigma;\mathbb{F}_\ell)$ is the $i$-th Chern class of the representation $G_\sigma\hookrightarrow \op{GL}_3(\overline{k(E)})$. Now we combine this with the above projection formula to see that the action of $\op{c}_i$ on a class $v\in \op{H}^j(G_\tau;\mathbb{F}_\ell)\subseteq E^{n,j}_1$ is given by 
\[
\left(\op{res}_{\op{GL}_3(\overline{k(E)})}^{G_{\tau}}\circ\delta\right)(\op{c}_i)
\cup v,
\]
where now the map applied to $\op{c}_i$ is the composition of $\op{res}\circ\delta$ with the restriction map $\op{H}^\bullet(G;\mathbb{F}_\ell)\to\op{H}^\bullet(G_\tau;\mathbb{F}_\ell)$.


As a consequence of the above discussion, we now have the following:

\begin{proposition}
\label{prop:multiplicative}
The Chern-class ring module structure induced via $\op{res}\circ\delta$ on the $E_1$-page of the isotropy spectral sequence associated to the $\op{GL}_3(k[E])$-action on the associated building $\mathfrak{B}(E/k,3)$ is the natural one: for a $\op{GL}_3(k[E])$-orbit $\tau$ of simplices in $\mathfrak{B}(E/k,3)$ with stabilizer group $G_\tau$, the action of $\op{c}_i$ on the factor $\op{H}^\bullet(G_\tau;\mathbb{F}_\ell)$ is given by cup product with the $i$-th Chern class of the representation $G_\tau\hookrightarrow \op{GL}_3(\overline{k(E)})$. 
\end{proposition}

\subsection{Description of restriction maps}
\label{sec:restrictions}

The next goal is the computation of the $\op{d}_1$-differentials for the isotropy spectral sequence. As noted above, these are given by restriction maps induced from inclusions of stabilizer subgroups. In this subsection, we will explicitly describe all the relevant restriction maps. We follow the case distinction introduced in Section~\ref{sec:inclusions}. 


First, if $\sigma$ is a $0$-cell of type (I), corresponding to a bundle $\mathcal{V}_{\mathcal{L}}(2,1)\oplus\mathcal{L}^{-1}$, its stabilizer is isomorphic to $k^\times\times k^\times$. For each of the $q+1$ adjacent edges in the parabolic graph, the edge stabilizer is also $k^\times\times k^\times$ and the inclusion map is the identity. Therefore, the restriction map for type (I) is the identity on cohomology. 

Now we consider the stabilizer groups of vertices of type (II). We consider the finite field $k=\mathbb{F}_q$ and cohomology coefficients $\mathbb{F}_\ell$ with $\ell\mid q-1$. This also leads to simplifications for bundles of types (d) and (e). The formulas are obtained by the obvious evaluation of the composition of the inclusion $\op{H}^\bullet(G_\sigma;\mathbb{F}_\ell)\subseteq \op{H}^\bullet((\mathbb{F}_q^\times)^3;\mathbb{F}_\ell)$ with the restriction maps $\op{H}^\bullet((\mathbb{F}_q^\times)^3;\mathbb{F}_\ell)\to\op{H}^\bullet((\mathbb{F}_q^\times)^2;\mathbb{F}_\ell)$ induced from the diagonal inclusions $(x,y)\mapsto (x,x,y)$ etc.

\begin{enumerate}[(a)]
\item For $\sigma$ a vertex of subtype (IIa), the vertex stabilizer is $(\mathbb{F}_q^\times)^{\times 3}$, and there are three adjacent edges $\tau$ all stabilized by $(\mathbb{F}_q^\times)^{\times 2}$.  As determined in Section~\ref{sec:inclusions}, the inclusions of the edge group are 
\[
(x,y)\mapsto (x,x,y),\, (x,y)\mapsto (x,y,x),\,\textrm{ and }\,(x,y)\mapsto (y,x,x),
\]
respectively. We can rewrite these as linear maps between $\mathbb{F}_\ell$-vector spaces, and the representing matrices for the standard bases are
\[
\left(\begin{array}{cc}
1&0\\1&0\\0&1
\end{array}\right),\,
\left(\begin{array}{cc}
1&0\\0&1\\1&0
\end{array}\right),\,\textrm{ and }\,
\left(\begin{array}{cc}
0&1\\1&0\\1&0
\end{array}\right).
\]
Using the identification from Section~\ref{sec:reccohom}, the three induced maps on  cohomology $\op{H}^\bullet(\mathbb{F}_\ell^{3};\mathbb{F}_\ell)\to \op{H}^\bullet(\mathbb{F}_\ell^2;\mathbb{F}_\ell)$ can be written as
\[
\mathbb{F}_\ell[X_1,X_2,X_3]\langle Y_1,Y_2,Y_3\rangle\to
\mathbb{F}_\ell[X_1',X_2']\langle Y_1',Y_2'\rangle
\]
\[
\left\{\begin{array}{c}
X_1\mapsto X_1'\\ X_2\mapsto X_1'\\ X_3\mapsto X_2'
\end{array}\right\},\,
\left\{\begin{array}{c}
X_1\mapsto X_1' \\ X_2\mapsto X_2' \\ X_3\mapsto X_1'
\end{array}\right\},\,\textrm{ and }\,
\left\{\begin{array}{c}
X_1\mapsto X_2' \\ X_2\mapsto X_1' \\ X_3\mapsto X_1'
\end{array}\right\}
\]
and the same definitions for the exterior variables $Y_i$.
\item For $\sigma$ a vertex of subtype (IIb), the vertex stabilizer is $\op{GL}_2(\mathbb{F}_q)\times \mathbb{F}_q^\times$, and there are two adjacent edges $\tau$ stabilized by $(\mathbb{F}_q^\times)^{\times 2}$. The inclusions of the edge group  are $(x,y)\mapsto \op{diag}(x,x,y)$ and $(x,y)\mapsto \op{diag}(x,y,x)$, respectively, where $\op{GL}_2(\mathbb{F}_q)\times\mathbb{F}_q^\times$ is thought embedded in $\op{GL}_3(\mathbb{F}_q)$ in the standard block diagonal way. The maps on cohomology 
\[
\op{H}^\bullet(\op{GL}_2(\mathbb{F}_q)\times \mathbb{F}_q^\times,\mathbb{F}_\ell)
\to\op{H}^\bullet((\mathbb{F}_q^\times)^2,\mathbb{F}_\ell)
\] 
induced by the above inclusions can then be written as 
\[
\mathbb{F}_\ell[\op{c}_1,\op{c}_2,X_1]\langle \op{e}_1,\op{e}_2,Y_1\rangle\to 
\mathbb{F}_\ell[X_1',X_2']\langle Y_1',Y_2'\rangle:
\]
\[
\left\{\begin{array}{l}
X_1\mapsto X_2'\\ Y_1\mapsto Y_2'\\
\op{c}_1\mapsto 2X_1'\\ \op{c}_2\mapsto (X_1')^2\\
\op{e}_1\mapsto 2Y_1'\\ \op{e}_2\mapsto 2X_1'Y_1'
\end{array}\right\},\quad
\left\{\begin{array}{l}
X_1\mapsto X_1'\\ Y_1\mapsto Y_1'\\
\op{c}_1\mapsto X_1'+X_2'\\ \op{c}_2\mapsto X_1'X_2'\\
\op{e}_1\mapsto Y_1'+Y_2'\\ \op{e}_2\mapsto X_1'Y_2'+X_2'Y_1'
\end{array}\right\}
\]
with $\deg \op{c}_i=2i$, $\deg \op{e}_i=2i-1$, $\deg X_i^{(')}=2$, $\deg Y_i^{(')}=1$.
\item For $\sigma$ a  vertex of subtype (IIc), the vertex stabilizer is $\op{GL}_3(\mathbb{F}_q)$, and there is a single adjacent edge $\tau$ stabilized by  $(\mathbb{F}_q^\times)^{\times 2}$. The inclusion of the edge group into the vertex group is $(\mathbb{F}_q^\times)^{\times 2}\to \op{GL}_3(\mathbb{F}_q):(x,y)\mapsto \op{diag}(x,x,y)$. The induced map on cohomology 
\[
\op{H}^\bullet(\op{GL}_3(\mathbb{F}_q);\mathbb{F}_\ell)\to
\op{H}^\bullet((\mathbb{F}^\times_q)^2;\mathbb{F}_\ell)
\]
can be written, using the identifications of Section~\ref{sec:reccohom}, as follows:
\[
\mathbb{F}_\ell[\op{c}_1,\op{c}_2,\op{c}_3]\langle\op{e}_1,\op{e}_2,\op{e}_3\rangle\to
\mathbb{F}_\ell[X_1,X_2]\langle Y_1,Y_2\rangle:
\]
\[
\left\{\begin{array}{l}
\op{c}_1\mapsto 2X_1+X_2\\
\op{c}_2\mapsto X_1^2+2X_1X_2\\
\op{c}_3\mapsto X_1^2X_2\\
\op{e}_1\mapsto 2Y_1+Y_2\\
\op{e}_2\mapsto 2X_1Y_1+2(X_1Y_2+X_2Y_1)\\
\op{e}_3\mapsto X_1^2Y_2+2X_1X_2Y_1.
\end{array}\right\}.
\]
\item 
For $\sigma$ a vertex of subtype (IId), the vertex stabilizer is $\mathbb{F}_{q^2}^\times\times \mathbb{F}_q^\times$. There is a single edge stabilized by $(\mathbb{F}_q^\times)^{\times 2}$ adjacent to $\sigma$, and the inclusion of the edge group is the natural one $(\mathbb{F}_q^\times)^{\times 2}\to \mathbb{F}_{q^2}^\times\times \mathbb{F}_q^\times$. The induced map on cohomology with $\mathbb{F}_\ell$-coefficients is the identity.
\item
Vertices of type (IIe) are isolated in the parabolic graph. Their stabilizer is isomorphic to $\mathbb{F}_{q^3}^\times$. The inclusion $\mathbb{F}_q^\times\hookrightarrow\mathbb{F}_{q^3}^\times$ induces the identity on cohomology with $\mathbb{F}_\ell$-coefficients, and therefore these vertices are not visible in the $\mathbb{F}_\ell$-cohomology of $\op{GL}_3(k[E])$. 
\end{enumerate}

Finally, it remains to describe the restriction maps induced from the inclusion of the center $\mathcal{Z}(\op{GL}_3(\mathbb{F}_q[E]))\cong\mathbb{F}_q^\times$. For the edge stabilizers of the parabolic subgraph, the inclusion of the center is given by the homotheties, i.e., it is
\[
\mathbb{F}_q^\times\mapsto (\mathbb{F}_q^\times)^2:x\mapsto \op{diag}(x,x,x).
\]
In particular, the corresponding restriction map on cohomology is
\[
\mathbb{F}_\ell[X_1,X_2]\langle Y_1,Y_2\rangle\cong \op{H}^\bullet((\mathbb{F}_q^\times)^2;\mathbb{F}_\ell)\to\op{H}^\bullet(\mathbb{F}_q^\times;\mathbb{F}_\ell)\cong \mathbb{F}_\ell[Z]\langle W\rangle
\]
\[
X_1,X_2\mapsto Z; \quad Y_1,Y_2\mapsto W.
\]
The restrictions corresponding to inclusions of the center into vertex stabilizers of type (II) above can be obtained from the above case distinction by composition with further restriction to $\mathbb{F}_\ell[Z]\langle W\rangle$. For types (IIa) and (IIb), we can check that the compositions of the different restriction maps for the different edges are equalized by the restriction to the center.

The above case distinction describes completely the restriction map for inclusions of stabilizer groups, and therefore describes up to sign the differential $\op{d}_1:E^{i,j}_1\to E^{i+1,j}_1$. 

\subsection{Simplification of the \texorpdfstring{$E_1$}{E1}-page}
\label{prop:simplee1}

At this point, we have described the $E_1$-page of the spectral sequence. Since the quotient $\op{GL}_3(k[E])\backslash\mathfrak{B}(E/k,3)$ has infinitely many cells, the relevant entries in the $E_1$-page are infinite direct products of cohomology groups of cell stabilizers. Using Theorems~\ref{thm:quotient} and \ref{thm:parabolic}, there is an $\mathcal{H}_G^\ast$-reduction of the quotient $\op{GL}_3(k[E])\backslash\mathfrak{B}(E/k,3)$ to a finite subcomplex $\mathfrak{X}$ having the homotopy type of the suspension of the flag complex for $k^3$ which contains the (finite) parabolic graph. The subcomplex is essentially the union of the links of the two stable bundles of rank three on $\overline{E}$ with determinant concentrated at $\{O\}$.

Using Proposition~\ref{prop:reduction}, we find that there is a subcomplex $\tilde{E}^{s,t}_1\subset E^{s,t}_1$ of finitely generated $\mathbb{F}_\ell[\op{c}_1,\op{c}_2,\op{c}_3]$-modules such that the inclusion $\tilde{E}^{s,t}_1\subset E^{s,t}_1$ is a quasi-isomorphism. Henceforth, we will only compute in this finitely generated subcomplex, omitting notational distinction between $\tilde{E}^{s,t}_1$ and $E^{s,t}_1$. Two important things are noteworthy: first, because of the finiteness, the direct products of  stabilizer cohomology groups in the reduced $E_1$-page are now direct sums. The second thing to note, which follows from the explicit description in Theorem~\ref{thm:parabolic} is that if an elementary abelian $\ell$-group or rank $m$ stabilizes an $n$-cell of the $\mathcal{H}_G^\ast$-reduction $\mathfrak{X}$, then $n\leq 3-m$. 

\section{A detection filtration on the \texorpdfstring{$E_1$}{E1}-page}
\label{sec:detection}

After having obtained a full description of the $E_1$-page with differentials and Chern-class module structure, the next goal is to work out the $E_2$-page of the spectral sequence. Since the $E_1$-page has quite a number of non-trivial entries and differentials relating them, this requires a bit more work. We first outline the main steps of the computation. 

First of all, we will define a filtration on the $E_1$-page, called \emph{detection filtration}, which essentially is given by the kernels of restriction maps to stabilizers of adjacent cells. In a sense, this is a rank refinement of the idea of essential classes in the cohomology of elementary abelian groups as well as the idea of torsion subcomplexes used in the study of rank one groups, cf. \cite{rahm}. There are only three steps of the filtration, given by $\ell$-ranks of stabilizers between $1$ and $3$. On the subquotients of the detection filtration, the $\op{d}_1$-differential will be linear over the respective Chern-class ring, and we will be able to precisely compute the cohomology of the $\op{d}_1$-differential on the subquotients in terms of the structure of the quotient $\op{GL}_3(k[E])\backslash\mathfrak{B}(E/k,3)$. This provides a spectral sequence converging to the $E_2$-page of the spectral sequence. 

The goal of the present section is to describe explicitly the detection filtration on the cohomology of each type of stabilizer subgroup, as well as explicit generators (as appropriate Chern-class modules) and Hilbert--Poincar{\'e} series for the pieces of the detection filtration.

\subsection{Definition of the detection filtration} 
The detection filtration on the $E_1$-page is defined by defining the filtration on the cohomology $\op{H}^\bullet(G_\sigma;\mathbb{F}_\ell)$ for each type of stabilizer group $G_\sigma$. 

\begin{definition}[Detection filtration]
\label{def:e1filtration}
We define the \emph{detection filtration $\op{D}^\bullet E_1^{s,t}$} as follows: 
\begin{itemize}
\item Let $\sigma$ be a simplex with stabilizer group $G_\sigma$. A class in $\op{H}^\bullet(G_\sigma;\mathbb{F}_\ell)$ is in the $i$-th step $\op{D}^i$ of the detection filtration if it is detected on an elementary abelian $\ell$-subgroup of $G_\sigma$ of rank $\geq i+1$ and restricts trivially to elementary abelian $\ell$-subgroups of rank $i$ in all stabilizer groups $G_\tau$ with $\tau\supset \sigma$.
\item Then $\op{D}^iE^{s,t}_1$ is defined to be the direct product of the $\op{D}^i\op{H}^t(G_\sigma;\mathbb{F}_\ell)$ with $\sigma$ an $s$-simplex of the orbit space. 
\end{itemize}
\end{definition}

\begin{remark}
Note that the direct product is in fact a finite direct sum in the reduced spectral sequence, cf. \ref{prop:simplee1}.
\end{remark}

\begin{remark}
The definition of the detection filtration above can be made for arbitrary proper actions and is made with a view towards possible computations in higher ranks. In our specific case, $\op{D}^3E_1^{s,t}=0$ because $\op{GL}_3(k[E])$ contains elementary abelian $\ell$-groups of rank at most $3$. Also, the condition defining $\op{D}^0$ is vacuously true, giving $\op{D}^0E_1^{s,t}=E_1^{s,t}$. The step $\op{D}^1E_1^{s,t}$ is given by classes in stabilizers of the parabolic subgraph whose restriction to the center $\mathcal{Z}(\op{GL}_3(\mathbb{F}_q[E]))$ is trivial. Finally, the step $\op{D}^2E_1^{s,t}$ is given by classes in stabilizers of type (IIa-c) vertices whose restriction to the adjacent edges in the parabolic graph are trivial.
\end{remark}

\begin{lemma}
\label{lem:e1props}
\begin{enumerate}
\item The $\op{d}_1$-differential is compatible with this filtration in the sense that $d_1\op{D}^iE^{s,t}_1\subseteq \op{D}^iE^{s+1,t}_1$.
\item $\op{D}^iE_1^{s,t}\subseteq E_1^{\leq (2-i),t}$. 
\item $\op{D}^2E_1^{s,t}$ consists exactly of sums of cohomology classes $\gamma\in \op{H}^t(G_\sigma;\mathbb{F}_\ell)$ which are in the kernel of $\op{d}_1^{s,t}|_{\op{H}^t(G_\sigma;\mathbb{F}_\ell)}$, where $G_\sigma$ is a type (IIa-c) stabilizer.
\item The filtration steps $\op{D}^iE_1^{s,t}$ are $\mathbb{F}_\ell[\op{c}_1,\op{c}_2,\op{c}_3]$-submodules of $E_1^{s,t}$. \end{enumerate}
\end{lemma}

\begin{proof}
(1) follows since both the differential as well as the detection filtration are defined in terms of restriction maps. 

(2) Step $\op{D}^2$ is contained in $E^{\leq 0,t}_1$ since only $0$-cells can possibly be contained in simplices stabilized by rank 2 elementary abelian groups. Step $\op{D}^1$ is contained in $E^{\leq 1,t}_1$ since $E^{2,t}_1$ contains only cells stabilized by the center. 

(3) is essentially the definition, plus the structural statements from Section~\ref{sec:inclusions} 

(4) The cohomology of an elementary abelian $\ell$-group of rank $n$ is a free module over the Chern-class subring generated by $\op{c}_1,\dots,\op{c}_n$; the Chern classes $\op{c}_m$ with $m>n$ act trivially. This means that detectability of a class on rank $\geq i+1$ abelian subgroups is not destroyed by multiplication with elements from the Chern-class ring unless the element is completely annihilated. The other part of the definition of $\op{D}^i$ related to the trivial restriction is now essentially a statement about the restriction map associated to $E\hookrightarrow E'$ where $E$ and $E'$ are elementary abelian $\ell$-groups of rank $i$ and $i+1$, respectively. Since restriction commutes with cup-product, $\op{res}(\op{c}_i\cup x)=\op{res}(\op{c}_i)\cup\op{res}(x)$ and therefore the class $\op{c}_i\cup x$ has trivial restriction whenever $x$ has trivial restriction. This establishes the claim.
\end{proof}

Next, to determine the filtration $\op{D}^\bullet E_1^{s,t}$ we will compute the detection filtration for each type of stabilizer. Note that only stabilizers $G_\sigma$ of types (IIa-c) will have a non-trivial $\op{D}^2$, which by Lemma~\ref{lem:e1props} is given as kernel of the restriction maps 
\[
\op{d}_1^{0,j}:\op{H}^j(G_\sigma;\mathbb{F}_\ell)\to\bigoplus_\tau\op{H}^j(G_\tau;\mathbb{F}_\ell),
\]
where $\tau$ in the direct sum runs over the edges containing the vertex $\sigma$ and having a rank 2 elementary abelian $\ell$-subgroup. For all the other stabilizers, it suffices to compute the kernel and image of the restriction to the center. 

\subsection{Detection filtration on rank 2 stabilizers}

In this section, we want to compute the step $\op{D}^1$ of the detection filtration. 

We first note that the restriction to the center is always surjective for the cases we consider. This will also be helpful for the computation of the detection filtration for rank 3 stabilizers in subsequent subsections.

\begin{lemma}
\label{lem:center}
Let $k=\mathbb{F}_q$ be a finite field, and let $\ell\mid q-1$ be a prime different from $2$ and $3$. Let $\sigma$ be any simplex of $\mathfrak{B}(E/k,3)$ and let $G_\sigma$ be the stabilizer of $\sigma$ in $\op{GL}_3(k[E])$. Then the restriction associated to the natural inclusion $\mathbb{F}_q^\times\cong\mathcal{Z}(\op{GL}_3(k[E]))\hookrightarrow G_\sigma$ is surjective. 
\end{lemma}

\begin{proof}
Since the restriction is an algebra map, it suffices to show that the algebra generators of $\op{H}^\bullet(\mathbb{F}_q^\times;\mathbb{F}_\ell)\cong \mathbb{F}_\ell[X]\langle A\rangle$ are in the image (this description of the algebra is where we use the assumption $\ell\neq 2$). We can prove the statement using a case distinction, going through the list provided in Section~\ref{sec:inclusions}. The surjectivity is easy to see in all the cases where we have a factor $k^\times$ in the stabilizer. The cases (IIe) is not relevant for us since $\ell|q-1$, so the only case left is case (IIc), the restriction from $\op{GL}_3(k)$. In this case, $\op{c}_1$ and $\op{e}_1$ map to $3X$ and $3A$, respectively, and by our assumption $\ell\neq 3$ we have surjectivity.
\end{proof}

The stabilizers whose maximal elementary abelian $\ell$-subgroups are of rank 2 are always of the form $k^\times\times k^\times$, with the inclusion of the center $k^\times$ given by the diagonal map $x\mapsto(x,x)$. The next proposition determines the kernel of the associated restriction map.

\begin{proposition}
\label{prop:detectionrk2}
Let $k=\mathbb{F}_q$ be a finite field, and let $\ell\mid q-1$ be a prime different from $2$ and $3$. The kernel of the restriction morphism
\[
\mathbb{F}_\ell[X,Y]\langle A,B\rangle\cong \op{H}^\bullet(\mathbb{F}_q^\times\times\mathbb{F}_q^\times;\mathbb{F}_\ell)\to
\op{H}^\bullet(\mathbb{F}_q^\times;\mathbb{F}_\ell)\cong 
\mathbb{F}_\ell[Z]\langle C\rangle
\]
given by $X,Y\mapsto Z$ and $A,B\mapsto C$ is a free graded module of rank $4$ over the ring $\mathbb{F}_\ell[X,Y]$ generated by the elementary alternating polynomials in $\mathbb{F}_\ell[X,Y]\langle A,B\rangle$, i.e., 
\[
A-B,\, AB,\, X-Y,\, (X-Y)(A+B).
\]
Consequently, the kernel of the restriction morphism is a free module of rank $8$ over the second Chern-class ring $\mathbb{F}_\ell[\op{c}_1,\op{c}_2]$, generated by the elementary alternating polynomials and 
\[
(X-Y)(A-B),\, (X-Y)AB,\, (X-Y)^2,\, (X-Y)^2(A+B).
\]
The Hilbert--Poincar{\'e} series of the kernel is 
\[
\frac{(T+2T^2+T^3)(1+T^2)}{(1-T^2)(1-T^4)}.
\]
\end{proposition}

\begin{proof}
We get surjectivity of the restriction morphism from Lemma~\ref{lem:center}. The short exact sequence connecting the kernel with source and target of the restriction map implies that the Hilbert--Poincar{\'e} series of the kernel is
\[
\frac{(1+T)^2}{(1-T^2)^2}-\frac{(1+T)}{(1-T^2)}=
\frac{(1+T)^2-(1+T)(1-T^2)}{(1-T^2)^2}=
\frac{(T+2T^2+T^3)(1+T^2)}{(1-T^2)(1-T^4)}. 
\]
The freeness of the $\mathbb{F}_\ell[X,Y]$-submodule generated by the elementary alternating polynomials follows from the fact that $1$, $(A-B)$, $(A+B)$ and $AB$ are $\mathbb{F}_\ell[X,Y]$-linearly independent. The submodule generated by elementary alternating polynomials is also obviously contained in the kernel, giving a lower bound for the kernel. Since the Hilbert--Poincar{\'e} series of the kernel and its submodule agree, they must be equal. The statement for the generators of the Chern class ring follows since $\mathbb{F}_\ell[X,Y]$ is a free module over $\mathbb{F}_\ell[\op{c}_1,\op{c}_2]$ with a possible choice of generators given by $1$ and $X-Y$.
\end{proof}

\subsection{Detection filtration for stabilizers of type (IIa)}

The first case to consider are stabilizer groups of the form $(\mathbb{F}_q^\times)^3$. There are three adjacent edges stabilized by $(\mathbb{F}_q^\times)^2$. The restriction maps have been determined in Section~\ref{sec:restrictions}. 

The following result determines the kernel of the restriction map
\[
\mathbb{F}_\ell[X,Y,Z]\langle A,B,C\rangle\cong \op{H}^\bullet((\mathbb{F}_q^\times)^3;\mathbb{F}_\ell)\to 
\bigoplus_i\op{H}^\bullet((\mathbb{F}_q^\times)^2;\mathbb{F}_\ell)\cong
\bigoplus_i\mathbb{F}_\ell[U_i,V_i]\langle D_i,E_i\rangle
\]
\[
X\mapsto (U_1,U_2,V_3),\, Y\mapsto (U_1,V_2,U_3),\, Z\mapsto (V_1,U_2,U_3)
\]
\[
A\mapsto (D_1,D_2,E_3),\, B\mapsto (D_1,E_2,D_3),\, C\mapsto (E_1,D_2,D_3).
\]

\begin{proposition}
\label{prop:detectiona}
Let $k=\mathbb{F}_q$ be a finite field and let $\ell\mid q-1$ be a prime different from $2$ and $3$. The following three statements describe the detection filtration on the cohomology algebra
\[
\op{H}^\bullet((\mathbb{F}_q^\times)^3;\mathbb{F}_\ell)\cong\mathbb{F}_\ell[X,Y,Z]\langle A,B,C\rangle.
\]

\begin{enumerate}
\item The kernel of the above sum of restriction maps, which equals the filtration step $\op{D}^2$, is the $\mathbb{F}_\ell[X,Y,Z]$-submodule of $\mathbb{F}_\ell[X,Y,Z]\langle A,B,C\rangle$ generated by the alternating polynomials of Proposition~\ref{prop:eltalt3}. In particular, it is a free graded $\mathbb{F}_\ell[X,Y,Z]$-module of rank $8$ with Hilbert--Poincar{\'e} series 
\[
\frac{T^2(T+1)^3(T^2-T+1)}{(1-T^2)^3}
\]
Consequently, it is also a free graded module over the Chern-class subring $\mathbb{F}_\ell[\op{c}_1,\op{c}_2,\op{c}_3]$, of rank $48$. 

\item The Hilbert--Poincar{\'e} series of the image is given by 
\[
\frac{(1+T)^2(1+T+T^3)}{(1-T^2)^2}= \frac{(1+T)^2(1+T+T^3)(1+T^2)}{(1-T^2)(1-T^4)},
\]
where the numerator is $1+3T+4T^2+5T^3+5T^4+3T^5+2T^6+T^7$.
\item 
The subquotient $\op{D}^1/\op{D}^2$ of the cohomology ring, consisting of classes essentially of rank two, is a free module over $\mathbb{F}_\ell[U,V]$, generated by the residue classes of the elements 
\begin{itemize}
\item in degree 1: $A-B$ and $B-C$,
\item in degree 2: $X-Y$, $Y-Z$, $AB$ and $BC$,
\item in degree 3: $(X-Y)A$, $(Y-Z)A$, and 
\[
(2X-Y-Z)A+(-X+2Y-Z)B+(-X-Y+2Z)C
\] (equal to $2\op{c}_1\op{e}_1-3\op{e}_2$ in terms of symmetric polynomials)
\item in degree 4: 
\[
AB(X+Y+2Z)+AC(X+2Y+Z)+BC(2X+Y+Z)
\]
 (equal to $\op{e}_1\op{e}_2$ in terms of symmetric polynomials) and \[X^2+Y^2+Z^2-XY-XZ-YZ
\] (equal to $\op{c}_1^2-3\op{c}_2$)
\item in degree 5: 
\[(Y^2+Z^2-XY-XZ)A+(X^2+Z^2-XY-YZ)B+(X^2+Y^2-XZ-YZ)C
\]
(equal to $\op{c}_1\op{e}_2-2\op{c}_2\op{e}_1$ in terms of symmetric polynomials)
\end{itemize}
In particular, the subquotient is free over the Chern-class ring $\mathbb{F}_\ell[\op{c}_1,\op{c}_2]$, of rank 24, with Hilbert--Poincar{\'e} series 
\[
\frac{T(2+T^2)(1+T)^2}{(1-T^2)^2}=\frac{T(1+T^2)(2+T^2)(1+T)^2}{(1-T^2)(1-T^4)}
\]
\end{enumerate}
\end{proposition}

\begin{proof}
We already know by Proposition~\ref{prop:eltalt3} that the $\mathbb{F}_\ell[X,Y,Z]$-submodule of $\mathbb{F}_\ell[X,Y,Z]\langle A,B,C\rangle$ generated by the alternating polynomials is free, and we know its Hilbert--Poincar{\'e} series. Moreover, the alternating polynomials of Proposition~\ref{prop:eltalt3} are obviously in the kernel of the sum of restriction maps. In particular, we have a free submodule of the kernel with known Hilbert--Poincar{\'e} series. This gives a lower bound for the Hilbert--Poincar{\'e} series, i.e., the coefficients for the Hilbert--Poincar{\'e} series of the kernel are at least as big as the coefficients of the series claimed in (1). 

It can be checked that the Hilbert--Poincar{\'e} series in (1) and (2) sum to the Hilbert--Poincar{\'e} series of $\mathbb{F}_\ell[X,Y,Z]\langle A,B,C\rangle$. Hence we have shown that (2) implies (1), i.e., if we can prove the statement about the Hilbert--Poincar{\'e} series of the image, we obtain the claim for the kernel. Actually, to prove the claim it suffices to show that the Hilbert--Poincar{\'e} series claimed in (2) is a lower bound for the Hilbert--Pincar{\'e} series, again in the sense of coefficient comparison. Essentially, having lower bounds for image and kernel whose sum is the Hilbert--Poincar{\'e} series for the full cohomology algebra implies that the lower bounds are in fact sharp. 

To prove that the series claimed in (2) is a lower bound for the Hilbert--Poincar{\'e} series of the image of the restriction map, we note that the restriction map from $\op{H}^\bullet((\mathbb{F}_q^\times)^3;\mathbb{F}_\ell)$ to the center is surjective. The sum of the Hilbert--Poincar{\'e} series for the cohomology of the center with the one claimed in (3) is the one claimed in (2). Therefore, showing that the series claimed in (3) is a lower bound for the Hilbert--Poincar{\'e} series of $\op{D}^1/\op{D}^2$ will prove (2) and complete the proof. 

For (3), we now show that the given generators are $\mathbb{F}_\ell[U,V]$-linearly independent in $\bigoplus_3\op{H}^\bullet$. It is easy to see that their restriction to the center is trivial, hence this will show that the Hilbert--Poincar{\'e} series claimed in 3 is a lower bound. 

To prove freeness, we note again that the exterior parts are $\mathbb{F}_\ell[U,V]$-linearly independent. There is hence the following case distinction: 

The polynomials without exterior components are $X-Y$, $Y-Z$ and $X^2+Y^2+Z^2-XY-XZ-YZ$. Their restrictions to the three components are 
\[
(0,U-V,V-U),\quad (U-V,V-U,0),\quad ((U-V)^2,(U-V)^2,(U-V)^2).
\]
Since we are in characteristic $\neq 2$, this shows linear independence. 

The polynomials with degree 2 exterior part are $AB$, $BC$ and $AB(X+Y+2Z)+AC(X+2Y+Z)+BC(2X+Y+Z)$ which are obviously linearly independent. 

The final part is to determine linear independence of the polynomials with exterior degree 1. Here we can use that $A-B$ and $B-C$ are linearly independent: the relations in degrees 3 and 5 are
\[
(2X-Y-Z)(A-B)+(X+Y-2Z)(B-C), \textrm{ and}
\]
\begin{eqnarray*} 
&&(Y^2+Z^2-XY-XZ)(A-B)\\&+&(Y^2+2Z^2+X^2-XZ-2XY-YZ)(B-C)\\&+&
(2X^2+2Y^2+2Z^2-2XZ-2XY-2YZ)C
\end{eqnarray*}
The realization of the $A$-component of the  degree 3 element is $(U-V,U-V,2(V-U))$, that of the degree 5 element $(U(V-U),V(V-U),2U(U-V))$. 
The realization of the $B$-component of the degree 3 element is 
$(2(U-V),V-U,V-U)$, that of the degree 5 element is $(2V(V-U),2U^2+V^2-3UV,2U^2+V^2-3UV)$. Then it is clear that these are linearly independent. This is even true after including $(A-B)$, $(B-C)$, $(X-Y)A$ and $(Y-Z)A$ because the $\mathbb{F}_\ell[X,Y,Z]$-linear dependence of the degree 3 generator does not descend to a diagonal $\mathbb{F}_\ell[U,V]$-linear dependence. 
\end{proof}

\subsection{Detection filtration for stabilizers of type (IIb)}
Next we consider the case of stabilizer groups of the form $\op{GL}_2(\mathbb{F}_q)\times\mathbb{F}_q^\times$. There are two adjacent edges stabilized by $(\mathbb{F}_q^\times)^2$. The restriction maps were determined in Section~\ref{sec:restrictions}. The following result determines the kernel of the restriction map 
\begin{eqnarray*}
\mathbb{F}_\ell[\op{c}_1,\op{c}_2,Z]\langle \op{e}_1,\op{e}_2,C\rangle&\cong&
\op{H}^\bullet(\op{GL}_2(\mathbb{F}_q)\times\mathbb{F}_q^\times;\mathbb{F}_\ell)\\
&\longrightarrow & \bigoplus_i\op{H}^\bullet((\mathbb{F}_q^\times)^2;\mathbb{F}_\ell)\\
&\cong&\bigoplus_i\mathbb{F}_\ell[U_i,V_i]\langle D_i,E_i\rangle:
\end{eqnarray*}
\[
\op{c}_1\mapsto (2U_1,U_2+V_2), \op{c}_2\mapsto (U_1^2,U_2V_2), Z\mapsto (V_1,U_2)
\]
\[
\op{e}_1\mapsto (2D_1,D_2+E_2), \op{e}_2\mapsto (2U_1D_1,U_2E_2+V_2,D_2), C\mapsto (E_1,D_2)
\]

\begin{proposition}
\label{prop:detectionb}
Let $k=\mathbb{F}_q$ be a finite field and let $\ell\mid q-1$ be a prime different from $2$ and $3$. The following three statements describe the detection filtration on the cohomology algebra
\[
\op{H}^\bullet(\op{GL}_2(\mathbb{F}_q)\times\mathbb{F}_q^\times;\mathbb{F}_\ell)
\cong \mathbb{F}_\ell[\op{c}_1,\op{c}_2,Z]\langle \op{e}_1,\op{e}_2,C\rangle.
\]
\begin{enumerate}
\item
Consider the standard embedding 
\[
\mathbb{F}_\ell[\op{c}_1,\op{c}_2,Z]\langle \op{e}_1,\op{e}_2,C\rangle\hookrightarrow \mathbb{F}_\ell[X,Y,Z]\langle A,B,C\rangle,
\]
where $\op{c}_1,\op{c}_2,\op{e}_1,\op{e}_2$ are the elementary symmetric polynomials in $X,Y,A,B$. With this notation, the kernel of the restriction map is the $\mathbb{F}_\ell[\op{c}_1,\op{c}_2,Z]$-submodule generated by elements of the form $(X-Y)\cdot F$ with $F$ one of the alternating polynomial of Proposition~\ref{prop:eltalt3}. It is a free module over the ring $\mathbb{F}_\ell[\op{c}_1,\op{c}_2,Z]$, with Hilbert--Poincar{\'e} series given  by 
\[
\frac{T^4+2T^5+T^6+T^7+2T^8+T^9}{(1-T^2)^2(1-T^4)}
\]
Consequently, it is a free module over the Chern class ring $\mathbb{F}_\ell[\op{c}_1,\op{c}_2,\op{c}_3]$, of rank $24$.
\item The Hilbert--Poincar{\'e} series of the image is given by 
\[
\frac{(1+T)^3(1-T+T^2)}{(1-T^2)^2}=
\frac{(1+T)^3(1-T+T^2)(1+T^2)}{(1-T^2)(1-T^4)}
\]
\item 
The rank two subquotient of  
\[
\op{H}^\bullet(\op{GL}_2(\mathbb{F}_q)\times\mathbb{F}_q^\times;\mathbb{F}_\ell)\cong 
\mathbb{F}_\ell[\op{c}_1,\op{c}_2,Z]\langle \op{e}_1,\op{e}_2,C\rangle
\]
is a free module over $\mathbb{F}_\ell[U,V]$ generated by residue classes of the elements 
\begin{itemize}
\item in degree 1: $A+B-2C$,
\item in degree 2: $X+Y-2Z$, $(A+B)C$, 
\item in degree 3: $XB+XC+YA+YC$, 
 \[
(2X-Y-Z)A+(-X+2Y-Z)B+(-X-Y+2Z)C
\] (equal to $2\op{c}_1\op{e}_1-3\op{e}_2$ in terms of symmetric polynomials \emph{in three variables!})
\item in degree 4: 
\[
AB(X+Y+2Z)+AC(X+2Y+Z)+BC(2X+Y+Z)
\]
 (equal to $\op{e}_1\op{e}_2$ in terms of symmetric polynomials of three variables) and \[X^2+Y^2+Z^2-XY-XZ-YZ
\] (equal to $\op{c}_1^2-3\op{c}_2$ in terms of symmetric polynomials in three variables)

\item in degree 5: $(Y^2+Z^2-XY-XZ)A+(X^2+Z^2-XY-YZ)B+(X^2+Y^2-XZ-YZ)C$ (equal to $\op{c}_1\op{e}_2-2\op{c}_2\op{e}_1$ in terms of symmetric polynomials in three variables!).
\end{itemize}
Its Hilbert--Poincar{\'e} series is
\[
\frac{T(1+T^2)(1+T)^2}{(1-T^2)^2}=
\frac{T(1+T)^2(1+T^2)^2}{(1-T^2)(1-T^4)}
\]
\end{enumerate}
\end{proposition}

\begin{proof}
As before, we know from Proposition~\ref{prop:eltalt3} that the $\mathbb{F}_\ell[\op{c}_1,\op{c}_2,Z]$-submodule generated by the product of $(X-Y)$ with alternating polynomials is free. The Hilbert--Poincar{\'e} series is the one claimed in (1). Moreover, the elements $(X-Y)\cdot F$ are obviously in the kernel of the restriction map. In particular, we have a free submodule of the kernel with known Hilbert--Poincar{\'e} series giving us a lower bound in the sense of pointwise coefficient comparison.

It can be checked that the Hilbert--Poincar{\'e} series in (1) and (2) sum to the Hilbert--Poincar{\'e} series of $\mathbb{F}_\ell[\op{c}_1,\op{c}_2,Z]\langle \op{e}_1,\op{e}_2,C\rangle$ showing that (2) implies (1). Actually, it suffices to show that the Hilbert--Poincar{\'e} series claimed in (2) is a lower bound for the Hilbert--Poincar{\'e} series of the image. Additivity plus the lower-bound assertion for the series in (1) will imply that neither kernel nor image can be bigger than the established lower bounds. 

The restriction map from the algebra $\mathbb{F}_\ell[\op{c}_1,\op{c}_2,Z]\langle\op{e}_1,\op{e}_2,C\rangle$ to the center is surjective. The sum of the Hilbert--Poincar{\'e} series for the cohomology of the center with the one claimed in (3) is the one claimed in (2). Therefore, proving (3) will establish a lower bound for the Hilbert--Poincar{\'e} series of the image, as required to complete the proof. 

The freeness is proved as in (IIa), the generators are essentially the same. 
\end{proof}


\subsection{Detection filtration for stabilizers of type (IIc)} 
Finally, we consider the case of stabilizer groups of the form $\op{GL}_3(\mathbb{F}_q)$. There is only one adjacent edge stabilized by $(\mathbb{F}_q^\times)^2$. The restriction map was determined in Section~\ref{sec:restrictions}. The following result determines the kernel of the restriction map
\[
\mathbb{F}_\ell[\op{c}_1,\op{c}_2,\op{c}_3]
\langle\op{e}_1,\op{e}_2,\op{e}_3\rangle \cong \op{H}^\bullet(\op{GL}_3(\mathbb{F}_q);\mathbb{F}_\ell)\to \op{H}^\bullet((\mathbb{F}_q^\times)^2;\mathbb{F}_\ell)\cong \mathbb{F}_\ell[U,V]\langle D,E\rangle:
\]
\[
\op{c}_1\mapsto 2U+V,\,\op{c}_2\mapsto U^2+2UV,\,\op{c}_3\mapsto U^2V
\]
\[
\op{e}_1\mapsto 2D+E,\,\op{e}_2\mapsto 2(UD+UE+VD),\,\op{e}_3\mapsto U^2E+2UVD. 
\]

\begin{proposition}
\label{prop:detectionc}
Let $k=\mathbb{F}_q$ be a finite field and let $\ell\mid q-1$ be a prime different from $2$ and $3$. The following three statements describe the detection filtration on the cohomology algebra.
\begin{enumerate}
\item 
Consider the standard embedding 
\[
\mathbb{F}_\ell[\op{c}_1,\op{c}_2,\op{c}_3]\langle \op{e}_1,\op{e}_2,\op{e}_3\rangle\hookrightarrow \mathbb{F}_\ell[X,Y,Z]\langle A,B,C\rangle,
\]
where $\op{c}_i,\op{e}_i$ are the elementary symmetric polynomials. With this notation, the kernel of the restriction map is the $\mathbb{F}_\ell[\op{c}_1,\op{c}_2,\op{c}_3]$-submodule generated by elements of the form $(X-Y)(X-Z)(Y-Z)\cdot F$ with $F$ one of the alternating polynomials in Proposition~\ref{prop:eltalt3}. It is a free module over the ring $\mathbb{F}_\ell[\op{c}_1,\op{c}_2,\op{c}_3]$, of rank $8$, with Hilbert--Poincar{\'e} series given  by 
\[
\frac{T^8+2T^9+T^{10}+T^{11}+2T^{12}+T^{13}}{(1-T^2)(1-T^4)(1-T^6)}.
\]

\item 
The Hilbert--Poincar{\'e} series of the image is
\[
\frac{(1+T)^2(1-T+T^3)}{(1-T^2)^2}
=\frac{(1+T)^2(1-T+T^3)(1+T^2)}{(1-T^2)(1-T^4)}.
\]
\item 
The rank two subquotient of $\op{H}^\bullet(\op{GL}_3(\mathbb{F}_q);\mathbb{F}_\ell)$ is a free module over the ring $\mathbb{F}_\ell[U,V]$ generated by residue classes of the elements 
\[
2\op{c}_1\op{e}_1-3\op{e}_2, \quad \op{e}_1\op{e}_2, \quad \op{c}_1^2-3\op{c}_2, \quad \op{c}_1\op{e}_2-2\op{c}_2\op{e}_1.
\]
Its Hilbert--Poincar{\'e} series is
\[
\frac{T^3(1+T)^2}{(1-T^2)^2}=\frac{T^3(1+T)^2(1+T^2)}{(1-T^2)(1-T^4)}.
\]
\end{enumerate}
\end{proposition}

\begin{proof}
As before, we know from Proposition~\ref{prop:eltalt3} that the $\mathbb{F}_\ell[\op{c}_1,\op{c}_2,\op{c}_3]$-submodule generated by the product of $(X-Y)(X-Z)(Y-Z)$ with alternating polynomials is free. The Hilbert--Poincar{\'e} series is the one claimed in (1). Moreover, the elements $(X-Y)(X-Z)(Y-Z)\cdot F$ are obviously in the kernel of the restriction map. In particular, we have a free submodule of the kernel with known Hilbert--Poincar{\'e} series giving us a lower bound in the sense of pointwise coefficient comparison.

It can be checked that the Hilbert--Poincar{\'e} series in (1) and (2) sum to the Hilbert--Poincar{\'e} series of $\mathbb{F}_\ell[\op{c}_1,\op{c}_2,\op{c}_3]\langle \op{e}_1,\op{e}_2,\op{e}_3\rangle$ showing that (2) implies (1). Actually, it suffices to show that the Hilbert--Poincar{\'e} series claimed in (2) is a lower bound for the Hilbert--Poincar{\'e} series of the image. Additivity plus the lower-bound assertion for the series in (1) will imply that neither kernel nor image can be bigger than the established lower bounds. 

The restriction map from the algebra $\mathbb{F}_\ell[U,V]\langle D,E\rangle$ to the center is surjective. The sum of the Hilbert--Poincar{\'e} series for the cohomology of the center with the one claimed in (3) is the one claimed in (2). Therefore, proving (3) will establish a lower bound for the Hilbert--Poincar{\'e} series of the image, as required to complete the proof. 

For (3), the claimed generators can be checked to have non-trivial restriction to $\mathbb{F}_\ell[U,V]\langle D,E\rangle$ but trivial restriction to the center. The exterior elements $1,D,E$ and $DE$ are $\mathbb{F}_\ell[U,V]$-linearly independent, hence the only possible relation can arise between the images 
\[
\op{res}(2\op{c}_1\op{e}_1-3\op{e}_2)=2(U-V)D+2(V-U)E  \textrm{ and}
\]
\[
\op{res}(\op{c}_1\op{e}_2-2\op{c}_2\op{e}_1)=2(V^2-UV)D+2(U^2-UV)E.
\]
However, these are obviously $\mathbb{F}_\ell[U,V]$-linearly independent. This establishes that the Hilbert--Poincar{\'e} series claimed in (3) is a lower bound. From the previous reductions, this is all that remained to prove all the claims.
\end{proof}

\section{Cohomology of detection-graded pieces and the \texorpdfstring{$E_2$}{E2}-page} 
\label{sec:e2}

The goal of this section is now to compute the cohomology of the $\op{d}_1$-differential on the subquotients of the detection filtration. More precisely, we get a precise relation of the cohomology groups of these subquotients to the rank stratification of the quotient $\op{GL}_3(k[E])\backslash\mathfrak{B}(E/k,3)$. The most difficult part to deal with will be the essential rank $2$, related to the structure of the parabolic graph. Then we need to put the pieces back together to determine the Chern-class module structure of the $E_2$-page of the isotropy spectral sequence. Most of the work evaluating the spectral sequence will be done at this point, only a small discussion of the $\op{d}_2$-differential will have to be done in the next section. 

\subsection{Essential rank 3 and maximal elementary abelian $\ell$-subgroups}

The easiest piece to deal with is the filtration step $\op{D}^2$. Since elementary abelian $\ell$-subgroups of rank $3$ only appear as stabilizers of $0$-cells, the part $\op{D}^2$ of the $E_1$-page is concentrated completely in the column $E^{0,\bullet}_1$, cf. Lemma~\ref{lem:e1props}. 

\begin{proposition}
\label{prop:rankthree}
Let $k=\mathbb{F}_q$ be a finite field, let $\overline{E}$ be an elliptic curve over $k$ with $k$-rational point $\op{O}$ and set $E=\overline{E}\setminus\{\op{O}\}$. The filtration step $\op{D}^2$ is concentrated in the column $E^{0,\bullet}_1$ and is a free $\mathbb{F}_\ell[\op{c}_1,\op{c}_2,\op{c}_3]$-module with Hilbert--Poincar{\'e} series 
\[
\frac{T^2(1+T)^3(1-T+T^2)
\left((1+2T^2+2T^4+T^6)N_a+(T^2+T^4+T^6)N_b+T^6N_c\right)}
{(1-T^2)(1-T^4)(1-T^6)}.
\]
\end{proposition}

\begin{proof}
Using Lemma~\ref{lem:e1props}, the result follows immediately from Propositions~\ref{prop:detectiona}, \ref{prop:detectionb} and \ref{prop:detectionc}, and reordering the sum
\begin{eqnarray*}
&&N_a\cdot\frac{T^2(T+1)^3(T^2-T+1)}{(1-T^2)^3}\\&+&
N_b\cdot\frac{T^4(T+1)^3(T^2-T+1)}{(1-T^2)^2(1-T^4)}\\&+&
N_c\cdot\frac{T^8(T+1)^3(T^2-T+1)}{(1-T^2)(1-T^4)(1-T^6)}
\end{eqnarray*}
\end{proof}

\begin{corollary}
\label{cor:631}
The total rank of the $\op{D}^2$-part of the $\mathbb{F}_\ell$-cohomology of $\op{GL}_3(k[E])$ is $8\cdot(\#\overline{E}(\mathbb{F}_q))^2$.
\end{corollary}

\begin{proof}
The rank statement follows directly from Proposition~\ref{prop:rankthree}. The numerical identification follows from Lemma~\ref{lem:ellform}.
\end{proof}

\subsection{Essential rank 1 and the quotient of the building}

As a next step, we compute the quotient $E_1^{s,t}/\op{D}^1E_1^{s,t}$ of classes essentially of rank 1, i.e., the quotient of the $E_1$-page modulo the cohomology classes whose restriction to the center is trivial. The result contains cohomology classes which are torsion for the Chern classes $\op{c}_2$ and $\op{c}_3$, and these are related to the quotient $\op{GL}_3(k[E])\backslash\mathfrak{B}(E/k,3)$. 

\begin{proposition}
\label{prop:rankonee1}
Let $k=\mathbb{F}_q$ be a finite field, let $\overline{E}$ be an elliptic curve over $k$ with $k$-rational point $\op{O}$ and set $E=\overline{E}\setminus\{\op{O}\}$. The graded piece $\op{gr}^{\op{D}}_1E_1=E_1/\op{D}^1E_1$ of the $E_1$-term of the isotropy spectral sequence is given by 
\[
\op{gr}^{\op{D}}_1E_1^{s,t}\cong \op{H}^t(k^\times;\mathbb{F}_\ell)\otimes_{\mathbb{F}_\ell}
\op{C}^s(\op{GL}_3(k[E])\backslash\mathfrak{B}(E/k,3);\mathbb{F}_\ell). 
\]
Under this identification, the $\op{d}_1$-differentials satisfy 
\[
\op{d}_1^{s,t}=\op{id}_{\op{H}^t(k^\times;\mathbb{F}_\ell)}\otimes \op{d}_1^{s,0}.
\]
\end{proposition}

\begin{proof}
For any stabilizer subgroup $G\subseteq\op{GL}_3(k[E])$, the restriction map for the inclusion $k^\times\cong\mathfrak{Z}(\op{GL}_3(k[E]))\hookrightarrow G$ of the center is surjective by Lemma~\ref{lem:center}. By Definition~\ref{def:e1filtration} the filtration step $\op{D}^1$ is exactly given by the kernels of the restriction maps to the center. Since the center stabilizes the whole quotient, there will be a copy of $\op{H}^\bullet(k^\times;\mathbb{F}_\ell)$ for every cell in the quotient $\op{GL}_3(k[E])\backslash\mathfrak{B}(E/k,3)$. Moreover, the induced maps between the quotients $\op{H}^\bullet/\op{D}^1$ of the cohomology of the stabilizers will always be the identity. This implies that the quotient $E_1/\op{D}^1E_1$ of the $E_1$-page has the form claimed above.
\end{proof}

\begin{corollary}
\label{cor:rankonee1}
In the situation of Proposition~\ref{prop:rankonee1}, the cohomology of the graded piece $\op{gr}^{\op{D}}_1E_1=E_1/\op{D}^1E_1$ is of the form 
\[
\op{H}^s(\op{gr}^{\op{D}}_1E_1)\cong \op{H}^\bullet(k^\times;\mathbb{F}_\ell)\otimes_{\mathbb{F}_\ell} \op{H}^s(\op{GL}_3(k[E])\backslash\mathfrak{B}(E/k,3);\mathbb{F}_\ell). 
\]
For the cohomology of the quotient we have 
\[
\op{H}^s(\op{GL}_3(k[E])\backslash\mathfrak{B}(E/k,3);\mathbb{F}_\ell)\cong \left\{\begin{array}{ll}
\mathbb{F}_\ell & s=0\\
\op{St}_3(k)\otimes\mathbb{F}_\ell & s=2\\
0 & \textrm{otherwise}\end{array}\right.
\]
where $\op{St}_3(k)$ is the Steinberg representation of $\op{GL}_3(k)$. 
Therefore, it is a free $\mathbb{F}_\ell[\op{c}_1]$-module of total rank $2(q^3+1)$ with Hilbert--Poincar{\'e} series (in the total grading $n=s+t$)
\[
\frac{(1+T)(1+q^3T^2)}{(1-T^2)}.
\]
\end{corollary}

\begin{proof}
The general claim about the form of the cohomology of the graded piece $\op{gr}_1^{\op{D}}$ follows from Proposition~\ref{prop:rankonee1}. The description of the cohomology of the quotient follows from Theorem~\ref{thm:quotient} and the Solomon--Tits theorem which identifies the reduced cohomology of the flag complex for $k^n$ as being trivial except in degree $n-1$, where it is the Steinberg representation of $\op{GL}_n(k)$. 
For the Chern-class module structure, we use the identification of the multiplicative structure of the $E_1$-term in Proposition~\ref{prop:multiplicative}. Essentially, the universal Chern class  $\op{c}_1$ acts by cup-product with the restriction of $\op{c}_1$ to the center. But the cohomology of the center is $\mathbb{F}_\ell[\op{c}_1]\langle \op{e}_1\rangle$, a free module of rank $2$ over $\mathbb{F}_\ell[\op{c}_1]$. The remaining claims on the rank and Hilbert--Poincar{\'e} series now follow from the fact that the dimension of the Steinberg representation for $\op{GL}_n(\mathbb{F}_q)$ is $q^3$.
\end{proof}

\begin{remark}
Since the second column of the $E_1$-page only contains cohomology of cells stabilized by the center, the classes in the second column are torsion in $E_1$ for $\op{c}_2$ and $\op{c}_3$. However, the classes in the zero-column are non-torsion for $\op{c}_2$ and $\op{c}_3$; they only appear as torsion in the subquotient $E_1/\op{D}^1E_1$. In fact, the restriction of the universal classes $\op{e}_1$ and $\op{c}_1$ to the $E_1$-page of the isotropy spectral sequence for the  $\op{GL}_3(\mathbb{F}_q[E])$-action land in $E_1^{0,1}$ and $E_1^{0,2}$, respectively. The images of these elements in the $E_1$-page are natural lifts for the generators of the $s=0$-cohomology of the quotient $E_1/\op{D}^1E_1$. In particular, we see that these classes generate a free module of rank 2 over the Chern class ring $\mathbb{F}_\ell[\op{c}_1,\op{c}_2,\op{c}_3]$. Such matters are discussed below when we put together the detection-graded pieces to the $E_2$-page of the spectral seqeuence.
\end{remark}

\subsection{Essential rank 2 and the parabolic graph}

Now the most complicated subquotient for which we need to compute the cohomology is $\op{D}^1/\op{D}^2$ consisting of classes essentially of rank $2$. Its cohomology is very much related to the structure of the parabolic graph. From the results in Section~\ref{sec:detection}, we know the individual subquotients $\op{D}^1/\op{D}^2$ for the cohomology groups of stabilizers. 

By Lemma~\ref{lem:e1props}, the filtration step $\op{D}^1/\op{D}^2$ is a two-term complex 
\[
\overline{\op{d}}:\op{D}^1E^{0,\bullet}_1/\op{D}^2E^{0,\bullet}_1 \to\op{D}^1E^{1,\bullet}_1
\]
where the map $\overline{\op{d}}$ is the one induced from the $\op{d}^1$-differential of the isotropy spectral sequence. The term $\op{D}^1E^{1,\bullet}_1$ consists of the $\op{D}^1$-terms of the cohomology rings of edge stabilizers of the parabolic graph. We know from Proposition~\ref{prop:detectionrk2} that these are free graded $\mathbb{F}_\ell[\op{c}_1,\op{c}_2]$-modules of rank 8. The terms $\op{D}^1E^{0,\bullet}_1/\op{D}^2E^{0,\bullet}_1$ are the direct sums of $\op{D}^1/\op{D}^2$-subquotients of the cohomology rings of vertex stabilizers, which are free of $\mathbb{F}_\ell[\op{c}_1,\op{c}_2]$-modules of ranks described in Propositions~\ref{prop:detectiona}, \ref{prop:detectionb} and \ref{prop:detectionc}. 

\begin{remark}
The generators always arise in diamond-shaped groups of 4 elements. This is directly related to the structure of alternating polynomials in two variables. 
\end{remark}

The first important step is to show that the differential $\overline{\op{d}}$ of the complex $\op{D}^1/\op{D}^2$ is in fact linear over the Chern-class ring. This will allow to compute kernel and cokernel of the differential by only looking at the degrees in which the generators are living. 

\begin{lemma}
\label{lem:linearization}
The map $\overline{\op{d}}$ induced by the $\op{d}_1$-differential on the subquotient $\op{D}^1/\op{D}^2$ is linear over the Chern-class ring $\mathbb{F}_\ell[\op{c}_1,\op{c}_2]$. 
\end{lemma}

\begin{proof}
The Leibniz rule for right multiplication with a Chern class is
\[
\op{d}_1(x\cup\op{c}_i)=(\op{d}_1 x)\cup \op{c}_i+(-1)^{2i}x\cup(\op{d}_1\op{c}_i)
\]
Since the Chern classes $\op{c}_i$ here are the restrictions of the universal Chern classes to $\op{H}^\bullet(\op{GL}_3(k[E]);\mathbb{F}_\ell)$, they survive to the $E_\infty$-page. In particular, their differentials are always trivial, hence the second term in the above Leibniz rule vanishes. 
\end{proof}

The identification of the cohomology of the complex $\op{D}^1/\op{D}^2$ will be done in several steps. We first consider the $\mathbb{F}_\ell[X,Y]$-generators of the subquotients in vertex cohomology rings which lie in degrees 1,2,2,3. These generators are directly given by partially alternating polynomials. The following result shows that this part of the cohomology of $\op{D}^1/\op{D}^2$ is related to the $\ell$-rank of the elliptic curve $\overline{E}/\mathbb{F}_q$. 

\begin{proposition}
\label{prop:step1}
\begin{enumerate}
\item 
The kernel of the differential $\overline{\op{d}}$ of the complex $\op{D}^1/\op{D}^2$ in degree 1 has $\mathbb{F}_\ell$-dimension $\op{rk}_\ell(\overline{E})$, i.e., the $\ell$-rank of the elliptic curve. 
\item The cokernel in degree 1 has $\mathbb{F}_\ell$-dimension $N_a+\op{rk}_\ell(\overline{E})$.
\item The differential $\overline{\op{d}}$ on the other alternating $\mathbb{F}_\ell[X,Y]$-generators in  degrees 2,2,3 is the same as in degree 1. 
\item
The intersection of the kernel of $\overline{\op{d}}$ with the submodule generated by the generators in degree 1,2,2,3 is a free $\mathbb{F}_\ell[\op{c}_1,\op{c}_2]$-module with Hilbert--Poincar{\'e} series
\[
\frac{\op{rk}_\ell(\overline{E})T(1+T)^2(1+T^2)}{(1-T^2)(1-T^4)}
\]
\item
The  cokernel of the restriction of $\overline{d}$ to the generators in degrees 1,2,2,3 is a free module over $\mathbb{F}_\ell[\op{c}_1,\op{c}_2]$ with  Hilbert--Poincar{\'e} series
\[
\frac{\left(N_a+\op{rk}_\ell(\overline{E})\right)T(1+T)^2(1+T^2)}{(1-T^2)(1-T^4)}
\]
\end{enumerate}
\end{proposition}

\begin{proof}
We recall from Section~\ref{sec:detection} the description of the generators in degree 1 and from Section~\ref{sec:isotropy} the description of the differential. For vertices of type (IIa) we have generators $A-B$ and $B-C$, for vertices of type (IIb) we have the generator $\op{e}_1-2C$, and there is no generator for type (IIc) vertices. For stable bundles as well as the edges we have a generator $A-B$. Writing out the description of restriction maps on cohomology, we get the following matrices: 
\[
\left(\begin{array}{cc}0&1\\1&-1\\-1&0\end{array}\right),
\left(\begin{array}{c}2\\-1\end{array}\right),
\textrm{ and }(\underbrace{-1,\dots,-1}_{q+1})^{\op{t}}
\]
for type (IIa), (IIb) and the stable bundles, respectively. Using $-1$ instead of $1$ for the transition matrix of the stable bundles is due to our choice of orientation: all edges are oriented to point away from the stable bundles. Note that the bundles of type (IId) also contribute to the parabolic graph; the corresponding transition matrices are identity matrices.

We use the computation of Proposition~\ref{prop:k1}. Recall that the cokernel of the induced maps on abelianizations are given by 
\[
k^\times\oplus\left(\overline{E}(k)\otimes k^\times\right)\oplus \bigoplus_Pk(P)^\times/k^\times.
\]
The kernel of the above transition matrices is the $\mathbb{F}_\ell$-dual of the cokernel computed in Proposition~\ref{prop:k1}, up to the following modifications. First, the sum in the last term coming from degree 2 points on the curve does not contribute anything to the $\mathbb{F}_\ell$-dual of $\widehat{\op{H}}_1(E/k)$ because our assumption $\ell\mid q-1$ implies $k(P)^\times/k^\times\otimes k^\times=0$. The first summand $k^\times$ comes from the determinant, hence is detected on the diagonal. Since we are dealing with the subquotient $\op{D}^1/\op{D}^2$, this does not contribute. Therefore, the kernel we want to compute is the $\mathbb{F}_\ell$-dual of the middle term above. This proves that its $\mathbb{F}_\ell$-dimension equals the rank of the maximal elementary abelian $\ell$-subgroup of $\overline{E}(\mathbb{F}_q)$, which by definition is the $\ell$-rank. This proves (1). 

For (2) we can use an Euler characteristic computation together with part (4) of Lemma~\ref{lem:ellform} to deduce the dimension of the cokernel, cf. Corollary~\ref{cor:d1e2}. 

For (3), the argument given is essentially the same for the other degrees. The generators and restriction maps are explicitly identified, then the transition matrices are the same. The independence of the generators was established earlier, so the computation of kernel and cokernel as above applies. Statements (4) and (5) are then direct consequences.

The freeness statements in (4) and (5) follow immediately from the independence of the generators established in Section~\ref{sec:detection} and the linearization of Lemma~\ref{lem:linearization}. This implies that the kernel and cokernel are freely generated by respective bases in degrees 1,2,2,3. 
\end{proof}

\begin{remark}
\label{rem:cokernel}
We can also get a description of the kernel of the differential in Proposition~\ref{prop:k1}, which leads to a description of the generators of the cokernel of the restriction map in Proposition~\ref{prop:step1}. To determine the kernel of the differential, we label each edge with an automorphism of an edge bundle, i.e., an element of $k^\times\times k^\times$. The corresponding choices have to be compatible, in the sense that the differential of the labelling is $1$ (the unit of the stabilizer group). This means that at each vertex of the parabolic graph, the product of the labels of the edges must be $1$.\footnote{Note that there are no signs here. Our choice of orientation of the parabolic graph implies that at any vertex either all edges are inbound or all edges are outbound.} Then we need to check how many choices we have.
\begin{itemize}
\item At a point of type (IId), we have one choice of edge label for the one adjacent edge.
\item At a point of type (IIc), there is no choice. There is only one edge adjacent to this point, so in order for the differential to be trivial, the edge label must be $1$.
\item At a point of type (IIb), the label of one of the adjacent edges determines the other up to a square. Since we are only concerned with $\mathbb{F}_\ell$-coefficients for odd primes, this indeterminacy does not affect the end result. 
\item At a point of type (IIa), we can freely choose two edge labels, the third is then completely determined by the requirement that the differential be trivial. 
\item From the above possibilities, we have to subtract the conditions arising from points corresponding to stable bundles: there is one condition requiring that the product of the edge labels of the $q+1$ adjacent edges should be the unit of the vertex stabilizer.
\end{itemize}
Summing these possibilities gives the $N_a$-part in Proposition~\ref{prop:step1}, cf. also the Hilbert--Poincar{\'e} computation in Corollary~\ref{cor:d1e2}. The above way of choosing edge labels provides obvious generators for the $N_a$-part which corresponds to elements of the kernel where the differential is $0$ at each vertex. The elements in the $\ell$-rank part will not correspond to elements of the kernel; they only correspond to elements whose differential is $\ell$-divisible. Generators for this part are more complicated to describe, but we will not actually need a completely explicit description.
\end{remark}

Now we still need to consider the generators in degrees 3,4,4,5. By the computations in Section~\ref{sec:detection}, we always have a diamond of  generators which come from the cohomology of $\op{GL}_3(\mathbb{F}_q)$ and are the same for all vertices. However, since such a diamond of generators does not appear for the edge stabilizers if we consider the module structures over the polynomial ring $\mathbb{F}_\ell[X,Y]$, we now need to restrict attention to module structures over $\mathbb{F}_\ell[\op{c}_1,\op{c}_2]$.

\begin{proposition}
\label{prop:step2}
\begin{enumerate}
\item
The $\mathbb{F}_\ell[\op{c}_1,\op{c}_2]$-generators of $\op{D}^1/\op{D}^2$ in degrees 3,4,4,5 coming from the symmetric polynomials in the cohomology of $\op{GL}_3(\mathbb{F}_q)$ map surjectively onto the cokernel of the restricted differential identified in Proposition~\ref{prop:step1}. 
\item The intersection of the kernel of $\overline{d}$ with the $\mathbb{F}_\ell[X,Y]$-submodule generated by the degree 3,4,4,5 generators is  $\mathbb{F}_\ell[\op{c}_1,\op{c}_2]$-free linearly independent of the kernel identified in Proposition~\ref{prop:step1}, with Hilbert--Poincar{\'e} series
\[
\frac{\left((N_b+N_c-\op{rk}_\ell(\overline{E}))T^3+ (N_a+N_b+N_c)T^5\right)(1+T)^2}{(1-T^2)(1-T^4)}.
\]
\item The cokernel of the map (induced from $\overline{d}$) from the generators in degree 3,4,4,5 to the cokernel identified in Proposition~\ref{prop:step2} is a free $\mathbb{F}_\ell[\op{c}_1,\op{c}_2]$-module with Hilbert--Poincar{\'e} series
\[
\frac{\left(N_a+\op{rk}_\ell(\overline{E})\right)T(1+T)^2}{(1-T^2)(1-T^4)}.
\]
\end{enumerate}
\end{proposition}

\begin{proof}
Recall from Proposition~\ref{prop:detectionc} that the relevant generators of $\op{D}^1/\op{D}^2$ are the following ones, which also appear (explicitly written out) in Propositions~\ref{prop:detectiona} and \ref{prop:detectionb}. 
\[
2\op{c}_1\op{e}_1-3\op{e}_2, \;\op{e}_1\op{e}_2,\;\op{c}_1^2-3\op{c}_2, \; \op{c}_1\op{e}_2-2\op{c}_2\op{e}_1
\]
Their restrictions to edge stabilizers are written as follows: 
\begin{eqnarray*}
2\op{c}_1\op{e}_1-3\op{e}_2&\mapsto& 2(X_1-X_2)(Y_1-Y_2), \\
\op{e}_1\op{e}_2&\mapsto& 2Y_1Y_2(X_1-X_2),\\
\op{c}_1^2-3\op{c}_2&\mapsto & (X_1-X_2)^2\\
\op{c}_1\op{e}_2-2\op{c}_2\op{e}_1&\mapsto&2(X_1-X_2)(X_1Y_2-X_2Y_1).
\end{eqnarray*}
Note that all the classes are in the submodule generated by $(X_1-X_2)$, in fact they are given by the product of $(X_1-X_2)$ with a generator in degree 1,2,2,3; this produces exactly the $\mathbb{F}_\ell[\op{c}_1,\op{c}_2]$-generators in degrees 3,4,4,5 from Proposition~\ref{prop:step1}.

Note that the restriction maps are the same for all the vertices and all adjacent edges. We know from the description of generators of the cokernel in Remark~\ref{rem:cokernel} that there were choices of edge labels for the points of type (IIa) and (IIb). The restriction maps at these points were given by the matrices in the proof of Proposition~\ref{prop:step1}. The above computation of restrictions for the generators in degrees 3,4,4,5 means that the transition matrices are extended to 
\[
\left(\begin{array}{ccc}
0&1&1\\
1&-1&1\\
-1&0&1
\end{array}\right),\qquad
\left(\begin{array}{cc}
2&1\\
-1&1
\end{array}\right)
\]
These matrices now have full rank, which means that the restrictions of the generators in degrees 3,4,4,5 hit the $N_a$ generators of the cokernel in degree 1,2,2,3 determined in Proposition~\ref{prop:step1}. 

To deal with the $\ell$-rank part, we note that the transition maps above are the same ones we would use to compute the $\mathbb{F}_\ell$-cohomology of the graph. The kernel of the boundary map is one-dimensional giving all edges the same value. The cokernel detects the loops in the graph. Note that the global classes for the $\ell$-rank in Remark~\ref{rem:cokernel} do not come from loops of the graph, they only appear because some coefficient is $\ell$-divisible. Therefore, the classes contributing to the $\ell$-rank part of the cokernel in Proposition~\ref{prop:step1} lie in the image of the differential. 

Combining the above two paragraphs, we see that the symmetric generators in degree 3,4,4,5 map surjectively onto the generators of the cokernel from Proposition~\ref{prop:step1}. This proves (3). 
\end{proof}

\begin{remark}
\label{rem:split}
It is a bit complicated to identify the generators. Certainly the degree 3,4,4,5 generators of the points of type (IIa) are not in the kernel, as the proof shows. The $\mathbb{F}_\ell[\op{c}_1,\op{c}_2]$-generators in degrees 5,6,6,7 all lie in the kernel and generate a free submodule by the results of Section~\ref{sec:detection}. If the $\ell$-rank is $0$, then the same is true for the generators in degrees 3,4,4,5 for the points of type (IIb) and (IIc). However, if the $\ell$-rank is not zero, then there is only a linear subspace of the generators for (IIb) and (IIc) which lies in the kernel of the differential, and generators for this subspace are more difficult to identify. 

Anyway, the relevant point for later is that in the case of $\ell$-rank $0$, the kernel of $\overline{d}$ on $\op{D}^1/\op{D}^2$ is freely generated by the corresponding generators of the $\op{D}^1/\op{D}^2$-parts of the individual stabilizer cohomology groups, except the ones in degree 3,4,4,5 for the points of type (IIa). Put differently, the kernel of $\overline{d}$ is the direct sum of its intersections with summands in the direct sum decomposition of the $E_1$-page.
\end{remark}


\begin{corollary}
\label{cor:d1e2}
Let $k=\mathbb{F}_q$ be a finite field, let $\overline{E}$ be an elliptic curve over $k$ with a $k$-rational point $\op{O}$ and set $\overline{E}\setminus\{O\}$. 
\begin{enumerate}
\item The kernel of $\overline{\op{d}}$ is a free module over the Chern-class ring $\mathbb{F}_\ell[\op{c}_1,\op{c}_2]$ with Hilbert--Poincar{\'e} series
\[
\frac{T(1+T)^2\left(\op{rk}_\ell(\overline{E})+ (N_b+N_c)T^2+(N_a+N_b+N_c)T^4\right)}{(1-T^2)(1-T^4)}.
\]
The total rank is $4(N_a+\op{rk}_\ell(\overline{E}))+8\cdot\#\overline{E}(\mathbb{F}_q)$.
\item The cokernel of $\overline{\op{d}}$ is a free module over the Chern-class ring $\mathbb{F}_\ell[\op{c}_1,\op{c}_2]$ with Hilbert--Poincar{\'e} series 
\[
\frac{\left(N_a+\op{rk}_\ell(\overline{E})\right)T(1+T)^2}{(1-T^2)(1-T^4)}.
\]
The total rank is $4(N_a+\op{rk}_\ell(\overline{E}))$.
\end{enumerate}
\end{corollary}

\begin{proof}
By the linearization in Lemma~\ref{lem:linearization}, the entries in the two-term complex are free modules over $\mathbb{F}_\ell[\op{c}_1,\op{c}_2]$ and the differential is linear. The result is proved by computing the kernel and cokernel on generators. A description of kernel and cokernel with the required freeness properties and the Hilbert--Poincar{\'e} series is provided by  Propositions~\ref{prop:step1} and \ref{prop:step2}. 
\end{proof}

\begin{remark}
As a sanity check, we show  that the alternating sum of Hilbert--Poincar{\'e} series for the complex 
\[
0\to\ker\to \op{D}^1/\op{D}^2 E^{0,\bullet}_1\to \op{D}^1 E^{1,\bullet}_1\to \op{coker}\to 0
\]
is actually zero. For improved readability, we only discuss the numerators, all denominators will be $(1-T^2)(1-T^4)$ since we have modules over $\mathbb{F}_\ell[\op{c}_1,\op{c}_2]$. In $\op{D}^1E^{0,\bullet}_1/\op{D}^2E^{0,\bullet}_1$, we have the following contribution coming from points of type (IIa-c):
\[
N_aT(1+T^2)(2+T^2)(1+T)^2+N_bT(1+T)^2(1+T^2)^2+N_cT^3(1+T^2)(1+T)^2
\]
Propositions~\ref{prop:detectiona}, \ref{prop:detectionb} and \ref{prop:detectionc}. There are additional contributions from points of type (IId) and the stable bundles. Denoting by $N_d$ the number of points of type (IId), the contribution is
\[
(N_d+\#\overline{E}(\mathbb{F}_q))T(1+T)^2(1+T^2)
\]
by Proposition~\ref{prop:detectionrk2}. In $\op{D}^1E^{1,\bullet}_1$, we only have the contribution from the edges, which equals 
\[
\#\overline{E}(\mathbb{F}_q)(q+1)T(1+T)^2(1+T^2).
\]
Subtracting the edge contributions from the vertex contributions we get
\[
T(1+T)^2(1+T^2)((N_a+N_b+N_c)T^2+2N_a+N_b+N_d- q\cdot\#\overline{E}(\mathbb{F}_q)).
\]

We need to show that the difference of the Hilbert--Poincar{\'e} series of the kernel and the cokernel yield exactly the same result. We can check easily
\begin{eqnarray*}
&&T(T+1)^2(\op{rk}_\ell(\overline{E})+(N_b+N_c)T^2+(N_a+N_b+N_c)T^4)\\&&- (N_a+\op{rk}_\ell(\overline{E}))T(1+T)^2 \\
&=&T(1+T)^2(-N_a+(N_b+N_c)T^2+(N_a+N_b+N_c)T^4).
\end{eqnarray*}
Comparing coefficients, what remains to be proved is the equality
\[
3N_a+N_b+N_d-q\cdot\#\overline{E}(\mathbb{F}_q)=0
\]
For this, we note that 
\[
\chi':=N_a+N_b+N_c+N_d-q\cdot\#\overline{E}(\mathbb{F}_q)
\]
is the Euler characteristic of the nontrivial connected component of the parabolic graph. Using part (4) of Lemma~\ref{lem:ellform}, we find the required formula
\[
0=\chi'+2N_a-N_c=3N_a+N_b+N_d-q\cdot\#\overline{E}(\mathbb{F}_q).
\]
\end{remark}


\subsection{Computation of the \texorpdfstring{$E_2$}{E2}-page}

Now that we have computed the cohomology of the subquotients of the detection filtration on the $E_1$-page of the isotropy spectral sequence, we need to combine these calculations to a description of the $E_2$-page. Via the detection filtration (modulo forgetting the additional cohomological grading), the $E_1$-page is a complex filtered by the detection filtration, so there is an associated spectral sequence converging to the cohomology of this filtered complex (which is the $E_2$-page) and starting with the cohomology of the subquotients of the filtration. 

\begin{lemma}
\label{lem:degenerate}
The detection spectral sequence degenerates.
\end{lemma}

\begin{proof}
We argue successively via long exact sequences. Consider the long exact sequence used to combine the cohomologies of $\op{D}^2$ and $\op{D}^1/\op{D}^2$ to compute the cohomology of $\op{D}^1$. The boundary map from the cohomology is defined by lifting a cohomology class in $\op{D}^1/\op{D}^2$ to $\op{D}^1$, applying the differential which lands in the cohomology of $\op{D}^2$. However, the complex $\op{D}^2$ is only concentrated in degree 0, where the complex $\op{D}^1/\op{D}^2$ sits in (cohomological) degrees 0 and 1. The boundary map is then necessarily trivial and the cohomology of $\op{D}^1$ is the direct sum of the cohomologies of $\op{D}^2$ and $\op{D}^1/\op{D}^2$. 

Now we want to combine the cohomology of $\op{D}^1$ with the cohomology of $\op{D}^0/\op{D}^1$. We know that the cohomology of $\op{D}^0/\op{D}^1$ is concentrated in degrees $0$ and $2$, and the earlier argument shows that the classes in degree $2$ do not support non-trivial boundary map. In principle, the classes in degree $0$ could have non-trivial boundary. However, these classes come from the cohomology of $\op{GL}_3(\mathbb{F}_q)$ viewed as stabilizer of the trivial bundle. Since these classes survive to the $E_\infty$-page as can be checked by restriction to the center, the corresponding boundary map must be trivial. 
\end{proof}

At this point, the additive structure of the $E_2$-page is already determined, because there are no extension problems for $\mathbb{F}_\ell$-vector spaces.

\begin{corollary}
The Hilbert--Poincar{\'e} series of the column $E_2^{0,\bullet}$ has denominator $(1-T^2)(1-T^4)(1-T^6)$ and the following numerator polynomial:
\begin{eqnarray*}
&&(1+T)^3(1-T+T^2)\left((2N_a-N_c+1)T^6
+(N_a+N_b+N_c-1)T^5\right. \\&&+(2N_a+N_b-\op{rk}_\ell(\overline{E})+1)T^4+
(N_b+N_c-2)T^3
\\&& \left.+(N_a+1)T^2+(\op{rk}_\ell(\overline{E})-1)T+1\right)
\end{eqnarray*}
\end{corollary}

\begin{proof}
We just need to add the previously identified terms from Propositions~\ref{prop:rankthree}, \ref{prop:rankonee1} and Corollary~\ref{cor:d1e2}: 
\begin{eqnarray*}
&&(1+T)(1-T^4)(1-T^6)\\
&+&(\op{rk}_\ell(\overline{E})+(N_b+N_c)T^2+(N_a+N_b+N_c)T^4)T(1+T)^2(1-T^6)\\&+&
T^2(1+T)^3(1-T+T^2)\left((1+2T^2+2T^4+T^6)N_a\right.\\&&\left.+(T^2+T^4+T^6)N_b+T^6N_c\right)
\end{eqnarray*}
and reorder a bit.
\end{proof}

However, we still want to establish a statement on freeness of module structure over the Chern-class rings. For the columns $E^{i,\bullet}_2$ with $i=1,2$ this is no problem because the  column $i=1$ comes only from $\op{D}^1/\op{D}^2$ and the column  $i=2$ comes only from $\op{D}^0/\op{D}^1$. At this point, we run into trouble establishing freeness for the column $E^{0,\bullet}_2$, essentially because the filtration part $\op{D}^1$ for the points of type (IIc) is not free. 

\begin{lemma}
\label{lem:helpfree}
The $\mathbb{F}_\ell[\op{c}_1,\op{c}_2,\op{c}_3]$-submodule $\op{D}^1\subset \op{H}^\bullet(\op{GL}_3(\mathbb{F}_q);\mathbb{F}_\ell)$ is not free. Its Hilbert--Poincar{\'e} series is
\[
\frac{T^3+2T^4+2T^5+2T^6+T^7+T^8+T^9-T^{10}-T^{11}}{(1-T^2)(1-T^4)(1-T^6)}
\]
\end{lemma}

\begin{proof}
The computation of the Hilbert--Poincar{\'e} series follows directly from Proposition~\ref{prop:detectionc}. Free graded modules always have Hilbert--Poincar{\'e} series such that the coefficients of the numerator polynomial are all positive. Since this is not the case here, $\op{D}^1$ is not free as graded $\mathbb{F}_\ell[\op{c}_1,\op{c}_2,\op{c}_3]$-module.
\end{proof}

\begin{lemma}
\label{lem:splitd0}
Let $k=\mathbb{F}_q$ be a finite field. Assume the elliptic curve $\overline{E}$ has at least one $k$-rational point other than $\op{O}$. Then the $E_2$-page splits off a free $\mathbb{F}_\ell[\op{c}_1,\op{c}_2,\op{c}_3]$-submodule of rank 2 which canonically lifts the part $\op{D}^0/\op{D}^1$ part in $E^{0,\bullet}_2$. 
\end{lemma}

\begin{proof}
We always have the inclusion of constants $\op{GL}_3(k)\hookrightarrow\op{GL}_3(k[E])$. Under the assumption, we also have an evaluation at the rational point $\op{P}$ in $E=\overline{E}\setminus\{\op{O}\}$ which is a group homomorphism $\op{GL}_3(k[E])\to\op{GL}_3(k)$. The composition of both maps is the identity on $\op{GL}_3(k)$ and therefore induces a splitting in cohomology:
\[
\op{H}^\bullet(\op{GL}_3(k[E]);\mathbb{F}_\ell)\cong \op{H}^\bullet(\op{GL}_3(k);\mathbb{F}_\ell)\oplus\mathcal{Q}. 
\] 
The cohomology of $\op{GL}_3(k)$ above contains the image of the universal Chern-class ring, diagonally embedded in the vertex stabilizer cohomologies via restriction. The restriction map to the center provides a surjective map from the summand $\op{H}^\bullet(\op{GL}_3(k);\mathbb{F}_\ell)$ to the cohomology of the center, and therefore this summand surjects onto the $\op{D}^0/\op{D}^1$ part of $E^{0,\bullet}_2$ which was identified to be isomorphic to the cohomology of the center in Proposition~\ref{prop:rankonee1}. 
\end{proof}

\begin{theorem}
\label{thm:e2nonfree}
Let $k=\mathbb{F}_q$ be a finite field and let $\overline{E}$ be an elliptic curve over $\mathbb{F}_q$. Let $\ell\geq 5$ be a prime such that $\ell\mid q-1$. Assume that $\overline{E}$ satisfies $\op{rk}_\ell(\overline{E})=0$ and $\op{rk}_3(\overline{E})>0$. Then $E^{0,\bullet}_2$ is not free as $\mathbb{F}_\ell[\op{c}_1,\op{c}_2,\op{c}_3]$-module. 
\end{theorem}

\begin{proof}
The assumption $\op{rk}_\ell(\overline{E})=0$ is used to not have to deal with the $\ell$-rank contribution in $\op{D}^1/\op{D}^2$. It implies that $\op{D}^1E^{0,\bullet}_2$ is isomorphic to the direct sum of its intersections with the vertex stabilizer cohomologies in the $E_1$-page, cf.~Remark~\ref{rem:split}. Lemma~\ref{lem:splitd0} shows that the $\op{D}^0/\op{D}^1$ comes from a free rank one direct summand of $E_2^{0,\bullet}$, diagonally embedded. By intersection with $\op{D}^1$, the module $\op{D}^1E_2^{0,\bullet}$ contains a diagonally embedded copy of the module $\op{D}^1\op{H}^\bullet(\op{GL}_3(\mathbb{F}_q);\mathbb{F}_\ell)$. To prove that the module $E_2^{0,\bullet}$ is not free, it suffices to show that the quotient
\[
\op{D}^1E_2^{0,\bullet}/ \Delta(\op{D}^1\op{H}^\bullet(\op{GL}_3(\mathbb{F}_q);\mathbb{F}_\ell))
\]
is not free. 
Now take the intersection (in $\op{D}^1E_1$) of the product of $\op{D}^1$-submodules of vertex stabilizer cohomology rings for stabilizers of type (IIa) and (IIb) with the $\op{D}^1E_2^{0,\bullet}$. Since the product of stabilizer cohomologies is a direct summand of $E_1$, it is also a direct summand of the quotient $\op{D}^1E_2^{0,\bullet}/\Delta$. The complement is the direct sum of $N_c$ copies of $\op{D}^1\op{H}^\bullet(\op{GL}_3(\mathbb{F}_q);\mathbb{F}_\ell)$ modulo a diagonally embedded copy of $\op{D}^1\op{H}^\bullet(\op{GL}_3(\mathbb{F}_q);\mathbb{F}_\ell)$. The Hilbert--Poincar{\'e} series is now, by Lemma~\ref{lem:helpfree}
\[
(N_c-1)\frac{T^3+2T^4+2T^5+2T^6+T^7+T^8+T^9-T^{10}-T^{11}}{(1-T^2)(1-T^4)(1-T^6)}. 
\]
By assumption $N_c>1$ because $N_c=3\op{rk}_3(\overline{E})$, cf.~Lemma~\ref{lem:ellform}. We have thus exhibited a direct summand of $E^{0,\bullet}_2$ which is not free as a module over the Chern-class ring. By \cite{atiyah}, the Krull--Schmidt theorem holds for graded modules over the Chern-class ring because these can be identified with coherent sheaves on the weighted projective space. In particular, a graded module with a non-free graded direct summand is not free, which proves the claim.
\end{proof}

\begin{remark}
The curve $\overline{E}$ given by $Y^2=X^3+X+8$ over $\mathbb{F}_{11}$ has $\overline{E}(\mathbb{F}_{11})=\mathbb{Z}/6\mathbb{Z}$ and satisfies the conditions for $\ell=5$.
\end{remark}

\section{The \texorpdfstring{$\op{d}_2$}{d2}-differential} 
\label{sec:d2}

It remains to evaluate the $\op{d}_2$-differential. Since the building $\mathfrak{B}(E/k,3)$ has dimension 2, the differentials $\op{d}_i$ for $i\geq 3$ are obviously trivial, and the spectral sequence degenerates at the $E_3$-term. The only interesting second differential is of the form 
\[
\op{d}_2^{0,t}:E^{0,t}_2\to \left(\op{H}^{t-1}(\mathbb{F}_q^\times;\mathbb{F}_\ell)
\otimes\op{St}_3(\mathbb{F}_q)\right).
\]

Unfortunately, it turns out that the differential does not completely vanish; in particular, general-purpose vanishing results cannot be applied and at least some amount of computation is necessary. The computation of the $\op{d}_2$-differential will also be done with the help of the detection filtration, resp. the induced filtration on the $E_2$-page. 

\subsection{Vanishing on classes ``essentially of rank 3''} 

As a first step, we show that the $\op{d}_2$-differential vanishes on the submodule $\op{D}^2\subseteq E^{0,\bullet}_2$. It is an instructive example that the detection filtration is also very useful for dealing with higher differentials. 

\begin{lemma}
\label{lem:d2vanish}
The differential $\op{d}_2$ vanishes on $\op{D}^2\subseteq E^{0,\bullet}_2$. 
\end{lemma}

\begin{proof}
By definition, the submodule $\op{D}^2\subseteq E^{0,\bullet}_2$ is the direct sum of the submodules $\op{D}^2\subseteq \op{H}^\bullet(G_\sigma;\mathbb{F}_\ell)$ of vertex stabilizer cohomologies. Moreover, a class $\gamma\in \op{H}^t(G_\sigma;\mathbb{F}_\ell)$ is in $\op{D}^2$ if and only if it is in the kernel of $\op{d}^1$ restricted to $\op{H}^t(G_\sigma;\mathbb{F}_\ell)$.\footnote{Note that this is much stronger than  what can usually be said about classes in kernels of differentials in spectral sequences.} 
Now we can represent the class as a cohomology class for an elementary abelian $\ell$-group and take a minimal complex for the elementary abelian $\ell$-groups. In this particular model, the cohomology class in the $\op{D}^2$ will have restrictions which are trivial (as opposed to being null-cohomologous). But then we can choose the coboundaries to be trivial as well and therefore the differential $\op{d}_2$ is trivial on $\op{D}^2$. 
\end{proof}

\begin{remark}
The conceptual remark to be made here is that the detection filtration is defined using stronger vanishing restrictions for the restriction maps in group cohomology. The stronger such a local vanishing condition is, the more differentials of the spectral sequence will vanish on the corresponding filtration step. This restriction for the higher differentials is very similar to the expectations for rank filtrations in algebraic K-theory -- the rank-graded pieces of the group homology spectral sequence should actually degenerate to complexes, cf. Goncharov's work on the trilogarithm and the work of Dupont and Sah on cohomology of $\op{PGL}_3(k)$, cf. \cite{dupont:book}. Of course, part of the present work is inspired by the study of rank filtrations in K-theoretic contexts.
\end{remark}

\subsection{Reduction to degree $1$}

In the next step, we investigate the $\op{d}_2$-differential on the kernel of $\overline{\op{d}}:\op{D}^1E^{0,\bullet}_1/\op{D}^2E^{0,\bullet}_1\to\op{D}^1E^{1,\bullet}_1$. This is well-defined by the earlier Lemma~\ref{lem:d2vanish}. 

\begin{lemma}
\label{lem:d2lower}
Let $\gamma\in\ker\overline{\op{d}}$ be a class in the $\ell$-rank contribution identified in Proposition~\ref{prop:step1}. Then $\gamma$ is obtained from a class in degree 1 or 2 via multiplication with the universal class $\op{e}_1$ and possibly a diagonal alternating polynomial $X_1-X_2$. 
As a consequence, the differential $\op{d}_2$ on these classes maps surjectively to a free $\mathbb{F}_\ell[\op{c}_1]\langle\op{e}_1\rangle$-submodule of $E^{2,\bullet}_2$ of rank $\op{rk}_\ell(\overline{E})$. 
\end{lemma}

\begin{proof}
The statement about multiplicative generation from  degrees 1 and 2  follows from the computations in Section~\ref{sec:detection}. 
Now consider the Leibniz rule for the differential. For multiplication with the universal $\op{e}_1$, we have 
\[
\op{d}_2^{0,2+i}(\op{e}_1\cup a)=(-1)^i\op{e}_1\cup \op{d}_2^{0,i}(a)
\]
because the universal $\op{e}_1$ survives to the $E_3$ page and hence its differential is trivial. 

As identified in Proposition~\ref{prop:step1}, the classes in degree 1 and 2 are linear combinations of classes of the form $A_1-A_2$ and $X_1-X_2$ in the various stabilizer cohomologies, respectively. For multiplication with $(X_1-X_2)$ we have
\[
\op{d}_2((X_1-X_2)\cup a)=\op{d}_2(X_1-X_2)\cup a +(-1)^{i}(X_1-X_2)\cup(\op{d}_2a)
\]
The cup product $(X_1-X_2)\cup(\op{d}_2a)$ is computed by first restricting $(X_1-X_2)$ to the relevant stabilizer subgroups and then taking cup product with $\op{d}_2a$. But the restriction of the alternating polynomial to the center will be trivial. This shows $\op{d}_2((X_1-X_2)\cup a)=\op{d}_2(X_1-X_2)\cup a$. As a consequence, the differential applied to the classes from Proposition~\ref{prop:step1} will always be cup-multiples of the corresponding classes in degree 1 and 2. 

To prove the claim, it now suffices to compute the differential in degree 1. In this degree, we have $E^{0,1}_2\cong \mathbb{F}_\ell^{1+\op{rk}_\ell(\overline{E})}$ and the differential is
\[
\op{d}_2^{0,1}:E^{0,1}_2\to \op{St}_3(k)\otimes\mathbb{F}_\ell.
\]
Since $E^{1,0}_2=0$ and the spectral sequence converges to $\op{H}^\bullet(\op{GL}_3(k[E]);\mathbb{F}_\ell)$, we have 
\[
\ker\op{d}_2^{0,1}\cong \op{H}^1(\op{GL}_3(k[E]);\mathbb{F}_\ell).
\]
The latter is known, by stabilization for group homology and the computation of K-theory of curves over finite fields, cf.~e.g.~\cite{weibel:kbook}, to be isomorphic to $\mathbb{F}_\ell$, and the corresponding cohomology classes are detected on the center. In particular, the differential $\op{d}_2^{0,1}$ must be injective. This proves the assertion about the rank of the image in $E^{2,0}_2$. Since the generators of $E^{2,1}_2$ are cup-products of classes from $E^{2,0}$ with the universal $\op{e}_1$, we get the same statement for $E^{2,1}_2$. The rest now follows from the above reduction to rank 1 and 2.
\end{proof}

\begin{lemma}
\label{lem:d2upper}
Let $\gamma\in\ker\overline{\op{d}}$ be among the classes identified in Proposition~\ref{prop:step2}. Then its $\op{d}_2$-differential is trivial. 
\end{lemma}

\begin{proof}
This is now the same argument as in Lemma~\ref{lem:d2vanish}. 
\end{proof}

It remains to say something about the $\op{d}_2$-differential on those classes not in the filtration step $\op{D}^1$. By Lemma~\ref{lem:splitd0}, we can canonically lift the $\op{D}^0/\op{D}^1$-part to the $E_2^{0,\bullet}$-column, as the cohomology $\op{H}^\bullet(\op{GL}_3(\mathbb{F}_q);\mathbb{F}_\ell)$ of the stabilizer of the trivial bundle. Since this splits off the cohomology, it survives to the $E_\infty$-page and therefore the $\op{d}_2$-differential is trivial on the summand $\op{H}^\bullet(\op{GL}_3(\mathbb{F}_q);\mathbb{F}_\ell)$. This implies that we have completely computed the $\op{d}_2$-differential. As a consequence, the $E_3$-page only differs from the $E_2$-page by a free $\mathbb{F}_\ell[\op{c}_1]\langle \op{e}_1\rangle$-module of rank $\op{rk}_\ell(\overline{E})$ in the columns $E^{0,\bullet}$ and $E^{2,\bullet}$, respectively.

\section{Quillen's conjecture on cohomology of arithmetic groups}
\label{sec:quillenconj}

In this section, we discuss Quillen's conjecture on the structure of cohomology rings of $S$-arithmetic groups. The computations of $\op{H}^\bullet(\op{GL}_3(k[E]);\mathbb{F}_\ell)$ provide new insights into the nature of the possible failures of the conjecture and a variation of the original formulation of the conjecture in \cite{quillen:spectrum}.

\subsection{Conjecture and known results} 
\label{known}
First, let us state Quillen's original conjecture, cf. \cite[Conjecture 14.7, p. 591]{quillen:spectrum}.  For any number field $K$, and any set of places $S$ of $K$, the natural embedding $\op{GL}_n(\mathcal{O}_{K,S})\hookrightarrow\op{GL}_n(\mathbb{C})$ induces a  restriction map in cohomology
\[
\op{res}_{K,S}:
\op{H}^\bullet(\op{GL}_n(\mathbb{C});\mathbb{F}_\ell)\to  
\op{H}^\bullet(\op{GL}_n(\mathcal{O}_{K,S});\mathbb{F}_\ell). 
\]
Moreover, denoting by $\op{c}_i$ the $i$-th Chern class in $\op{H}^\bullet_{\op{cts}}(\op{GL}_n(\mathbb{C});\mathbb{F}_\ell)$,
there is a change-of-topology map 
\[
\delta:\mathbb{F}_\ell[\op{c}_1,\dots,\op{c}_n]\cong 
\op{H}^\bullet_{\op{cts}}(\op{GL}_n(\mathbb{C});\mathbb{F}_\ell)\to 
\op{H}^\bullet(\op{GL}_n(\mathbb{C});\mathbb{F}_\ell).
\]
The conjecture of Quillen can now be stated as follows: 
\begin{conjecture}[Quillen]
\label{conj:quillen}
Let $\ell$ be a prime number. Let $K$ be a number field with $\zeta_\ell\in K$, and $S$ a finite set of places containing the infinite places and the places over $\ell$. The the cohomology ring $\op{H}^\bullet(\op{GL}_n(\mathcal{O}_{K,S});\mathbb{F}_\ell)$ is a free
module over the cohomology ring $\op{H}^\bullet_{\op{cts}}(\op{GL}_n(\mathbb{C});\mathbb{F}_\ell)\cong
\mathbb{F}_\ell[\op{c}_1,\dots,\op{c}_n]$ via the composition $\op{res}_{K,S}\circ \delta$.  
\end{conjecture}
The range of validity of the conjecture has been unclear for quite a long time. It has been proved in the cases of the group $\op{GL}_2(\mathbb{Z}[1/2])$ by Mitchell, the group $\op{GL}_3(\mathbb{Z}[1/2])$ by Henn, and the group $\op{GL}_2(\mathbb{Z}[\zeta_3,1/3])$ by Anton. A more recent positive result is the fact that the Quillen conjecture is true above the virtual cohomological dimension for groups $\op{SL}_2(\mathcal{O}_{K,S})$, cf. \cite{quillen-note}. Based on arguments of Henn--Lannes--Schwartz, counterexamples to Quillen's conjecture have been established by  Dwyer  for $\op{GL}_n(\mathbb{Z}[1/2])$, $n\geq 32$, Henn and Lannes for $\op{GL}_n(\mathbb{Z}[1/2])$, $n\geq 14$, and by Anton for $\op{GL}_n(\mathbb{Z}[\zeta_3,1/3])$, $n\geq 27$. For precise literature references, we refer to the discussion in \cite{knudson:book} or \cite{quillen-note}. Note that all the counterexamples to the Quillen conjecture known so far are established by  first showing that Quillen's conjecture implies detection of cohomology classes on the diagonal matrices, and then producing examples of this failure of detection. Because of this rather indirect method, known cases or counterexamples to Quillen's conjecture do not provide us with general information on the structure of cohomology rings of $S$-arithmetic groups.

To get a better feeling for the range of validity of Quillen's conjecture for $S$-integer rings, its tendency to be true or false, it is also useful to look at the analogous question in the function field situation. Let $K/\mathbb{F}_q(T)$ be a global function field of characteristic $p$ with chosen algebraic closure $\overline{K}$, let $S$ be a non-empty finite set of places and let $\ell$ be a prime different from $p$. As in the number field case, we can embed $\op{GL}_n(\mathcal{O}_{K,S})\hookrightarrow\op{GL}_n(\overline{K})$ and get an induced restriction map 
\[
\op{res}_{K,S}:\op{H}^\bullet(\op{GL}_n(\overline{K});\mathbb{F}_\ell)\to \op{H}^\bullet(\op{GL}_n(\mathcal{O}_{K,S});\mathbb{F}_\ell). 
\]
At this point there is a slight problem because the cohomology of $\op{H}^\bullet(\op{GL}_n(\overline{K});\mathbb{F}_\ell)$ is not presently known, though predicted to be isomorphic to  $\op{H}^\bullet(\op{GL}_n(\overline{\mathbb{F}_q});\mathbb{F}_\ell)$ by the Friedlander--Milnor conjecture. The proper analogue of continuous cohomology to use in this situation is then \'etale cohomology. As before, there is a change-of-topology map and we get
\[
\delta:\mathbb{F}_\ell[\op{c}_1,\dots,\op{c}_n]\cong \op{H}^\bullet_{\rm \acute{e}t}(\op{GL}_n(\overline{K});\mathbb{F}_\ell)\to
\op{H}^\bullet(\op{GL}_n(\overline{K});\mathbb{F}_\ell).
\]
Alternatively, we can use the following argument: using Quillen's computations of cohomology of linear groups over finite fields, which will also be recalled in Section~\ref{sec:reccohom}, combined with the map induced from the inclusion of $\op{GL}_n(\overline{\mathbb{F}_q})\hookrightarrow\op{GL}_n(\overline{K})$: 
\[
\op{H}^\bullet(\op{GL}_n(\overline{K});\mathbb{F}_\ell)\to \op{H}^\bullet(\op{GL}_n(\overline{\mathbb{F}_q});\mathbb{F}_\ell)\cong \mathbb{F}_\ell[\op{c}_1,\dots,\op{c}_n].
\]
Now consider the composition
\[
\mathbb{F}_\ell[\op{c}_1,\dots,\op{c}_n]\hookrightarrow \mathbb{F}_\ell[\op{c}_1,\dots]\cong\op{H}^\bullet(\op{GL}_\infty(\overline{K});\mathbb{F}_\ell)\to \op{H}^\bullet(\op{GL}_n(\overline{K});\mathbb{F}_\ell),
\]
where the first map is the inclusion of the subring generated by the first $n$ variables, the second is Suslin's rigidity for K-theory with finite coefficients, and the last map is induced from the  inclusion $\op{GL}_n(\overline{K})\hookrightarrow\op{GL}_\infty(\overline{K})$. This composition splits the former map, so that we obtain a split injective map
\[
\delta:\mathbb{F}_\ell[\op{c}_1,\dots,\op{c}_n]\hookrightarrow \op{H}^\bullet(\op{GL}_n(\overline{K});\mathbb{F}_\ell).
\]
These two descriptions of $\delta$ actually agree, as can be seen from stabilization to $\op{GL}_\infty$ and the rigidity theorem. After these preparations, we are ready to state the following analogue of Quillen's conjecture.

\begin{conjecture}[Function field version of Quillen's conjecture]
  \label{conj:quillenff}
Let $\ell$ be a prime number. Let $K$ be a function field of a smooth projective curve $C/\mathbb{F}_q$ with $\ell\mid q-1$, and let $S$ be a nonempty finite set of places of $K$. Then the cohomology ring $\op{H}^\bullet(\op{GL}_n(\mathcal{O}_{K,S});\mathbb{F}_\ell)$ is a free module over the Chern class ring  $\mathbb{F}_\ell[\op{c}_1,\dots,\op{c}_n]$ via the 
composition $\op{res}_{K,S}\circ \delta$.
\end{conjecture}

\begin{remark}
As an aside, it would also be possible to generalize a bit further: if we take a split Chevalley group $G/\mathbb{Z}$, we can ask if the natural map $\op{H}^\bullet_{\op{top}}(G(\mathbb{C});\mathbb{F}_\ell)\to\op{H}^\bullet(G(\mathcal{O}_K);\mathbb{F}_\ell)$ makes the target a free module. Versions of the present results for $\op{SL}_3$ and $\op{PGL}_3$ are possible but not included in the paper.
\end{remark}

We now give a short list of the known positive and negative results pertaining to the function field version of Quillen's conjecture. As in the number field cases, these examples are not sufficiently general because either they concern arithmetically trivial situations or groups of rank one.

\begin{itemize}
\item The Quillen conjecture is generally true for all $\op{GL}_n$ for $C=\mathbb{P}^1$ and $S=\{\infty\}$. This is a consequence of Soul{\'e}'s computations which show 
\[
\op{H}^\bullet(\op{GL}_n(\mathbb{F}_q[T]);\mathbb{F}_\ell)\cong \op{H}^\bullet(\op{GL}_n(\mathbb{F}_q);\mathbb{F}_\ell),
\]
again combined with Quillen's computations of cohomology of general linear  groups over finite fields.
\item The Quillen conjecture is true for $\op{GL}_2$ and the curves $C=\mathbb{P}^1\setminus\{0,\infty\}$,  $C=\mathbb{P}^1\setminus\{0,1,\infty\}$.  This follows from \cite{sl2parabolic}
\item The Quillen conjecture is true for $\op{GL}_2$ and $C=\overline{E}\setminus\{\op{O}\}$ with $\overline{E}$ an elliptic curve with a $k$-rational point $\op{O}$, by the computations in \cite[Section 4.5]{knudson:book}.  
\item More generally, the Quillen conjecture is true for $\op{SL}_2(\mathcal{O}_{K,S})$ in degrees bigger than $\# S$, by \cite{sl2parabolic,quillen-note}. 
\item An example of the failure of the Quillen conjecture can be obtained by computing $\op{H}^\bullet(\op{SL}_2(k[\mathbb{P}^1\setminus\{P_1,\dots,P_s\}]);\mathbb{F}_\ell)$ for $s\geq 4$ and sufficiently big ground field of odd characteristic. A new phenomenon appears in this example: for $s\geq 4$, there are classes in $\op{H}^s$ which can not be detected on any finite subgroup of $\op{SL}_2(k[\mathbb{P}^1\setminus\{P_1,\dots,P_s\}])$, and which are annihilated by the second Chern class. 
\end{itemize}

\subsection{Elementary abelian $\ell$-groups and Quillen homomorphism}

Now we want to outline Quillen's analysis of the $E_2$-term of the isotropy spectral sequence in terms of the category of the elementary abelian $\ell$-subgroups, cf. \cite[Part I]{quillen:spectrum}. This allows to translate statements about the Quillen homomorphism, its kernel, image and related things, into statements about the isotropy spectral sequence.

Consider the action of the group $G=\op{GL}_3(k[E])$ on the associated Bruhat--Tits building $X=\mathfrak{B}(E/k,3)$. On the orbit space $X/G$, consider the sheaf  $\mathcal{H}_G$ of Definition~\ref{def:hg}. Using the identifications 
\[
E^{s,t}_2\cong \op{H}^s(X/G;\mathcal{H}_G^t),
\]
from \cite[Section 3]{quillen:spectrum}, Proposition 3.2 of \cite{quillen:spectrum} implies that the natural map $\op{H}^\bullet(\op{GL}_3(k[E]);\mathbb{F}_\ell)\to E^{0,\bullet}_2$ is an F-isomorphism. 

For the group $G=\op{GL}_3(k[E])$, denote by $\mathcal{A}_G$ the category of elementary abelian $\ell$-subgroups of $G$ with morphisms $\theta:A\to A'$ given by conjugation $\theta(a)=g^{-1}ag$ for $g\in G$. There is an associated contravariant diagram of graded-commutative rings given by sending
\[
A\in\mathcal{A}_G\mapsto \op{H}^\bullet(A;\mathbb{F}_\ell)
\]
and sending morphisms of elementary abelian groups to the associated restriction maps on cohomology. Taking the restriction map associated to $A\leq G$ for $A\in\mathcal{A}_G$ provides a cone on this diagram, and the associated universal arrow is called Quillen homomorphism.

Since all the relevant maps are restriction maps associated to subgroup inclusions, the Quillen homomorphism factors as 
\[
\op{H}^\bullet(\op{GL}_3(k[E]);\mathbb{F}_\ell)\stackrel{\pi}{\longrightarrow} E^{0,\bullet}_\infty\hookrightarrow E^{0,\bullet}_2\cong \op{H}^0(X/G;\mathcal{H}_G^\bullet)\to \lim_{A\in\mathcal{A}_G}\op{H}^\bullet(A;\mathbb{F}_\ell).  
\]
Since the Bruhat--Tits building over a finite field is locally compact and contractible, the combination of Proposition 5.11 and Lemma 6.3 of \cite{quillen:spectrum} shows that the last arrow, comparing $E^{0,\bullet}_2$ and $\lim_{A\in\mathcal{A}_G}\op{H}^\bullet(A;\mathbb{F}_\ell)$ is an isomorphism, cf. also \cite[Theorem 7.1]{quillen:spectrum}. 

Denote by $\op{R}_\bullet\op{H}^\bullet$ the filtration associated to the isotropy spectral sequence. The description above implies immediately that the kernel of the Quillen homomorphism is exactly given by the filtration steps containing $E^{\geq 1,\bullet}_\infty$ and the image coincides with $E^{0,\bullet}_\infty\subset E^{0,\bullet}_2$. In particular, kernel and image of the Quillen homomorphism can be investigated using the isotropy spectral sequence, and it seems that the spectral sequence is quite well-suited for this investigation.

\subsection{Structure of the kernel}
We first investigate the structure of the kernel of the Quillen homomorphism. By the previous description, the kernel is the filtration step $E^{\geq 1,\bullet}_\infty$ of the isotropy filtration on the cohomology ring $\op{H}^\bullet(\op{GL}_3(k[E]);\mathbb{F}_\ell)$. Moreover, the filtration by columns $E^{\geq i,\bullet}_\infty$ is the filtration of the kernel in the statement of Theorem~\ref{thm:kernel}.

\begin{proposition}
\label{prop:torsion}
The kernel of the Quillen homomorphism  
\[
\op{H}^\bullet(\op{GL}_3(k[E]);\mathbb{F}_\ell)\to E^{0,\bullet}_2
\]
is exactly the sub-$\mathbb{F}_\ell[\op{c}_1,\op{c}_2,\op{c}_3]$-module of torsion elements. 
\end{proposition}

\begin{proof}
This follows directly from Proposition~\ref{prop:multiplicative} describing the multiplicative structure of the isotropy spectral sequence. The Chern-class module structure on the $E_1$-term is given by multiplication with the Chern-classes of the representation $A\leq\op{GL}_3(k[E])\hookrightarrow\op{GL}_3(\overline{k(E)})$ of the stabilizer subgroups. The identification of the simplified $E_1$-term in \ref{prop:simplee1} shows elementary abelian subgroups of rank $3$ only appear as vertex stabilizers. Stabilizers of edges and 2-simplices contain elementary abelian subgroups of rank $\leq 2$. Therefore, any cohomology class in $E^{\geq 1,\bullet}_1$ is annihilated by $\op{c}_3$. This property is preserved by passing to subquotients, in particular it is true for $E^{\geq 1,\bullet}_\infty$ so that all these classes are torsion for $\mathbb{F}_\ell[\op{c}_1,\op{c}_2,\op{c}_3]$. 

On the other hand, $E^{0,\bullet}_\infty\subseteq E^{0,\bullet}_2$ and the latter is contained in a direct sum of cohomology rings $\op{H}^\bullet(G_\sigma;\mathbb{F}_\ell)$ of vertex stabilizers containing elementary abelian groups of rank 3. All the stabilizers appearing in our case have cohomology rings which are graded free modules over the Chern-class ring $\mathbb{F}_\ell[\op{c}_1,\op{c}_2,\op{c}_3]$, cf. Section~\ref{sec:reccohom}. In particular, the module $E^{0,\bullet}_\infty$ is torsion-less. 

The two statements above now imply the claim.
\end{proof}

We are now ready to prove the main result about the structure of the kernel of the Quillen homomorphism.

\begin{proofof}{Theorem~\ref{thm:kernel}}
\label{pfkernel}
(1) has been proved in Proposition~\ref{prop:torsion}. 

(2) This follows from the computation of $E^{2,\bullet}_\infty$. By Lemma~\ref{lem:e1props}, the only subquotient of the detection filtration contributing to the column $E^{2,\bullet}_\infty$ is $\op{D}^0/\op{D}^1$. The description of $E^{2,\bullet}_2$ as a free $\mathbb{F}_\ell[\op{c}_1]$-module with Hilbert--Poincar{\'e} series 
\[
\frac{(1+T)q^3}{(1-T^2)}
\]
is then a consequence of Proposition~\ref{prop:rankonee1} and Corollary~\ref{cor:rankonee1}. The computations in Lemmas~\ref{lem:d2lower} and \ref{lem:d2upper} show that the $\op{d}_2$-differential removes an $\mathbb{F}_\ell[\op{c}_1]$-free direct summand, changing the numerator polynomial to $(1+T)(q^3-\op{rk}_\ell(\overline{E}))$. 

(3) This follows from the computation of $E^{1,\bullet}_\infty$. For degree reasons, $E^{1,\bullet}_\infty=E^{1,\bullet}_2$. Moreover, subquotients of the detection filtration contributing to $E^{1,\bullet}_2$ are $\op{D}^0/\op{D}^1$ and $\op{D}^1/\op{D}^2$. By Proposition~\ref{prop:rankonee1} and Corollary~\ref{cor:rankonee1}, there is no contribution from $\op{D}^0/\op{D}^1$. Point (2) of Corollary~\ref{cor:d1e2} provides the freeness statement and Hilbert--Poincar{\'e} series for the contribution from  $\op{D}^1/\op{D}^2$. 
\end{proofof}

At this point, we have proved that the kernel of the Quillen homomorphism has a filtration whose subquotients are free over smaller Chern-class rings. 

\begin{remark}
We also see that the kernel of the Quillen homomorphism is fairly big. The contribution from $E^{2,\bullet}_\infty$ is a free $\mathbb{F}_\ell[\op{c}_1]$-module of rank $2(q^3-\op{rk}_\ell(\overline{E}))$. Under our general assumptions, $q\geq 11$ but $\op{rk}_\ell(\overline{E})$ is at most $2$, implying that the rank of $E^{2,\bullet}_\infty$ is at least 2658. 

The contribution from $E^{1,\bullet}$ is a free $\mathbb{F}_\ell[\op{c}_1,\op{c}_2]$-module of rank $4(N_a+\op{rk}_\ell)$. The number $N_a$ is not so easy to determine, but by Corollary~\ref{cor:631} is roughly of order 
\[
\frac{1}{6}\left(\#\overline{E}(\mathbb{F}_q)^2- 3\#\overline{E}(\mathbb{F}_q)-6\right).
\]
While this may not be so big for elliptic curves with a small group of $\mathbb{F}_q$-points, it is a free graded module over $\mathbb{F}_\ell[\op{c}_1,\op{c}_2]$ and hence becomes arbitrarily large in high-enough cohomological degrees. It is also interesting to note that the number $2N_a-N_c$ is the Euler characteristic of the nontrivial connected component of the parabolic graph.
\end{remark}

It should also be noted that the classes in the kernel of the Quillen homomorphism are very explicit. The classes in $E^{2,\bullet}_\infty$ can be thought of as cuspidal cohomology classes; they come from the compactly supported cohomology of the quotient $\op{GL}_3(k[E])\backslash\mathfrak{B}(E/k,3)$ of the building. In fact, the classes in degree 2 are visible in rational cohomology, cf. \cite{harder}. The classes in $E^{1,\bullet}_\infty$ come from the stabilizer groups of edges of the parabolic graph, i.e., they are related to stable vector bundles of rank 2 on $\overline{E}$. As a consequence, it should be expected that the Quillen conjecture will fail for groups of rank 2 in any arithmetically non-trivial situation: for curves of higher genus because of moduli of stable vector bundles, and in number-theoretic situations whenever there are non-trivial cusp forms. 


\subsection{Structure of the image: detection filtration}

Now we describe the structure of the image of the Quillen homomorphism. It seems that the refined information contained in the detection filtration is not only very useful for the explicit calculations, but has some conceptual meaning. Anyway, the analysis of the subquotients of the detection filtration on the isotropy spectral sequence yields the argument to prove Theorem~\ref{thm:image}. 

\begin{proofof}{Theorem~\ref{thm:image}}
\label{pfimage}
By Lemma~\ref{lem:degenerate}, it suffices to identify the $0$-th cohomology groups of the detection subquotients of the $E_\infty$-page. By Lemma~\ref{lem:d2vanish}, the $\op{d}_2$-differential vanishes on the submodule $\op{D}^2$, and the $\op{D}^2$-part of the $E_2$-page is identified in Proposition~\ref{prop:rankthree}. This proves part (1) of the theorem.

Part (1) of Corollary~\ref{cor:d1e2} identifies the $\op{D}^1/\op{D}^2$-subquotient of the $E_2$-page. Lemmas~\ref{lem:d2lower} and \ref{lem:d2upper} show that the $\op{d}_2$-differential is non-trivial and removes a quotient module with Hilbert--Poincar{\'e} series 
\[
\frac{\op{rk}_\ell(\overline{E})T(1+T)}{(1-T^2)}.
\]
This proves part (2) of the theorem.

Proposition~\ref{prop:rankonee1} and Corollary~\ref{cor:rankonee1} describe the $\op{D}^0/\op{D}^1$-contribution to the $E_2$-page. The $\op{d}_2$-differential on this contribution is trivial by the lifting argument at the end of Section~\ref{sec:d2}. Combining these proves part (3) of the theorem.
\end{proofof}

It is somewhat hidden in the above statements but rather easy to see in the spectral sequence calculations that the image of the Quillen homomorphism is torsionless because it embeds in the direct sum of the cohomology rings of the stabilizers which contain an elementary abelian group of rank 3 -- but this is a finitely generated free graded module over the Chern-class ring $\mathbb{F}_\ell[\op{c}_1,\op{c}_2,\op{c}_3]$. In light of the many torsion classes exhibited by Theorem~\ref{thm:kernel}, the next best thing to hope for Quillen's conjecture would be that the image of the Quillen homomorphism, which is the torsion-free quotient of the cohomology ring, would be free over the Chern-class ring. However, Theorem~\ref{thm:free} shows that even this is not the case. The basic reason is that, while the subquotients of the detection filtration are free over the individual Chern-class rings, the filtration steps are not necessarily free. The situation is very similar to ideals in polynomial rings of dimension $\geq 2$, only with an additional grading. 

\begin{proofof}{Theorem~\ref{thm:free}}
Not much to do here. The failure of freeness was already established in Theorem~\ref{thm:e2nonfree}. It remains to note that the condition $\op{rk}_\ell=0$ implies that the $\op{d}_2$-differential is trivial, by the computations in Section~\ref{sec:d2}. 
\end{proofof}

Of course it would now be necessary to investigate possible variations of the conjecture. Such reformulations should be guided by Quillen's original motivation of trying to understand structural properties of cohomology rings of arithmetic groups which could help in computations. Seeing that the module on the torsion-free part may fail to be free is a significant setback because of the wildness of torsion-free modules over polynomial rings. At the moment, the freeness of the subquotients of the detection filtration seems the closest thing we could expect to be true which could also be helpful in computations. 

\begin{conjecture}[Filtered Quillen conjecture]
\label{conj:filtered}
Let $K$ be a global field and let $\ell$ be a prime different from the characteristic of $K$ such that $\zeta_\ell\in K$. Let $S$ be a finite set of places containing the places dividing $\infty\cdot\ell$. If $\ell\nmid n!$ then the detection filtration of Definition~\ref{def:e1filtration} on the image of the Quillen homomorphism
\[
\op{H}^\bullet(\op{GL}_n(\mathcal{O}_{K,S});\mathbb{F}_\ell)\to \lim_A\op{H}^\bullet(A;\mathbb{F}_\ell)
\]
has the property that each subquotient $\op{D}^{i-1}/\op{D}^i$ is a free module over the corresponding Chern-class ring $\mathbb{F}_\ell[\op{c}_1,\dots,\op{c}_i]$. 
\end{conjecture}

The requirements on $K$ and $S$ relative to coefficient prime $\ell$ are the usual ones assumed in Quillen's conjecture, cf.~Conjectures~\ref{conj:quillen} and \ref{conj:quillenff}. The additional requirement $\ell\nmid n!$ is a natural strengthening leading to significant simplication in many places. For instance, it simplifies the structure of possible finite stabilizer subgroups for the action on the symmetric space and it allows to use standard character theory for the study of group actions of subgroups of the Weyl group $\Sigma_n$ on the cohomology of elementary abelian groups. The theory of alternating polynomials discussed in Appendix~\ref{sec:eltalt} then has a natural generalization to higher ranks and it can be hoped that these additional structures allow an analysis of the spectral sequence analogous to the one presented here. 

One could also ask specific questions about the size of the filtration steps. One reasonable guess for the rank of $\op{D}^{n-1}$ would be $2^n\cdot\left(\#\op{Pic}^0(\mathcal{O}_{K,S})\right)^{n-1}$, generalizing the statement in part (1) of Theorem~\ref{thm:image}. This is the product of the number of alternating polynomials in the cohomology of elementary abelian $\ell$-groups of rank $n$ with the weighted number of completely split vector bundles giving rise to stabilizers containing maximal rank elementary abelian subgroups. 

At the moment, there seems to be too little evidence for attempting to generalize the other aspects of the cohomology computations in this paper. While the isotropy spectral sequence always allows to provide a filtration on the kernel of the Quillen homomorphism, it seems unlikely that the subquotients will be free in general. 

As mentioned earlier, Quillen's motivation for posing his conjecture was to find structural properties of cohomology rings of arithmetic groups which could be used for computations. The calculations of the present paper show that the image of the Quillen homomorphism contains only the simplest part of unstable group cohomology. In the K-theoretic rank conjecture philosophy, the part detected on elementary abelian $\ell$-subgroups corresponds to Milnor K-theory as the simplest part of algebraic K-theory. Therefore, the arithmetically more interesting part of group cohomology (such as the cuspidal cohomology or K-theory classes related to stable bundles) is in the Chern class torsion. This means, in particular, that even a proof of Conjecture~\ref{conj:filtered} above would not go as long a way towards computation of $\op{H}^\bullet(\op{GL}_n(\mathcal{O}_{K,S});\mathbb{F}_\ell)$ as hoped in the original formulation of \cite[Conjecture 14.7]{quillen:spectrum}.

\appendix

\section{Alternating elements in cohomology rings}
\label{sec:eltalt}

The identification of alternating polynomials in polynomial rings is fairly standard. For the computation of the $E_2$-page, we need a variant of this standard description for the slightly different case of free graded-commutative algebras.

For a field $k$, we denote the free graded-commutative $k$-algebra by 
$$
\mathfrak{F}(n):=k[X_1,\dots,X_n]\langle A_1,\dots,A_n\rangle.
$$
This algebra has a natural action of the symmetric group $\Sigma_n$, given by the diagonal permutation action on the sets of generators $\{X_1,\dots,X_n\}$ and $\{A_1,\dots,A_n\}$. In the following, we will be interested in the special cases $n=2,3$ and $k=\mathbb{F}_\ell$ where $\ell$ a prime different from $2$ and $3$. Requiring $\ell\neq 2$ is a natural thing to do because the cohomology ring structure for elementary abelian $\ell$-groups is different for $\ell=2$. On the other hand, our analysis of the special case $n=3$ makes significant use of the representation theory of $\Sigma_3$, in particular of complete decomposability and the character table -- which is why we exclude the prime $\ell=3$ as well. 


\begin{definition}
\label{def:alternating}
An element $f(X_1,\dots,X_n,A_1,\dots,A_n)\in\mathfrak{F}(n)$ is called (in slight abuse of notation) \emph{symmetric polynomial}, if it satisfies 
$$
f(\underline{X},\underline{A})=f(\sigma(\underline{X}),\sigma(\underline{A})).
$$

An element $f(X_1,\dots,X_n,A_1,\dots,A_n)\in\mathfrak{F}(n)$ is called  an \emph{alternating polynomial} if it satisfies 
$$
f(\underline{X},\underline{A})
=\op{sgn}(\sigma)f(\sigma(\underline{X}),\sigma(\underline{A})).
$$ 

The sets of symmetric and alternating polynomials in $\mathfrak{F}(n)$ will be denoted by $\mathfrak{S}(n)$ and $\mathfrak{A}(n)$, respectively. 
\end{definition}

\begin{remark}
Some aspects of the situation are closely parallel to the case of symmetric and alternating polynomials in commuting variables: 

The symmetric polynomials form a subalgebra $\mathfrak{S}(n)$. In terms of cohomology of finite groups, $\mathfrak{F}(n)$ is the $\mathbb{F}_\ell$-cohomology of an elementary abelian $\ell$-group of rank $n$. The subalgebra $\mathfrak{S}(n)$ is the $\mathbb{F}_\ell$-cohomology of the group $\op{GL}_n(\mathbb{F}_q)$ whenever $\ell\mid q-1$. This follows from Quillen's computation of the cohomology of $\op{GL}_n(\mathbb{F}_q)$, cf. Theorem~\ref{thm:glnfq} or better \cite{quillen:ktheory}.

The alternating polynomials do not form a subalgebra. However, $\mathfrak{A}(n)$ is a sub-$\mathfrak{S}(n)$-module of $\mathfrak{F}(n)$. This, in particular, also implies that $\mathfrak{A}(n)$ is a submodule of $\mathfrak{F}(n)$ for the Chern class ring $\mathbb{F}_\ell[\op{c}_1,\dots,\op{c}_n]$. 
\end{remark}

If $\ell\nmid n!$, then the degree $i$ part of $\mathfrak{F}(n)$, denoted by $\mathfrak{F}(n)_i$ is a direct sum of irreducible $\Sigma_n$-representations. The subsets $\mathfrak{S}(n)_i$ and $\mathfrak{A}(n)_i$ are the subrepresentations given by the direct sum of all trivial and sign representations, respectively. 

While this seems to be a natural construction, it is difficult to find it in the literature, cf. the (unfortunately unanswered) MathOverflow-question 221399. Fortunately, the present paper only requires the cases $n=2,3$, which can be dealt with in a fairly straightforward fashion. For a more conceptual explanation due to Wolfgang Soergel, cf. Remark~\ref{rem:soergel}. 

\begin{proposition}
\label{prop:eltalt2}
Let $\ell$ be an odd prime. 
The $\mathfrak{S}(2)$-submodule $\mathfrak{A}(2)$ of $\mathfrak{F}(2)$ is generated by the alternating polynomials $A_1-A_2$ and $X_1-X_2$. As a module over the Chern-class ring $\mathbb{F}_\ell[\op{c}_1,\op{c}_2]$, it is freely generated by the alternating polynomials
\[
A_1-A_2,\, A_1A_2,\, X_1-X_2,\, (X_1-X_2)(A_1+A_2).
\]
In particular, it is a free module of rank $4$ with Hilbert--Poincar{\'e} series 
\[
\frac{T+2T^2+T^3}{(1-T^2)(1-T^4)}.
\]
\end{proposition}

\begin{proof}
In characteristic $\neq 2$, the algebra $\mathfrak{F}(2)$ decomposes completely as direct sum of trivial and sign representations. 

The polynomial ring $\mathbb{F}_\ell[X_1,X_2]$ decomposes as follows. The subring of symmetric polynomials is generated by $X_1+X_2$ and $X_1X_2$ hence its Hilbert--Poincar{\'e} series is
\[
\frac{1}{(1-T^2)(1-T^4)}.
\]
The submodule of alternating polynomials is free of rank 1 over the symmetric polynomials, generated by the discriminant $(X_1-X_2)$ and hence has the Hilbert--Poincar{\'e} series of the symmetric polynomials with an additional factor $T^2$. We can check that the sum of the two Hilbert--Poincar{\'e} series is the Hilbert--Poincar{\'e} series of the polynomial ring $\mathbb{F}_\ell[X_1,X_2]$. 

The exterior algebra $\mathbb{F}_\ell\langle A_1,A_2\rangle$ decomposes under the $\Sigma_2$-action as an invariant subring $\mathbb{F}_\ell\langle A_1+A_2\rangle$, and the alternating part which is a free module over the invariant part generated by $A_1-A_2$. 

In the tensor product $\mathfrak{F}(2)$, we now find that the alternating part is the direct sum of the alternating part of the polynomial algebra with the symmetric part of the exterior algebra, and the symmetric part of the polynomial algebra with the alternating part of the exterior algebra. Freeness over the Chern-class ring follows since $1,A_1-A_2, A_1+A_2$ and $A_1A_2$ are obviously $\mathbb{F}_\ell[\op{c}_1,\op{c}_2]$-linearly independent. This establishes the claims concerning the generators and the Hilbert--Poincar{\'e} series. 
\end{proof}

\begin{remark}
It is interesting to remark a difference to the classical case of alternating and symmetric polynomials in commuting variables:  $\mathfrak{A}(2)$ is not a free $\mathfrak{S}(2)$-module. For example, the element $(A_1-A_2)$ is annihilated by 
\begin{eqnarray*}
\op{c}_1\op{e}_1-2\op{e}_2&=&(X_1+X_2)(A_1+A_2)-2(X_2A_1+X_1A_2)\\
&=&X_1A_1-X_1A_2-X_2A_1+X_2A_2.
\end{eqnarray*}
\end{remark}

\begin{proposition}
\label{prop:eltalt3}
Let $\ell$ be a prime different from $2$ and $3$. Then the $\mathfrak{S}(3)$-submodule $\mathfrak{A}(3)$ of $\mathfrak{F}(3)$ is generated by the alternating polynomials 
\begin{itemize}
\item $(A_1-A_2)(A_2-A_3)$,
\item $-A_1(X_2-X_3)+A_2(X_1-X_3)-A_3(X_1-X_2)$, 
\item $A_1X_1(X_2-X_3)-A_2X_2(X_1-X_3)+A_3X_3(X_1-X_2)$, and
\item $(X_1-X_2)(X_1-X_3)(X_2-X_3)$.
\end{itemize}
As a module over the Chern-class ring $\mathbb{F}_\ell[\op{c}_1,\op{c}_2,\op{c}_3]$, it is by the above alternating polynomials together with the following (which essentially arise from the previous ones by multiplication with $\op{e}_1$)
\begin{itemize}
\item $A_1A_2A_3$
\item $-X_1A_2A_3+X_2A_1A_3-X_3A_1A_2$
\item $-X_1X_2A_1A_2+X_1X_3A_1A_3-X_2X_3A_2A_3$, and
\item $(A_1+A_2+A_3)(X_1-X_2)(X_1-X_3)(X_2-X_3)$.
\end{itemize}
The given eight generators are $\mathbb{F}_\ell[X_1,X_2,X_3]$-linearly independent. In particular, $\mathfrak{A}(n)$ is a free graded module  of rank $8$ over the Chern class ring $\mathbb{F}_\ell[\op{c}_1,\op{c}_2,\op{c}_3]$ with Hilbert--Poincar{\'e} series 
\[
\frac{T^2+2T^3+T^4+T^5+2T^6+T^7}{(1-T^2)(1-T^4)(1-T^6)}
\]
\end{proposition}

\begin{proof}
Consider the polynomial ring $\mathbb{F}_\ell[X_1,X_2,X_3]$, graded by $\deg X_i=2$, with the standard permutation action on $\{X_1,X_2,X_3\}$. The decomposition as $\Sigma_3$-representations is known: the trivial representations make up the ring of symmetric polynomials $\mathbb{F}_\ell[\op{c}_1,\op{c}_2,\op{c}_3]$ and the alternating polynomials make up a free $\mathbb{F}_\ell[\op{c}_1,\op{c}_2,\op{c}_3]$-module of rank one generated by the Vandermonde polynomial $(X_1-X_2)(X_1-X_3)(X_2-X_3)$ in degree $6$. Finally, the  remainder, made up of the standard representations, is generated by one standard representation in degree $2$ and one standard representation in degree $4$. This can be seen from the equality of Hilbert--Poincar{\'e} series
\[
\frac{1}{(1-T^2)^3}=\frac{1+2(T^2+T^4)+T^6}{(1-T^2)(1-T^4)(1-T^6)}
\]

Now consider the exterior algebra $\mathbb{F}_\ell\langle A_1,A_2,A_3\rangle$, graded by $\deg A_i=1$, with the standard permutation action on $\{A_1,A_2,A_3\}$. This decomposes as follows:
\begin{itemize}
\item a trivial representation in degree $0$,
\item a trivial and a standard representation in degree $1$,
\item a sign and a standard representation in degree $2$, and
\item  a sign representation in degree $3$. 
\end{itemize}
More precisely, the trivial representation is given by the subalgebra $\mathbb{F}_\ell\langle \op{e}_1\rangle$, the sign representation is the $\mathbb{F}_\ell\langle\op{e}_1\rangle$-submodule generated by $A_1A_2-A_1A_3+A_2A_3$, and also the standard representations form an $\mathbb{F}_\ell\langle \op{e}_1\rangle$-submodule. 

Now we consider the tensor product of the polynomial and exterior algebra. The character table for $\Sigma_3$ tells us completely how tensor products of representations decompose. The tensor product $\rho_1\otimes\rho_2$ of two irreducible representations contains a sign representation in the following cases: $\op{id}\otimes\sigma$, $\sigma\otimes\op{id}$ and $\op{s}\otimes \op{s}\cong\op{s}\oplus\op{id}\oplus\sigma$, where $\op{id}$ denotes the trivial representation, $\sigma$ the sign representation and $\op{s}$ the standard representation. 

Now we combine all this information, based on the decomposition of the exterior algebra to obtain the $\mathfrak{S}(3)$-generators of the alternating part. The alternating part of the exterior algebra tensored with the symmetric polynomials provides the first generator. The next two generators arise from the tensor products of the two standard representations (degree 2 and degree 4) in the polynomial ring with the standard representation (degree 1) of the exterior algebra. The last generator arises from tensoring the invariant part of the exterior algebra with the alternating polynomials. 
The $\mathbb{F}_\ell[\op{c}_1,\op{c}_2,\op{c}_3]$-generators are obtained from these by multiplication with the exterior algebra in $\mathfrak{S}(3)$, but only multiplication with $\op{e}_1$ provides new generators, as listed.

It remains to show that the eight generators given are in fact $\mathbb{F}_\ell[X_1,X_2,X_3]$-linearly independent. Because the exterior elements are linearly independent over the polynomial ring, we can argue by degree in the exterior elements. Degree $0$ and $3$ are clear because these only contain a single generator. Consider degree 1, degree 2 is proved similarly. We rewrite these generators as
\[
P=(X_3-X_2)(A_1-A_2)+(X_1-X_2)(A_2-A_3)\textrm{ and }
\]
\[
Q=(X_1X_2-X_1X_3)(A_1-A_2)+(X_2X_3-X_1X_3)(A_2-A_3),
\]
decomposing them as polynomial multiples of the two independent elements $(A_1-A_2)$ and $(A_2-A_3)$. There is an additional generator
\[
(A_1+A_2+A_3)(X_1-X_2)(X_1-X_3)(X_2-X_3),
\]
but this has to be independent of the other two because the invariant and alternating part of the exterior degree 1 are linearly independent over the polynomial ring. It remains to show that there can be no $\mathbb{F}_\ell[X_1,X_2,X_3]$-linear dependence $fP+gQ=0$ between the generators $P$ and $Q$. However, such a dependence would imply $f-gX_1=f-gX_3=0$ which is impossible. This proves the claim on linear independence, proving also linear independence over the Chern-class ring $\mathbb{F}_\ell[\op{c}_1,\op{c}_2,\op{c}_3]$. We have thus found an $\mathbb{F}_\ell[\op{c}_1,\op{c}_2,\op{c}_3]$-free submodule of $\mathfrak{A}(3)$ with the right Hilbert--Poincar{\'e} series, proving the claim. 
\end{proof}

\begin{remark}
\label{rem:soergel}
I am very grateful to Wolfgang Soergel for kindly explaining to me the right generalization of the above computations (for $\ell\nmid n!$): the polynomial ring $k[X_1,\dots,X_n]$ is a module of rank $\dim\op{H}^\bullet(\op{GL}_n/B;\mathbb{F}_\ell)$ over the subalgebra of symmetric polynomials, such that the quotient can in fact be identified with $k[\Sigma_n]$. Kostka--Foulkes polynomials determine the cohomological degrees of the irreducible representations in $k[\Sigma_n]$. Tensoring the polynomial ring with the exterior algebra, each irreducible $\Sigma_n$-representation in the exterior algebra contributes a sign representation with multiplicity the dimension of the irreducible factor. This implies that the sub-$\mathbb{F}_\ell[\op{c}_1,\dots,\op{c}_n]$-module $\mathfrak{A}(n)$ of $\mathfrak{F}(n)$ is free of rank $2^n$. A more detailed analysis could also allow to explicitly identify generators, their degrees and the resulting Hilbert--Poincar{\'e} series. The consequences for the theory of detection filtrations in higher ranks will be developed elsewhere, but the conceptual explanation of the alternating polynomials suggests that a proof of Conjecture~\ref{conj:filtered} may be within reach. 
\end{remark}

\begin{remark}
In retrospect, using the alternating polynomials in cohomology of elementary abelian groups is similar to the appearance of the sign representations in the analysis of the spectral sequence in \cite[Section 13]{dupont:book} (though loc.cit. is mostly interested in rational coefficents). This is further evidence that the theory of alternating polynomials could also be useful for higher-rank computations and even for computations of cohomology of linear groups of curves over infinite fields.
\end{remark}

\section{A sheaf view on torsion subcomplex reduction}
\label{sec:reduction}

In this appendix we discuss a technical result which allows to replace the $E_1$-page of the isotropy spectral sequence for the $\op{GL}_3(k[E])$-action on the Bruhat--Tits building $\mathfrak{B}(E/k,3)$ by a subcomplex which is  finitely generated over $\mathbb{F}_\ell[\op{c}_1,\op{c}_2,\op{c}_3]$. this is similar to the torsion subcomplex reduction of \cite{rahm}. Recall that the action of $\op{GL}_3(k[E])$ on the building is not cocompact and the quotient has infinitely many cells, leading to unpleasantries involving infinite direct products in the $E_1$-page of the isotropy spectral sequence. 

It has been known since the works of Serre, Harder, Stuhler and others that for a ring $\mathcal{O}_{K,S}$ of $S$-integers in a global field $K$, there is a $\op{GL}_n(\mathcal{O}_{K,S})$-stable  subcomplex $\mathfrak{X}$ in the associated Bruhat--Tits building $\mathfrak{B}(K,S,n)$ such that the quotient $\op{GL}_n(\mathcal{O}_{K,S})\backslash\mathfrak{X}(K,S,n)$ has only finitely many cells and the inclusion $\mathfrak{X}\hookrightarrow\mathfrak{B}(K,S,n)$ induces isomorphisms in $\op{GL}_n(\mathcal{O}_{K,S})$-equivariant cohomology with $\mathbb{F}_\ell$-coefficients where $\ell$ is different from the characteristic of $K$. This means that the isotropy spectral sequence is essentially controlled by only finitely many cell stabilizers, allowing to get rid of the infinite direct products. 

We provide some details on this reduction of the isotropy spectral sequence. This is well-known, but to the best of my knowledge not well-documented in the literature. 

\begin{definition}
\label{def:hg}
Let $X$ be a CW-complex of finite $\ell$-cohomological dimension and let $G$ be a group acting continuously on $X$ with finite stabilizers. For each $t\geq 0$ consider the sheaf  $\mathcal{H}_G^t$ on the orbit space $X/G$ which is associated to the presheaf given by 
\[
U\mapsto \op{H}^t_{G}(q^{-1}(U);\mathbb{F}_\ell)
\]
where $q:X\to X/G$ is the canonical quotient map.
\end{definition}

In \cite[Section 3]{quillen:spectrum}, Quillen proved that there are canonical isomorphisms
\[
E^{s,t}_2\cong \op{H}^s(X/G;\mathcal{H}_G^t)
\]
identifying the $E_2$-page of the isotropy spectral sequence associated to the $G$-action on $X$ with the cohomology of the sheaves $\mathcal{H}_G^\bullet$.

\begin{definition}
Let $X$ be a CW-complex of finite $\ell$-cohomological dimension, and let $G$ be a group acting continuously on $X$ with finite stabilizers. A $\mathcal{H}_G^\ast$-reduction of $X/G$ is a subcomplex $Y\subset X/G$ together with a cellular retraction $r:X/G\to Y$ such for each cell $\sigma$ of $Y$ there is a filtration 
\[
F^{(0)}\subset F^{(1)}\subset \cdots\subset F^{(k)}=r^{-1}(\sigma)
\]
such that 
\begin{enumerate}
\item each $F^{(i)}$ and each component of $F^{(i)}\setminus F^{(i-1)}$ is contractible, and 
\item the sheaves $\mathcal{H}_G^t$ are locally constant on each $F^{(i)}\setminus F^{(i-1)}$.
\end{enumerate}
\end{definition}

With this definition, we can now provide a generalization and cohomology version of \cite[Proposition A.2.7]{knudson:book}.

\begin{proposition}
\label{prop:reduction}
Let $X$ be a CW-complex of finite $\ell$-cohomological dimension and let $G$ be a group acting continuously on $X$ with finite stabilizers. Let $Y\subset X$ be a $\mathcal{H}_G^\ast$-reduction of $X/G$. Then the retraction $r:X/G\to Y$ induces an isomorphism 
\[
\op{H}^s(Y;\mathcal{H}^t_G)\to \op{H}^s(X/G;\mathcal{H}^t_G).
\]
In particular, if we consider the subcomplex $\tilde{E}^{s,t}_1\subset E^{s,t}_1$ of the isotropy spectral sequence for $G\looparrowright X$ given by those cohomology of stabilizers of cells contained in $Y$, then the subcomplex inclusion $\tilde{E}^{s,t}_1\subset E^{s,t}_1$ is a quasi-isomorphism. Under the identification $E^{s,t}_2\cong \op{H}^s(X/G;\mathcal{H}_G^t)$, the above isomorphisms of sheaf cohomology are equal to the induced isomorphisms on the $E_2$-page.
\end{proposition}

\begin{proof}
Recall that the higher direct image sheaves are the sheafifications of the presheaf $U\mapsto \op{H}^t(r^{-1}(U);\mathcal{H}^i_G)$. Associated to the morphism $:X/G\to Y$ and the coefficient sheaf $\mathcal{H}^i_G$ we have a Leray spectral sequence:
\[
E^{s,t}_2=\op{H}^s(Y;\op{R}^tr_\ast\mathcal{H}^i_G)\Rightarrow \op{H}^{s+t}(X/G;\mathcal{H}^i_G).
\] 
To prove the claims, it suffices to show 
\[
\op{R}^tr_\ast\mathcal{H}^i_G\cong \left\{\begin{array}{ll}
\mathcal{H}^i_G & t=0\\
0&\textrm{otherwise}\end{array}\right.
\]
The stalk of $\op{R}^tr_\ast\mathcal{H}^i_G$ at a point $y\in Y$ is $\op{H}^t(r^{-1}(y);\mathcal{H}^i_G)$. Now for any point $y\in Y$ consider the preimage $r^{-1}(y)$. By assumption, there is a filtration of $r^{-1}(y)$ such that the conditions of (a cohomology version of) \cite[Proposition A.2.7]{knudson:book} are satisfied. As a consequence, the stalk of $\op{R}^tr_\ast\mathcal{H}^i_G$ at a point $y$ is 
\[
\op{H}^t(\{y\};\mathcal{H}^i_G)\cong \left\{\begin{array}{ll}
(\mathcal{H}^i_G)_y& t=0\\
0&\textrm{otherwise}\end{array}\right.
\]
Restriction of sections from $r^{-1}(U)$ to $U$ defines a morphism of sheaves $\op{R}^tr_\ast\mathcal{H}^i_G\to \mathcal{H}^i_G$ which can be seen as an instance of the base-change map $\op{id}^\ast r_\ast\mathcal{H}\to \op{id}_\ast i^\ast\mathcal{H}$ from the six-functor formalism. The computations of stalks above shows that the inclusion induces isomorphisms on stalks. 
\end{proof}

\section{Generalization of a theorem of Henn}
\label{sec:henn}

In this section, we explain the generalization of Henn's result \cite[Theorem 0.6]{henn} which shows that the failure of the Quillen conjecture for rank $n_0$ implies the failure in all ranks $n\geq n_0$. 

\subsection{Finiteness properties for arithmetic groups, function field case} 

We first have to check that the arguments of \cite{henn} can be applied to arithmetic groups of function field type. Note that the formulation of Theorem 0.2 in \cite{henn} does not apply directly to groups of the form $\op{GL}_n(k[C])$ where $k=\mathbb{F}_q$ is a finite field and $C$ is a smooth affine curve over $k$. The reason is that the action of $\op{GL}_n(k[C])$ on the associated Bruhat--Tits building fails to be cocompact. In the number-field situations, this is also true, but it is possible to use the Borel--Serre compactification or suitable deformation retracts of the symmetric space to obtain acyclic spaces with finitely many cells. This is no longer true in the function field situation. Therefore, we will shortly explain how to modify the arguments in \cite{henn} so that they apply to the function-field situation. 

\begin{proposition}
\label{prop:a1}
Let $k=\mathbb{F}_q$ be a finite field, let $C$ be a smooth affine curve over $k$ and let $\ell$ be a prime different from the characteristic of $k$. 
\begin{enumerate}
\item Then $\op{H}^\bullet(G;\mathbb{F}_\ell)$ is an unstable noetherian algebra and there is a canonical equivalence of categories 
\[
\mathcal{A}(G)\to\mathcal{R}(\op{H}^\bullet(G;\mathbb{F}_\ell)): A\mapsto (A,\op{res}_G^A)
\]
from the Quillen category of elementary abelian $\ell$-subgroups of $G$ to the spectral category of the cohomology ring $\op{H}^\bullet(G;\mathbb{F}_\ell)$. 
\item There are isomorphisms
\[
T_A(\op{H}^\bullet(G);\op{res}_G^A)\cong\op{H}^\bullet(\op{C}_G(A))
\]
which are natural in $A\in\mathcal{A}(G)$.
\end{enumerate}
\end{proposition}

\begin{proof}
Note that Theorem A.1 and Corollary A.2 of \cite{henn} apply since the group $\op{GL}_n(k[C])$ acts on the associated Bruhat--Tits building with finite orbit type and finite isotropy. What is needed to complete the proof of Theorem 0.2(b) on p. 214 of \cite{henn} is an analogue of Theorem A.3. For the action of $G=\op{GL}_n(k[C])$ on the associated Bruhat--Tits building $\mathfrak{B}$, one can find a $G$-subcomplex $\mathfrak{X}$ of $\mathfrak{B}$ with finitely many $G$-cells such that the inclusion $\mathfrak{X}\hookrightarrow \mathfrak{B}$ induces an isomorphism 
\[
\op{H}^\bullet_G(\mathfrak{B};\mathbb{F}_\ell)\to\op{H}^\bullet_G(\mathfrak{X};\mathbb{F}_\ell),
\]
cf. the discussion in Section~\ref{sec:reduction}. 
Applying \cite[Theorem A.3]{henn} to the $G$-CW-complex $\mathfrak{X}$ implies the result.
\end{proof}

In particular, the general theory developed in \cite{henn} also applies to $S$-arithmetic groups of function field type.


\subsection{Henn's argument} 
We now provide an analogue of Henn's argument, cf. \cite[Section 4]{henn}, for the function field situation. The spectral category 
\[
\mathcal{R}_n=\mathcal{R}(\op{H}^\bullet(\op{GL}_n(k[E]);\mathbb{F}_\ell))
\]
of the Steenrod module given by cohomology of $\op{GL}_n(k[E])$ can be identified with the Quillen category of elementary abelian $\ell$-subgroups of $\op{GL}_n(k[E])$. In the function field situation at hand, every elementary abelian $\ell$-subgroup appears as automorphism group of a rank $n$ vector bundle $\mathcal{V}$ on $\overline{E}$. 

To state the main result, let $k=\mathbb{F}_q$ be a finite field, let $\overline{E}$ be an elliptic curve over $k$ with a $k$-rational point $\op{O}$ and set $E=\overline{E}\setminus\{\op{O}\}$. Let $\ell\geq 5$ be a prime different from the characteristic of $k$, and denote by 
\[
Q_n=\op{H}^\bullet(\op{GL}_n(k[E]);\mathbb{F}_\ell)\to \lim_A\op{H}^\bullet(A;\mathbb{F}_\ell)
\]
the kernel of the Quillen homomorphism for $\op{GL}_n(k[E])$. The following is the function field analogue of \cite[Theorem 4.2]{henn}:

\begin{theorem}
\label{thm:henn}
For $n\geq 3$, the Hilbert--Poincar{\'e} series $\op{HP}(Q_n,T)$ of the kernel of the Quillen homomorphism has a pole of order $n-1$ at $T=1$. 
\end{theorem}

\begin{proof}
The fact that $n-1$ is an upper bound for the pole order is almost clear. 
The pole order is certainly smaller than $n$ because this is the maximal rank of an elementary abelian $\ell$-subgroup. If the pole order was equal to $n$, then there would be classes from essential ideals of rank $n$ elementary abelian $\ell$-subgroups. From the local structure of the action on the building, this is impossible. 

Now we need to show that $n-1$ is also a lower bound for the pole order. 
As in the proof of \cite[Theorem 4.2]{henn}, it suffices to find an elementary abelian $\ell$-subgroup $\phi:A\to\op{GL}_n(k[E])$ such that the pole order for $T_A(Q_n;\phi^\ast)$ is $n-1$. An elementary abelian subgroup with a character $n_0\chi+\sum \chi_i$ can be obtained as subgroup of a vector bundle of the form $\mathcal{V}(n_0,1)\oplus\mathcal{O}^{n-n_0}$, i.e., by stabilization from a stable vector bundle in the rank where the first examples appear. Using \cite[Theorem A.1]{henn} (which applies to the action of $\op{GL}_n(k[E])$ on the associated Bruhat--Tits building), we get an exact sequence 
\[
0\to T_A(Q_n;\phi^\ast)\to \op{H}^\bullet(\op{GL}_n(k[E]);\mathbb{F}_\ell)\otimes \op{H}^\bullet((k^\times)^{n-n_0};\mathbb{F}_\ell)\to\prod_A\op{H}^\bullet(A;\mathbb{F}_\ell)
\]
where the product is taken over all elementary abelian subgroups conjugate to $\phi$. As in \cite{henn}, the kernel of this map is $Q_{n_0}\otimes\op{H}^\bullet((k^\times)^{n-n_0};\mathbb{F}_\ell)$. The results of the present paper establish that the Hilbert--Poincar{\'e} series $\op{HP}(Q_{n_0},T)$ has a pole of order $2$ at $T=1$  for $n_0=3$. Then 
\[
\op{HP}(Q_{n_0}\otimes\op{H}^{\bullet}((k^\times)^{n-n_0};\mathbb{F}_\ell),T)
\]
will have a pole of order $n-1$ for any $n\geq n_0$. 
\end{proof}

\begin{remark}
Note that the pole order is higher than the one given in \cite{henn} for $\op{H}^\bullet(\op{GL}_n(\mathbb{Z}[1/2]);\mathbb{F}_2)$. In the situations of $\op{GL}_n(\mathbb{Z}[1/2])$ or $\op{GL}_3(\mathbb{Z}[\zeta_3,1/3])$, the simplicity of the Quillen category seems to imply that there are no ``global'' counterexamples, i.e., any element in the kernel of the Quillen homomorphism is already in the kernel of the Quillen homomorphism for the centralizer of some elementary abelian $\ell$-subgroup. ``Local'' counterexamples must then be supported on elementary abelian $\ell$-groups of rank $\leq n-n_0+1$, giving the estimate in Henn's paper. In the function field situation with non-trivial class groups, the results of the present paper show that ``global'' counterexamples to the Quillen conjecture can indeed appear. Hence, even though the possible non-triviality of the kernel of the Quillen homomorphism was already well-known, there are new types of counterexamples appearing in the elliptic curve computations. 

The reason for the propagation of non-triviality of the kernel to higher ranks stems from the fact that the counterexamples from $\op{GL}_{n_0}$ appear in the kernels of  centralizers of suitable elementary abelian $\ell$-subgroups. Hence, while counterexamples to the Quillen conjecture in any given rank do not need to be local, they persist as local counterexamples in higher ranks.
\end{remark}

\begin{proofof}{Corollary~\ref{cor:henn}}
\label{proof:henn}
This follows directly from Theorem~\ref{thm:henn}. For $n\geq 3$, we have a pole order $\geq 2$ for the kernel of the Quillen homomorphism. Therefore, the kernel has to be non-trivial, and Quillen's conjecture fails.
\end{proofof}

\begin{remark}
The possibility of generalizations of Henn's arguments to other $S$-arithmetic groups was already remarked in \cite{henn} as well as in \cite{knudson:book}. However, as opposed to the case $\op{GL}_n(\mathbb{Z}[1/2])$ discussed in \cite{henn}, in more general cases one has to consider the kernel of the Quillen homomorphism (not just the kernel of restriction to the diagonal). Moreover, the size of the kernel may differ from Henn's bound $n-n_0+1$ due to the possibility of ``global'' counterexamples to the Quillen conjecture. 
\end{remark}

\subsection{Generic non-triviality of the kernel}

We finally sketch arguments towards a general failure of the Quillen conjecture due to non-triviality of the kernel of the Quillen homomorphism. In the function field cases, a sufficient amount of hyperbolicity for the open curve should imply the failure of the Quillen conjecture for all $\op{GL}_n$, $n\geq 2$:

\begin{conjecture}
Let $k=\mathbb{F}_q$ be a finite field, let $\overline{C}$ be a smooth projective curve over $k$, let $D$ be an effective divisor on $\overline{C}$ and set $C=\overline{C}\setminus D$. Assume that $\deg D\geq 4-2g(\overline{C})$. Then the following hold:
\begin{enumerate}
\item $\dim \op{H}^{S}(\op{GL}_2(k[C]);\mathbb{Q})>0$, where $S$ denotes the number of points in the support of $D$.
\item The Quillen conjecture for $\op{GL}_n(k[C])$ fails for all $n\geq 2$.
\item The Hilbert--Poincar{\'e} series of the kernel of the Quillen homomorphism for $\op{GL}_n(k[C])$ has a pole of order $n-1$ at $T=1$.
\end{enumerate}
\end{conjecture} 

Part (1) of the conjecture is true in the following cases:
\begin{itemize}
\item for $\overline{C}=\mathbb{P}^1$ and $S=1$ by the computations in \cite{koehl:muehlherr:struyve},
\item for $\overline{C}=\mathbb{P}^1$, a divisor $D$ consisting only of $k$-rational points and a sufficiently big field $k$ by a cell count in the quotient of a product of Bruhat--Tits trees, similar to computations in \cite{sl2parabolic},
\item for curves $\overline{C}$ of genus $\geq 2$ and $\deg D=1$ because the existence of stable bundles on $\overline{C}$ implies the non-triviality of the fundamental group of the quotient of the Bruhat--Tits tree modulo $\op{GL}_2(k[C])$.
\end{itemize}
The preferred way of proving part (1) of the above conjecture in full generality would be to establish a Gau\ss{}--Bonnet formula computing $\dim \op{H}^{S}(\op{GL}_2(k[C]);\mathbb{Q})$ as suggested in \cite[p. 136]{harder}. Such a formula seems not currently known, cf. the unanswered MO-question 167408. 

In the case $S=1$, part (1) of the conjecture implies the failure of the Quillen conjecture for $\op{GL}_2(k[C])$. In this case, the quotient of the Bruhat--Tits tree has a non-trivial fundamental group. The simple evaluation of the isotropy spectral sequence implies that the non-trivial loops lead to classes in $E^{1,\bullet}_2=E^{1,\bullet}_\infty$ which are annihilated by $\op{c}_2$. Morevoer, in this case we also see that the classes in $E^{1,\bullet}_\infty$ are non-torsion for $\op{c}_1$, implying that the pole order for the Hilbert--Poincar{\'e} series is $1$ as claimed in part (3). 
The failure of the Quillen conjecture for $\op{GL}_2(k[C])$ and arbitrary $S$ is more complicated, since the dimension of the product of Bruhat--Tits trees equals $S$ and for $S>1$ higher differentials could kill the classes coming from the non-triviality of the quotient implied by (1). 

In any case, the failure of the Quillen conjecture for $\op{GL}_2(k[C])$ implies part (2) of the conjecture, by arguments similar to the ones used to prove Theorem~\ref{thm:henn}. This argument also establishes part (3) of the above conjecture whenever we have a non-trivial kernel for $\op{GL}_2(k[C])$. 

In light of the above arguments, the cases $\mathbb{Z}[1/2]$ and $\mathbb{Z}[\zeta_3,1/3]$ seem very special, analogues of the affine curves $\mathbb{A}^1$ or $\mathbb{P}^1\setminus\{0,1\}$ in the function field situation.

\end{document}